\documentclass[onefignum,onetabnum]{siamonline190516}
\usepackage{upgreek}
\usepackage{amsfonts, amsmath, amssymb, mathtools}
\usepackage{graphicx}
\usepackage{epstopdf}
\usepackage{mathabx}
\usepackage{subcaption}
\usepackage{algorithmic}
\usepackage{enumerate}
\usepackage{bm}
\ifpdf
  \DeclareGraphicsExtensions{.eps,.pdf,.png,.jpg}
\else
  \DeclareGraphicsExtensions{.eps}
\fi

\usepackage[normalem]{ulem}

\usepackage{enumitem}
\setlist[enumerate]{leftmargin=.5in}
\setlist[itemize]{leftmargin=.5in}

\newsiamthm{claim}{Claim}
\newsiamremark{conjecture}{Conjecture}
\newsiamremark{remark}{Remark}
\newsiamremark{example}{Example}
\newsiamremark{hypothesis}{Hypothesis}
\newsiamremark{problem}{Problem}
\newsiamremark{assumption}{Assumption}
\usepackage{mathrsfs}
\usepackage{etoolbox}
\AtEndEnvironment{remark}{\null\hfill$\Diamond$}

\newcommand{\mcl}{\mathcal}

\newcommand{\mbf}{\mathbf}
\newcommand{\mbb}{\mathbb}

\newcommand{\dist}{{\rm dist}}

\newcommand{\dd}{{\rm d}}

\newcommand{\mK}{\mcl{K}}

\newcommand{\mmT}{\mathcal{T}}
\newcommand{\mT}{\mathscr{T}}

\newcommand{\TT}{T^{\dagger}}

\newcommand{\hT}{\widehat{T}}
\newcommand{\hnu}{\widehat{\nu}}

\newcommand{\W}{{\rm W}}
\newcommand{\KL}{{\rm KL}}
\newcommand{\MMD}{{\rm MMD}}

\newcommand{\PP}{\mbb{P}}
\newcommand{\R}{\mbb{R}}

\definecolor{darkred}{rgb}{.7,0,0}

\definecolor{darkgreen}{rgb}{.15,.55,0}

\definecolor{darkblue}{rgb}{0,0,0.7}

\newcommand\numberthis{\addtocounter{equation}{1}\tag{\theequation}}

\newcommand{\rev}[1]{\textcolor{red}{#1}}

\usepackage{amsopn}

\DeclareMathOperator*{\argmin}{arg\,min}

\DeclareMathOperator*{\essinf}{ess\,inf}

  \AtEndEnvironment{example}{\null\hfill$\Diamond$}
  \AtEndEnvironment{runningexample}{\null\hfill$\Diamond$}

\renewcommand{\rev}[1]{#1}

\reversemarginpar
\usepackage[colorinlistoftodos,prependcaption,textsize=tiny]{todonotes}

\headers{Error Analysis of Measure Transport}{R.\ Baptista, B.\ Hosseini, N.\ B.\ Kovachki, Y.\ Marzouk, A.\ Sagiv}

\title{An Approximation Theory Framework for Measure-Transport\\ Sampling Algorithms}

\author{Ricardo Baptista\thanks{Computing and Mathematical Sciences, Caltech, Pasadena, CA (\email{rsb@caltech.edu})}
\and  
Bamdad Hosseini\thanks{Applied Mathematics, University of Washington, Seattle, WA (\email{bamdadh@uw.edu})}
\and 
Nikola B.\thinspace Kovachki\thanks{NVIDIA Corporation, Santa Clara, CA (\email{nkovachki@nvidia.com})}
\and  Youssef Marzouk\thanks{Center for Computational Science and Engineering, MIT, Cambridge, MA (\email{ymarz@mit.edu})}
\and Amir Sagiv\thanks{Technion - Israel Institute of Technology, Haifa, Israel (\email{amirsagiv@technion.ac.il})}
}

\ifpdf
\hypersetup{
  pdftitle={ErrorAnalysisOfMeasureTransport},
  pdfauthor={R. Baptista, B. Hosseini, N. B. Kovachki, Y. Marzouk, Amir Sagiv}
}
\fi

\begin{document}

\maketitle

% REQUIRED
\begin{abstract}
This article presents a general approximation-theoretic framework to analyze measure transport algorithms for probabilistic modeling. A primary motivating application for such algorithms is sampling---a central task in statistical inference and generative modeling.
We provide a priori error estimates in the continuum limit, i.e., when the measures (or their densities) are given, but when the transport map is discretized or approximated using a finite-dimensional function space.
Our analysis relies on the regularity theory of transport maps and on classical approximation theory for high-dimensional functions. A third element of our analysis, which is of independent interest, is the development of new stability estimates
that relate the distance between two maps to the  distance~(or divergence) between the pushforward measures they define. We present a series of applications of our framework, where quantitative convergence rates are obtained for practical problems using Wasserstein metrics, maximum mean discrepancy, and Kullback--Leibler divergence. Specialized rates for approximations of the popular triangular Kn{\"o}the-Rosenblatt maps are obtained, followed by numerical experiments that demonstrate and extend our theory. 
\end{abstract}

% REQUIRED

% REQUIRED
% \begin{AMS}
% \end{AMS}

\section{Introduction}\label{sec:introduction}
This article presents a general framework for analyzing the approximation error of measure-transport approaches to probabilistic modeling. The approximation of high-dimensional probability measures is a fundamental problem in statistics, data science, and uncertainty quantification. Broadly speaking, probability measures can be characterized via sampling (generative modeling) or direct density estimation.
The sampling problem is, simply put, to generate independent and identically distributed (iid) draws from a target probability measure $\nu$---or, in practice, draws that are as close to iid as possible. 
Density estimation, on other hand, is the task of learning a tractable form for the density of $\nu$, given only a finite collection of samples.

% In this paper, we present a novel theoretical framework for understanding the approximation power of measure transport algorithms. 

Transport-based methods have recently emerged as a powerful approach to sampling and density estimation. They have attracted considerable attention in part due to the empirical success of their applications in machine learning, such as generative adversarial networks (GANs) \cite{goodfellow2014generative, gui2021review, creswell2018generative} and normalizing flows (NFs) \cite{marzouk-opt-map, rezende2015variational, kobyzev2020normalizing, papamakarios2021normalizing, tabak2013family, tabak2010density}. 
The transport approach can be summarized as follows:  Suppose we are given a {\em reference probability measure} $\eta$ from which sampling is easy, e.g., a uniform or standard Gaussian measure. Suppose further that we have a map $\TT$ which {\em pushes forward} the reference to the target $\TT_{\sharp}\eta = \nu$, i.e., $\nu(A)=\eta( (\TT)^{-1}(A))$ for every measurable set $A$. Then, samples $x_1, \ldots, x_n \stackrel{\rm iid}{\sim} \eta$ from the reference are transformed into samples $\TT(x_1),\ldots, \TT(x_n) \stackrel{\rm iid}{\sim} \nu$ from the target at negligible computational cost. 
Moreover, if $\TT$ is invertible and differentiable, the density of the target $\nu$ can be explicitly obtained via the change-of-variables formula \cite[Sec.~3.7]{bogachev1}. The challenge, then, is {\em to find a map $\hT$ that (exactly or approximately)} pushes forward  $\eta$ to  $\nu$.

While transport-based methods are empirically successful and popular in practice, our theoretical understanding of them is lacking (see Section~\ref{subsec:relevant-literature}). In particular, there is very little analysis of their
approximation accuracy. In practical settings, for example, one learns a map $\hT$ using some optimization scheme involving the target measure $\nu$ and some chosen reference measure $\eta$; here one must make a variety of approximation choices, and in general $\hat{T}$ does not transport $\eta$ to $\nu$. We therefore ask: 
\vspace{.2cm}
\begin{quote}
\begin{center}
 If an algorithm provides a map $\hT$, is the pushforward distribution $\widehat \nu = \hT _{\sharp} \eta$ a good approximation of the target measure $\nu$?
 \end{center}
\end{quote}
\vspace{.2cm}
{\em The primary goal of this article is to provide an answer to this question by {\bf (1)} providing error analysis and rates for a broad abstract class of measure-transport algorithms, which translate to the accuracy of the resulting sampling procedure described above;
and {\bf (2)} showing that many algorithms, including well-known methods
such as the triangular map approach of~\cite{marzouk2016sampling}, fall within this class and thus inherit our error analysis.} 

\smallskip
In considering measure transport algorithms, our primary motivating application is sampling. Measure transport is an emerging approach to sampling, where perhaps the most popular alternatives are Monte Carlo methods~\cite{casella}, which include Markov chain Monte Carlo (MCMC) and sequential Monte Carlo (SMC) algorithms. In general, these methods produce samples that are approximately distributed according to the target $\nu$. Such samples may also be highly correlated or non-uniformly weighted, and the associated algorithms might not be easily parallelizable \cite{casella}, leading to high computational costs.

When learned from a (typically unnormalized) 
density function, transport methods can be viewed as variational inference (VI) methods \cite{blei2017variational, zhang2018advances}. Broadly speaking, VI aims to approximate $\nu$ with a measure $\nu_\theta$ belonging to a parametric family; in the case of transport methods, this family can be identified as the set of measures that are pushforwards of the reference by a prescribed family of maps. The latter choice of transport family, therefore, has a direct impact on the accuracy of the approximation to $\nu$.

We primarily focus on the approximation problem, which as explained above, is immediately relevant to the task of drawing samples from $\nu$. We will not directly address the {\em statistical} problem of density estimation from \emph{finite} collections of samples using transport (see, e.g., \cite{wang2022minimax}), but our results are relevant to understanding the bias of such density estimation schemes.

% *format* of representing a target as the pushforward of some reference distribution through a map immediately links map estimation to density estimation, and suggests algorithms. But we are not tackling the statistical problem of density estimation here. 

% analysis is relevant to any subsequent analysis of the *bias* of density estimation; even in a nonparametric setting, one will need to understand how to increase the complexity of the transport map representation as the sample size increases, so that the “bias” and “variance” contributions to the overall error decrease at commensurate rates. 

We now summarize our main contributions in Section~\ref{subsec:overview}, followed by a detailed review of  the relevant literature  in Section~\ref{subsec:relevant-literature}. Key  notations  and definitions are provided in
in Section~\ref{subsec:Notation}. An outline of the article is presented in Section~\ref{subsec:outline}.

\subsection{Contributions}\label{subsec:overview}
Given a set $\Omega \subseteq \R^d$ equipped with two Borel probability measures, a {\it target} $\nu$ and a {\it reference} $\eta$, we consider the approximation to $\nu$: 
$$ \hnu \equiv \hT_\sharp \eta \, , \qquad  \hT \equiv 
\rev{ 
\argmin_{S \in \widehat{\mmT}} 
}
\, D(S_\sharp \eta, \nu) \, ,$$ 
where \rev{$\widehat \mmT$} denotes a parameterized approximation class of maps, e.g., 
polynomials of a certain degree, and $D\colon \PP(\Omega) \times \PP(\Omega) \to \R_+$ is a statistical divergence between probability measures, e.g., the Wasserstein distance or the Kullback-Leibler (KL) divergence. Our goal is to obtain bounds of the form 
\begin{equation*}
    D(\hnu, \nu) \le \rev{ C  
    \dist_{\| \cdot \|} ( \widehat \mmT, \TT),
    }
\end{equation*}
\rev{where $\widehat{\mmT}$ is the approximation class contained in a large enough Banach space of maps $\mmT$ from $\Omega$ to itself, the norm $\| \cdot \|$ is that of some space containing $\mmT$,
and $C> 0$ is a constant independent of $\widehat{\mmT}$}. \rev{We note that $\TT$ can be taken to be any} transport map that satisfies the {\em exact} pushforward relation $\TT_\sharp \eta = \nu$.
We present an abstract framework for obtaining such bounds in Section~\ref{sec:main-results} by combining three theoretical ingredients:
\begin{enumerate}[label=(\roman*)]
    \item {\bf Stability} estimates of the form 
    \rev{$D(F_\sharp \eta, G_\sharp \eta) \le C \| F - G \|$ 
for all maps $F, G \in \mmT$};
\item {\bf Regularity} results showing that $\TT \in \mT \subset \mmT$ where $\mT$ is a 
smoothness class, e.g., \rev{Sobolev space $H^{k}$ for $\mmT = L^2$ for a sufficiently 
large index $k$;}
\item {\bf Approximation} results that provide upper bounds  for \rev{$\dist_{\|\cdot \|} ( \widehat{\mmT}, T^\dagger)$.}
\end{enumerate} 
 Items (ii) and (iii) are independent of the choice of $D$ and
can be addressed using off-the-shelf results: Regularity can be derived from measure and elliptic PDE theory, e.g., on the regularity of optimal transport maps. Approximation bounds can be obtained from existing results, e.g., the approximation power of polynomials in $L^2$.

Stability ({i})
is the only part of the argument which depends explicitly on the choice of $D$, 
and its development is a major analytical contribution of this paper. While in the context of uncertainty propagation and inverse problems, some results have been proven in this direction when $D$ is the $L^q$ distance between the densities \cite{butler2018convergence, butler2022p, ditkowski2020density, sagiv2022spectral} or the Wasserstein distance \cite{sagiv2020wasserstein}, we provide new results for the Wasserstein distance, the maximum mean discrepency (MMD), and the KL divergence. These stability results (see Section~\ref{sec:stability}) are also  of independent interest in the statistics, applied probability, and data science communities.

%This From an analysis point of view, our framework uses three different types of results. The first two are classical (functional) approximation theory, and the regularity theory of optimal transport and KR maps. These are well developed and active fields from which we can relatively straightforwardly ``borrow'' existing results. The thirs element is so-called stability, which is of independent analytic interest. In short, the question here is: given two maps $f,g:\R ^n \to \R ^n$, how far are the pushforwards $f_{\sharp} \eta$ and $g_{\sharp}\eta$ from one another, and in what sense? 

Our third contribution is a series of applications (Section~\ref{sec:applications})
where we obtain rates of convergence of $D(\hnu, \nu)$ for various 
parameterizations of $\hT$ and under different assumptions on the target 
$\nu$ and the reference $\eta$. We supplement these 
applications with numerical experiments in Section~\ref{sec:numerics}
that demonstrate our theory, and even explore the validity of our approximation results beyond the current set of hypotheses. Lastly, for our applications, we present a new result concerning the approximation accuracy of neural networks on unbounded domains, Theorem \ref{thm:nn_unbounded_approx}, which is of independent interest.

\subsection{Review of relevant literature}\label{subsec:relevant-literature}
We focus our review of literature on theory and computational approaches to
measure transport. For a comprehensive review of Monte Carlo algorithms, see \cite{andrieu2003introduction, stuart-mcmc, casella}. 

%\subsubsection{Theory of transport}
Transportation of measure is a classic 
topic in probability and measure theory  
\cite[Ch.~9]{bogachev2}. While our paper is not limited to a particular type of transport map, let us first briefly review notable classes of such maps: optimal transport (OT) maps and triangular maps.

The field of OT is said to have been initiated by
 Monge \cite{monge1781memoire}, with the modern formulation introduced in the seminal work of
Kantorovich \cite{kantorovich1942translocation}. Since then, the theory of OT has flourished \cite{ambrosio2005gradient, benamou2000computational, gangbo1996geometry, santambrogio2015optimal, villani-OT}, with applications in PDEs \cite{figalli2017monge, gutierrez2001monge}, 
econometrics \cite{galichon2017survey, galichon2018optimal},
statistics \cite{panaretos2020invitation}, and data science \cite{peyre2019computational}, among other fields. Optimal maps enjoy many useful properties that we also utilize in our applications in 
Section~\ref{sec:applications}, such as uniqueness and regularity \cite{caffarelli1992regularity, caffarelli2000monotonicity, colombo2017lipschitz, figalli2017monge}. The development of numerical
algorithms for solving OT problems is a contemporary topic 
\cite{peyre2019computational}, although the majority of research in the 
field is focused on the solution of discrete OT problems and estimating 
Wasserstein distances \cite{cuturi2013sinkhorn, genevay2018learning},
with the Sinkhorn algorithm and its variants
being considered state-of-the-art \cite{peyre2019computational}. The numerical approximation of continuous OT maps is an even more recent subject of research. One approach has been to compute the Wasserstein-$2$ optimal transport map via numerical solution of the Monge--Amp\`{e}re PDE \cite{benamou2000computational, benamou2014numerical, froese2012numerical, lindsey2017optimal, nochetto2019pointwise}. Other modern approaches to this 
task involve \emph{plug-in estimators}: the discrete OT problem 
is first solved with the reference and target measures replaced by empirical approximations, then the 
discrete transport map is extended outside the sample set to obtain an approximate Monge map. The barycentric projection method, 
cf.~\cite[Rem.~4.11]{peyre2019computational} and \cite{seguy2018large, deb2021rates, pooladian2021entropic},
is the most popular among these, although other approaches such as 
 Voronoi tesselations \cite{li2021quantitative} are also possible. 
 The aforementioned works mainly consider the convergence of plug-in estimators 
 as a function of the number of samples in the empirical approximations of reference 
 and target measures (sample complexity), as well as the effect of 
 entropic regularization. This is in contrast to the problems of interest to us, where 
 we are mainly focused on the parameterizations of transport maps as opposed to sample 
 complexity. 

Triangular maps \cite[10.10(vii)]{bogachev2} are an alternative approach to transport 
that enjoy some of the useful properties of OT maps, together with additional structure 
that makes them computationally convenient. While triangular maps are  not optimal in the usual transport sense,
they can be obtained as the limit of a sequence of OT maps with increasingly asymmetric transport costs \cite{bonnotte2013knothe}.   The development of triangular maps 
in finite dimensions is attributed to the independent papers of Kn{\"o}the \cite{knothe1957contributions} 
and Rosenblatt \cite{rosenblatt1952remarks}; hence these maps are often called \emph{Kn{\"o}the--Rosenblatt (KR)} rearrangements. In the finite-dimensional setting, KR maps can be 
constructed explicitly and 
enjoy uniqueness and regularity properties that make them attractive in practice; cf.~\cite[Sec.~2.3]{santambrogio2015optimal} and \cite{bogachev2006nonlinear, bogachev2005triangular, zech2022sparseI, zech2022sparseII}. Triangular maps have other properties 
that make them particularly attractive for computation. 
For example, triangular structure enables fast (linear in the dimension) evaluation 
of log-determinants, which is essential when evaluating densities or identifying maps by minimizing KL divergence \cite{kobyzev2020normalizing, marzouk2016sampling, papamakarios2021normalizing, tabak2013family, tabak2010density}; \rev{see also  \cite{cui2022deep, westermann2023measure} where efficient transport maps are constructed by leveraging triangular maps within a deep tensor train formulation}. Triangular structure also enables the inversion of the maps using root-finding algorithms, akin 
to back substitution for triangular linear systems \cite{spantini2019coupling}. Finally, 
triangular maps can be used not only for transport but also for conditional simulation 
\cite{kovachki2020conditional, marzouk2016sampling}, a property that is not shared by OT maps 
unless additional constraints are imposed \cite{carlier2016vector, muzellec2019subspace}.
These properties have led to wide adoption of triangular flows in 
practical applications, ranging from Bayesian inference \cite{marzouk-opt-map, marzouk2016sampling, parno2018transport, spantini2018inference} to NFs 
and density estimation \cite{jaini2019sum, kobyzev2020normalizing, papamakarios2021normalizing, papamakarios2017masked}.

Analysis of transport map approximation and estimation has, for the most part, focused on the specific cases of optimal transport (OT) and Knothe-Rosenblatt (KR) maps.
The articles \cite{zech2022sparseI, zech2022sparseII} analyze the regularity of 
the KR map under assumptions on the regularity of reference and target densities, and 
obtain rates of convergence for sparse polynomial and neural network approximations 
of the KR map and for the associated pushforward distributions. In contrast, the articles \cite{irons2022triangular, wang2022minimax} 
analyze the sample complexity of algorithms for \emph{estimating} 
KR maps from empirical data: \cite{irons2022triangular} focuses on the statistical sample complexity of estimating the KR maps themselves; on the other hand, \cite{wang2022minimax} focuses on the  density estimation problem (e.g., characterizing error of the estimated pushforward distribution $\hT_\sharp \nu$ in Hellinger distance) using general classes of transports, though with specialized results for the triangular case. 
The article \cite{hutter2021minimax} studies convergence of a wavelet-based estimator of the Wasserstein-$2$ (or $L^2$) optimal (Brenier) map, obtaining minimax approximation rates (in $L^2$, and therefore also in the Wasserstein-$2$ metric on the transported measures) for estimation of the map from finite samples; this framework has been substantially generalized in \cite{pooladian2023minimax}. The work of \cite{lu2020universal} also analyzes the approximation error and sample complexity of a generative model based on deep neural network approximations of $L^2$ optimal transport maps.
 
\paragraph{Novelty} In the context of the measure transport literature, our contributions focus on a broader 
class of transport problems. In contrast with OT, we do not require our transport 
maps to push a reference measure exactly to a target; in other words, we relax the 
marginal constraints of OT, which enables freedom in the choice of the 
approximation class for the transport maps. Furthermore, inspired by  
the development of GANs and NFs, we consider transport maps that are obtained not as 
minimizers of a transport cost, 
\rev{
but as minimizers of a discrepancy between the pushforward 
of the reference and the target measure}.
Such problems are often solved in the computation of KR maps \cite{marzouk2016sampling} or NFs \cite{kobyzev2020normalizing} by
minimizing a KL divergence. We generalize this idea to other types of 
divergences and losses including Wasserstein and MMD; these losses have been shown to have good performance 
in the context of  GANs \cite{arjovsky2017wasserstein, binkowski2018demystifying, li2017mmd}. We further develop a general framework for obtaining error bounds that can be adapted to new divergences after 
proving appropriate stability results.

A second point of departure for our work is that we primarily consider the error that arises in the approximation of \emph{measures} via transport maps. This viewpoint stands in contrast to previous literature, in which a \emph{particular} map is first approximated and its pushforward is a derived object of interest, as in \cite{hutter2021minimax, irons2022triangular, zech2022sparseI,zech2022sparseII}. To our knowledge, this aspect of computational transport is mostly unexplored.

\subsection{Notation and definitions}\label{subsec:Notation}
Below we summarize some basic notation and definitions used throughout the article:
\begin{itemize}
    \item For $\Omega \subseteq \R ^d$ and $\Omega' \subseteq \R^m$, let $C^r(\Omega;\Omega')$ be the space of functions $f\colon \Omega \to \Omega'$ with $r\in \mathbb{N}$ continuous derivatives in all coordinates. 
    \item We write $J_f$ to denote the Jacobian matrix of $f:\Omega \to \Omega '$. 
    \item We use $\PP(\Omega)$ to denote the space of  Borel probability measures on $\Omega$.
    
    \item For $\mu \in \PP(\Omega)$, denote the weighted $L^p_{\mu}$ norm by 
    $\|f\|_{L^p_{\mu}(\Omega; \Omega')} \equiv  ( \int_{\Omega} |f|^p \, d\mu )^{1/p} \, , $
    where $|\cdot |$ is the usual Euclidean norm, as well as the 
    corresponding function space $L^p_\mu(\Omega; \Omega')$. 
    
    \item For all $p\geq 1$ and $k\in \mathbb{N}$, denote the weighted Sobolev space $W_{\mu}^{k,p}(\Omega;\Omega')$ as the space of functions $f:\Omega \to \Omega$ with (mixed) derivatives of degree $\leq k$ in $L^p_{\mu} (\Omega; \Omega')$ equipped with the norm 
    $ \|f\|_{W^{k,p}_{\mu}(\Omega; \Omega')} \equiv ( \sum_{\|{\bf j}\|_1 \leq k} \|D^{\bf j} f\|_{L^p_{\mu}(\Omega; \Omega')}^p )^{1/p},$ 
    where ${\bf j}= (j_1, \ldots, j_d)\in \mathbb{N}^d$ and $\|{\bf j}\|_1 = j_1 + \cdots +j_d$. 
    We write $H_\mu^k(\Omega; \Omega') = W^{k, 2}_\mu(\Omega; \Omega')$ following the standard notation in functional analysis 
    and suppress the subscript $\mu$ whenever the Lebesgue measure is considered. We will also suppress the range and domain 
    of the functions to simplify notation when they are clear from context.
    \item Given two measurable spaces $(X, \Sigma _x), (Y, \Sigma _y)$, a measurable function $T:X \to Y$, and a measure $\mu$ on $X$, we define $T_{\sharp} \mu$, the pushforward of $\mu$ by $T$, as a measure on $Y$ defined as  
    $ T_{\sharp} \mu (E) \equiv \mu \left( T^{-1}(E)\right)$ for all $ E \in \Sigma_y$ 
    where $T^{-1}$ is to be understood in the set-valued sense, i.e., $T^{-1}(E) = \{ x\in X ~~ | ~~ T(x)\in E \}$. Similarly, we denote by $T^{\sharp}\mu \coloneqq (T^{-1})_{\sharp}\mu$ the {\em pullback} of a measure (on $X$) for any measure $\mu$ on $Y$.
    
    \item Throughout this paper, $\eta$ is the reference probability measure on $\Omega$, $\nu$ is the target measure, and $\TT$ is an exact pushforward $\TT_{\sharp} \eta = \nu$.
    \item     We say that a function $D: \PP(\Omega) \times \PP(\Omega) \to [0, + \infty]$ is a divergence 
    if $D(\mu, \nu) = 0$ if and only if $\mu = \nu$.
\end{itemize}

\subsection{Outline}\label{subsec:outline}
The rest of the article is organized as follows:
Section~\ref{sec:main-results} summarizes our main contributions and a general framework for the error analysis of 
measure transport problems.
Section~\ref{sec:stability} follows with stability analyses for 
Wasserstein distances, MMD, and the KL divergences, with some of the major technical proofs postponed 
to Section~\ref{app:stability-proofs}.  Section~\ref{sec:applications} presents various 
applications of our general error analysis framework and of our stability analyses, including new approximation 
results for neural networks on unbounded domains, again with some technical proofs postponed to Section~\ref{sec:application-proofs}.
Section~\ref{sec:numerics} presents our numerical experiments, followed by 
concluding remarks in Section~\ref{sec:conclusions}.

\section{Error analysis for measure transport}\label{sec:main-results}
In this section we present our main theoretical results concerning the 
error analysis of measure transport problems.
We present a general strategy for obtaining error bounds by combining: stability 
results for a divergence of interest, regularity results for an appropriate 
fixed transport map, and approximation theory for high-dimensional functions. 
% , those results can be read as functional-analytic or probabilistic inequalities, independent of our application to sampling and/or density estimation. %the applicative settings of sampling algorithms.

%\subsection{Abstract error analysis}\label{}
%Proposition \ref{prop:wp_dist} proposes a template for proving convergence and convergence rates for sampling by transport map. The strategy for such results is as follows: 

Consider a Borel set $\Omega \subseteq \R ^d$ and let $\eta, \nu \in \PP( \Omega)$. Our goal is to approximate the target $\nu$
by a pushforward of the reference $\eta$. To do so, we consider:
\begin{itemize}
    \item \rev{$\mmT$, a class of functions in a Banach space of mappings from $\Omega$ to itself;}
 \item $\widehat{\mmT} \subseteq \mmT$, a closed (possibly finite-dimensional) subset; and
 \item $D$, a statistical divergence on $\PP (\Omega)$  \rev{(or some subset which contains both $\nu$ and $\eta$)}.
\end{itemize} 
We propose to approximate the target $\nu$ with another measure~$\hnu$, defined as follows:
\begin{equation}\label{eq:abst_opt}
\hnu = \hT_\sharp \eta \, , \qquad \hT \in 
\rev{ 
\argmin\limits_{S \in \widehat{\mmT}} 
}
\,D(S_\sharp \eta, \nu) \, . 
\end{equation} 
 Note that, in general, the minimization problem in  \eqref{eq:abst_opt} does not admit a unique solution; hence, by writing ``$\hT \in \argmin $'' we mean, here and throughout the paper, an arbitrary choice of a global minimizer.
Our goal  is to bound the approximation error $D(\hnu,\nu)$. While our focus 
in this article is on cases where
$D$ is a Wasserstein-$p$ metric, the MMD distance,  or the KL divergence, we 
give an abstract theoretical result that is applicable to any choice of $D$ 
once a set of assumptions are verified. 

\begin{assumption}\label{ass:abstract}
The  measures $\eta, \nu \in \PP(\Omega)$ and the divergence $D$ satisfy the following conditions:
\begin{enumerate}[label=(\roman*)]
\item {\it (Stability)} For any set of maps $F, G \in \mmT$, \rev{there exists a constant $C>0$ (independent of $F, G$)}  such that
% There exists a constant $C >0$ so that for any maps $F, G \in \mmT$
%it holds that 
\begin{equation}\label{eq:general-stability-estimate}
    D(F_\sharp \eta, G_\sharp \eta) 
\leq C 
\rev{
\|F- G\| \, ,
}
\end{equation}
 \rev{for some norm $\| \cdot \|$ of an ambient space containing $\mmT$.}

\item {\it (Feasibility)} There exists a map $\TT \in \mmT$ satisfying  $\TT_\sharp \eta = \nu$.

% Note that even though $\TT_\sharp\eta =\nu$, it is generally {\em not the case} for $T^*$, and we expect that $T^*_\sharp\eta \neq \nu$.
\end{enumerate}
\end{assumption}

Condition (i) simply states that 
 the divergence between 
pushforwards of $\eta$ is controlled by the distance between the maps. It is important 
to highlight that this condition is independent of the target $\nu$ and only needs to 
be verified for a fixed reference $\eta$ and the class $\mmT$. This, in turn, implies that the constant 
$C>0$ may depend on the choice of $\eta$ and $\mmT$.\footnote{For example, in Section~\ref{subsec:KL-error-bounds} we verify stability of KL  when $\Omega$ is unbounded
 only when $\eta$ is a Gaussian.} Condition~(ii) involves both the reference and 
 target measures and requires the existence of a transport map 
 between the two measures. 
 Many choices of $T^\dagger$ are often possible; for example optimal or triangular
 transport maps 
 exist under mild conditions. Thus, condition (ii) asks for $\mmT$ to be sufficiently large to contain at least one such map. In order to obtain useful 
 error rates, we typically like to show a stronger result, that is \rev{$ T^\dagger$ belongs 
 to a smaller, more regular, subset of $\mmT$}.
 For example, one may take  $\mmT = L^2$ \rev{and have $T^\dagger \in H^{k}$ for some $k \geq 1$.} 
  We dedicate Section \ref{sec:stability} to verifying condition (i), while existing results 
  from literature will be used to verify condition (ii) depending on 
  the application at hand, as outlined in Section~\ref{sec:applications}.
We are now ready to present our main abstract theoretical result. 

\begin{theorem}\label{thm:abstract-error}
Suppose Assumption~\ref{ass:abstract} holds and consider $\hnu$ as in \eqref{eq:abst_opt}. 
Then 
it holds that 
\begin{equation}\label{eq:thm:abstract-error}
    D(\hnu, \nu) \le C \: \rev{ 
    \dist_{\|\cdot \|} \left( \widehat{\mmT} \:, \: \TT \right),
    }
\end{equation}
where $C>0$ is the same constant as in Assumption~\ref{ass:abstract}(i).
\end{theorem}

\begin{proof}
Since $\hT$ is the minimizer of \eqref{eq:abst_opt}, it follows that 
\begin{equation*}%\label{eq:abst_reg}
    D(\hnu,\nu) = D(\hT_\sharp\eta, \TT_\sharp \eta) \leq D(T_\sharp\eta, \TT_\sharp \eta) \, , \qquad \forall \, \rev{ T\in \widehat{\mmT} } \,.
    \end{equation*}
Then Assumption~\ref{ass:abstract}(i) yields
\begin{equation}\label{eq:reg+stab}
    D(\hnu,\nu) \le \rev{ C \|T - \TT \| \, , \qquad \forall \, T\in \widehat{\mmT} }  \, .
\end{equation}
Now consider the map 
\begin{equation}\label{eq:Tstar_def}
T^* \coloneqq 
\rev{
\argmin_{T\in \widehat \mmT} \|T- T^\dagger\| \, ,    
}
\end{equation}
which exists since $\widehat \mmT$ is closed in $\mmT$. Evaluating 
the right hand side of \eqref{eq:reg+stab} with $T = T^*$ yields the 
desired result since 
\rev{$\| T^* - T^\dagger \| = \dist_{\|\cdot \|} ( \widehat{\mmT}, T^\dagger)$.}
{}
\end{proof}

The above theorem  reduces the question of controlling the error between 
$\hnu$ and $\nu$ to that of controlling the approximation error of $T^\dagger$ 
within the class \rev{ $\widehat \mmT$ }---in other words, an exercise in high-dimensional 
function approximation. This observation can guide the design of practical algorithms:
 obtaining optimal 
convergence rates requires the identification of \rev{ $T^\dagger \in  \mmT$ that is maximally 
regular.} Observe that, the maps $\TT, T^*$ in \eqref{eq:Tstar_def} are purely analytic 
elements of our theory and are not explicit in the optimization problem \eqref{eq:abst_opt} {\em and we have some freedom in choosing both.}
We can then choose  a (possibly finite-dimensional) approximating 
class \rev{$\widehat \mmT \subseteq \mmT$} that can achieve the fastest possible convergence rate 
for \rev{$\TT$.}
Afterwards, choosing $D$ can be guided by two main considerations: first, whether the stability 
condition can be verified with the other elements of the framework in place and, second, whether minimizing $D$ is a computationally tractable task.

The error bound \eqref{eq:thm:abstract-error} quantifies the trade-off between 
the complexity of the approximating class \rev{$\widehat \mmT$} and the accuracy of the algorithm:
if the approximating class \rev{$\widehat \mmT$} is rich and large (e.g., it is the space of polynomials of a very high degree), it can approximate $\TT$ well and so the right hand side of  \eqref{eq:thm:abstract-error}, the error estimate, is small. On the other hand, in many cases a rich (or large) class \rev{$ \widehat \mmT$} would make the algorithm \eqref{eq:abst_opt} more costly, as we are optimizing over a larger family of parameters/functions.
 In the extreme case where \rev{$\widehat \mmT= \mmT$} then $\hnu = \nu$ as expected and $D(\hnu,\nu)=0$ trivially (since $D$ is a divergence), independently of whether \rev{$\TT$ is unique in $\mmT$.}

\section{Stability analysis}\label{sec:stability}
As mentioned earlier, 
a major analytic advantage of  Theorem~\ref{thm:abstract-error}
is that it allows us to use existing results from approximation theory 
to control the error of $\hnu$. Applying this result requires us 
to verify Assumption~\ref{ass:abstract}. Among the two conditions, feasibility can also 
be verified using existing results from theory of transport maps, and optimal transport in particular.
The stability condition, however, needs development
and is the subject of this section. 

Let us review some existing results, 
\rev{stating with the case of scalar valued maps, i.e.,}
maps from $[0,1]^d$ to $\R$. If the divergence $D$ is taken to be the $L^p$-norm between the densities (which coincides with the total variation for $p=1$ and the mean square error for $p=2$), then \eqref{eq:general-stability-estimate} holds when $\mmT$ 
is taken as $C^1([0,1]^d)$ or $H^s([0,1]^d)$ for sufficiently large $s\geq 1$ \cite{ditkowski2020density, sagiv2022spectral}. In particular, when $d=1$, one can take $s=1$, which is conjectured to be sharp. Much more robust results are obtained when we choose $D=W_p$, the Wasserstein-$p$ distance, and $\mathcal{T}=L^p(\Omega;\varrho)$ for any $\varrho \in \PP(\Omega)$ \cite{sagiv2020wasserstein}; see Section \ref{subsec:Wp-error-bounds} for a generalized and simplified proof. 

\rev{For the case of vector valued maps,
the $L^p$ norm between the densities for $1\leq p \leq \infty$ is considered in \cite{butler2018convergence, butler2022p} for maps from $\R^d$ to $\R ^m$ for arbitrary dimensions $d$ and $m$, but under somewhat different and more stringent assumptions. When the maps are triangular from a compact domain $\Omega$ to itself and $\mathcal{T}= W^{1, \infty}(\Omega)$, stability results for a wide range of divergences are derived in \cite{zech2022sparseI}. In addition, a bound 
in the Hellinger distance for tensor train approximations of 
triangular maps can be found in \cite[Thm.~1]{cui2022deep}.}

Below we present our stability results for the Wasserstein distance in Section~\ref{subsec:Wp-error-bounds}, followed by 
MMD in Section~\ref{subsec:MMD-error-bounds} and KL in Section~\ref{subsec:KL-error-bounds}.

\subsection{Wasserstein distances}\label{subsec:Wp-error-bounds}
We now show the stability estimate when $D$ is taken to be a Wasserstein distance.
 We recall some basic definitions first:
Let $p\geq 1$ and denote by $\PP^p(\Omega)$ the subset of $\PP(\Omega)$ consisting of probability measures 
with finite $p$-th moments. 
Then for $\mu, \nu \in \PP^p(\Omega)$ we define their Wasserstein-$p$ distance \begin{equation}\label{eq:wasdef}
    \W_p(\mu, \nu) \coloneqq \inf\limits_{\pi \in \Gamma(\mu,\nu)} K_p^{\frac{1}{p}}(\pi) \, , 
    \qquad K_p(\pi) \coloneqq \int\limits_{\Omega \times \Omega} |x-y|^p \, \dd \pi (x,y) \, ,
\end{equation}
where $|x-y|$ is the Euclidean distance in $\R^d$ and $\Gamma (\mu, \nu)$ is the set of all Borel probability measures on $\Omega \times \Omega$ with marginals $\mu$ and $\nu$, i.e.,
\begin{equation}
    \mu(A) = \pi(A\times \Omega) \, , \qquad \nu (A)=  \pi(\Omega \times A) \, ,
\end{equation}
for any $\pi \in \Gamma(\mu, \nu)$ and any Borel set $A\subseteq \Omega$. See \cite{villani-OT, santambrogio2015optimal} for a detailed treatment of Wasserstein distances including their 
extensions to metric spaces. Our main stability result then reads as follows; 
% see proof in Section \ref{subsec:proof-of-Wasserstein-stability}.

\begin{theorem}\label{lem:wplp}
Let $\Omega \subseteq \mathbb{R}^d$ and $\Omega' \subseteq \mathbb{R}^s$
be  Borel sets and fix  $\mu \in \PP^p(\Omega)$ for $p \ge 1$. 
For any $q \ge p$ and  $F, G \in L^q_\mu(\Omega;\Omega')$
it holds that
% \begin{equation}\label{eq:wp_lp}
% W_p (f_\sharp\mu, g_\sharp\mu) \leq \|f-g\|_{L^p (\Omega\to \mathbb{R}^n;\mu)} \, .
% \end{equation}
% Furthermore, for $f,g \in L^q= L^q(\Omega \to \mathbb{R}^n ; \mu)$ and $q \geq p$, we have
\begin{equation}\label{eq:wp_lp}
W_p (F_\sharp \mu, G_\sharp \mu) \leq \|F-G\|_{L^q_\mu (\Omega; \Omega')}.
\end{equation}
\end{theorem}

\begin{proof}
    Let $\pi$ be a coupling with marginals $F_\sharp \mu$ and $G_\sharp \mu$. Since the $W_p$ distance is defined as the infimum over all couplings $\pi \in \Gamma(F_\sharp \mu, G_\sharp \mu)$, the distance can be bounded by one particular coupling. Choosing $\pi$ to be the joint law of $F_{\sharp}\mu$ and $G_{\sharp}\mu$, i.e., $$\pi (A\times B) = \mu \left\{ t\in \Omega ~{\rm s.t.}~F(t)\in A ~{\rm and } ~G(t)\in B\right\}, $$
    for every measurable $A,B\subseteq \R$. We have that
\begin{equation*}
W_p^p(F_\sharp \mu,G_\sharp \mu)= \inf_{\pi \in \Gamma(F_\sharp \mu, G_\sharp \mu)} \int_{\Omega' \times \Omega'} |x - y|^p \pi(\dd x,\dd y) \leq \int_{\Omega' \times \Omega' } |x - y|^p 
(F \times G)_\#\mu( \dd x, \dd y).
\end{equation*} 
Then, by a change of variables, we write
\begin{equation*}
\int_{ \Omega' \times \Omega'} |x - y|^p (F \times G)_\#\mu(\dd x, \dd y) = \int_{\Omega} |F(t) - G(t)|^p \dd\mu(t) = \| F -G \|_{L^p_{\mu}(\Omega ;\Omega')}^p.
\end{equation*}
Lastly, using Jensen's inequality with the concave function $x \mapsto x^{p/q}$ for $p/q \leq 1$, we have
\begin{align*}
\|F - G \|_{L^p_{\mu}(\Omega ;\Omega')}^{p} &= \int_{\Omega} (|F(t) - G(t)|^{q})^{p/q} \dd\mu(t)\\  &\leq \left(\int_{\Omega} |F(t) - G(t)|^{q} \dd\mu(t) \right)^{p/q} = \| F - G \|_{L^q_{\mu}(\Omega ;\Omega')}^{p} \, .
\end{align*}

\end{proof}

We note the simplicity of the above result and its proof, and in particular the fact 
that we only need the maps $F,G$ to be appropriately integrable with respect to 
the reference measure $\mu$. Indeed, the Wasserstein stability result 
is the most robust and theoretically convenient, of our three stability 
theorems. Furthermore, the above result implies that $W_p$ 
satisfies Assumption~\ref{ass:abstract}(i) with constant $C = 1$.

\subsection{MMD}\label{subsec:MMD-error-bounds}
We now turn our attention to the case where $D$ is taken to 
be the MMD defined by a kernel $\kappa$. We recall the definition 
of MMD following \cite{muandet2017kernel}.
A function $\kappa: \Omega \times \Omega \to \R$ is
called a Mercer kernel if it is symmetric, i.e., $\kappa(x, y) = \kappa(y, x)$, 
and positive definite, in the sense that
\begin{equation*}
    \sum_{i,j=1}^m c_i c_j \kappa(x_i, x_j) \ge 0 \, , \qquad 
    \forall m \in \mathbb{N}, \:
    c_1, \dots,  c_m \in \R, \:  
    x_1, \dots, x_m \in \Omega.
\end{equation*}
Any Mercer kernel $\kappa$ defines a unique reproducing kernel Hilbert space (RKHS) of functions from $\Omega$ to $\R$. Consider first the set of functions 
\begin{equation*}
    \widetilde{\mK} \coloneqq \left\{ f: \Omega \to \R : 
     f(x) = \sum_{j=1}^m a_j \kappa(x, x_j), \quad \text{ for } m \in \mathbb{N}, 
     \: a_j \in \R, \: x_j \in \Omega
    \right\} \, .
\end{equation*}
Given two functions 
$f(x) =  \sum_{j=1}^m a_j \kappa(x, x_j)$
and $f'(x) = \sum_{j=1}^{m'} a'_j \kappa(x, x'_j)$ in $\widetilde{\mK}$,
define the inner product $\langle f, f' \rangle_{\widetilde{\mK}} 
\coloneqq \sum_{j=1}^m \sum_{k=1}^{m'} a_j a'_k \kappa(x_j, x'_k)$. 
The RKHS  $\mK$ of the kernel $\kappa$ is  defined as the completion of $\widetilde{\mK}$ with respect 
to the RKHS norm induced by the above inner product. It is a Hilbert space with inner product 
denoted by $\langle \cdot, \cdot \rangle_\mK$ and the associated norm. We note here that many standard Hilbert spaces are in fact RKHSs, e.g., the Sobolev space $H^s(\Omega)$ 
for a smooth domain $\Omega \subseteq \R^d$ and $s > d/2$.

The space $\mK$ has two important 
properties:  (i) $\kappa(x, \cdot) \in \mK$
for all $x \in \Omega$; and (ii)
(the reproducing property) $f(x) = \langle f, \kappa(x, \cdot) \rangle_{\mK}$ for all $f\in \mK$.
A map $\psi: \Omega \to \mK$ is called a {\it feature map} for $\kappa$ if 
$\kappa(x,y) = \langle \psi(x),  \psi(y) \rangle_\mK$ for every $x, y \in \Omega$. Such a map $\psi$ always exists since 
we can simply take $\psi = \kappa(x, \cdot)$, the {\it canonical} feature map.

We further define the kernel mean embedding of probability 
measures $\mu \in \PP(\Omega)$ with respect to $\kappa$ as 
    $\mu_\kappa \coloneqq  \int_\Omega \kappa(x, \cdot) \mu(\dd x) \in \mK $ along with the  subspace 
\begin{equation*}
    \PP_\kappa(\Omega) \coloneqq \left\{ \mu \in \PP(\Omega) : 
    \int_\Omega \sqrt{\kappa(x,x)} \mu( \dd x) < + \infty \right\} \, .
\end{equation*}
We finally define the MMD  between probability measures $\mu, \nu \in \PP(\Omega)$ with respect 
to  $\kappa$ as 
\begin{equation*}
    \MMD_\kappa(\mu, \nu) \coloneqq \left\{ 
    \begin{aligned}
        & \| \mu_\kappa - \nu_\kappa \|_{\mK},  && \text{if } \mu, \nu \in \PP_\kappa(\Omega) , \\ 
        & +\infty, && \text{otherwise} \, .
    \end{aligned}
    \right.
\end{equation*}
\begin{theorem}\label{thm:MMD-stability}
Let $\Omega \subseteq \R^d$ and $\Omega' \subseteq \R^s$ be Borel sets and 
fix $\mu \in \PP(\Omega)$ and a Mercer kernel $\kappa: \Omega' \times \Omega' \to \R$ with RKHS $\mK$.
Suppose $\kappa$ has a feature map $\psi: \Omega' \to \mK$  and there exists a function 
$L: \Omega' \times \Omega' \to \R_+$ so that 
\begin{equation*}
    \| \psi(x) - \psi(y) \|_\mK \le L(x, y) | x- y|.
\end{equation*}
Suppose $F, G \in L^q_\mu(\Omega; \Omega')$
such that $L( F(\cdot), G(\cdot) ) \in L^p_\mu(\Omega; \R)$ 
for H{\"o}lder exponents $p,q \in [1, \infty]$ satisfying $1/p + 1/q =1$.
Then it holds that 
\begin{equation*}
    \MMD_\kappa(F_\sharp \mu, G_\sharp \mu) \le \| L(F(\cdot), G(\cdot)) \|_{L^p_\mu(\Omega; \R)} 
    \| F - G \|_{L^q_\mu(\Omega; \Omega')}.
\end{equation*}
\end{theorem}
\begin{proof}
By the definition of MMD we have that 
\begin{equation*}
\begin{aligned}
      \MMD_\kappa( F_\sharp \mu, G_\sharp \mu ) 
    & = \left\| \int_{\Omega} \kappa( F(x), \cdot) \mu(\dd x) - \int \kappa(G(x), \cdot) \mu(\dd x) \right\|_\mK, \\
    & = \left\| \int_\Omega \kappa( F(x), \cdot)  - \kappa(G(x), \cdot) \mu(\dd x) \right\|_\mK, \\ 
    & \le  \int_\Omega \| \kappa( F(x), \cdot)  - \kappa(G(x), \cdot) \|_\mK \mu(\dd x).
\end{aligned} 
\end{equation*}
By the hypothesis of the theorem and H\"older's inequality we can further write 
\begin{equation*}
    \begin{aligned}
          \MMD_\kappa( F_\sharp \mu, G_\sharp \mu )  
          & \le  \int_\Omega \| \psi( F(x))  -  \psi(G(x)) \|_\mK \mu(\dd x)  \\
          & \le  \int_\Omega L(F(x), G(x)) | F(x) - G(x)|     \mu(\dd x)  \\
          & \le  \| L(F(\cdot), G(\cdot) ) \|_{L^p_\mu(\Omega; \R)}    \| F - G \|_{L^q_\mu(\Omega; \Omega')},
    \end{aligned}
\end{equation*}
    for H\"older exponents $p,q \in [1, +\infty]$.
\end{proof}

\begin{remark}
    We emphasize that in general, MMD is not a proper divergence, since for certain choices 
of $\kappa$ one can have
$\MMD_\kappa(\mu, \nu) =0$ even when $\mu \neq \nu$. However, if $\kappa$ is a  \emph{characteristic kernel}, then $\MMD_\kappa$ is a 
proper divergence; see precise statements and definitions in \cite[Sec.~3.3.1]{muandet2017kernel}. Many standard kernels are characteristic, e.g., the Gaussian kernel $\kappa(x,y) = \exp(- \gamma^2| x -y |^2)$) for any $\gamma^2 >0$.
Even so, Theorem \ref{thm:MMD-stability} holds regardless of whether $\kappa$ is characteristic and whether ${\rm MMD}_{\kappa}$ is a divergence.
\end{remark}
\begin{remark}[Applicability of the hypotheses]\label{rem:mmd_hyp}
We note that our stability result for MMD is also fairly simple and general, 
although it has more technical assumptions compared to the Wasserstein case. 
However, the technical assumptions only concern the choice of the kernel $\kappa$
and, in particular, the local Lipschitz property of its feature maps. Indeed, these conditions are 
easily verified for many common kernels used in practice (see also the proof of Proposition~\ref{prop:application-compact-rates}):
Simply taking $\psi(x) = \kappa(x, \cdot)$ we can use the reproducing property to write 
\begin{equation}\label{MMD-kernel-expansion}
\begin{aligned}
  \| \psi(x) - \psi(y) \|^2_\mK & = \langle \kappa(x, \cdot) - \kappa(y, \cdot), \kappa(x, \cdot) - \kappa(y, \cdot)  ) \rangle_\mK  \\     
  & = \kappa(x, x) + \kappa(y, y) - 2 \kappa(x, y).
\end{aligned}
\end{equation}
Now suppose the kernel has the form $\kappa(x,y) = h(|x - y|)$ (such kernels are often referred to as 
\emph{stationary}), then the
the condition of Theorem~\ref{thm:MMD-stability} 
simplifies to 
\begin{equation*}
    h(0) - h(|x-y|) \le \frac{1}{2} L^2(|x- y|) | x - y|^2. 
\end{equation*}
Thus the function $L$ is dependent on the regularity of $h$. In the case of the 
Gaussian kernel $h(t) = \exp( - \gamma^2 t^2)$ it  follows from the mean value theorem
that $L(\gamma) >0$ is simply a constant. 
\end{remark}

\subsection{KL divergence}\label{subsec:KL-error-bounds}
\rev{
For our final choice of the divergence $D$ we consider the KL divergence. 
Recall that for two probability measures $\mu, \mu' \in \PP(\Omega)$ the KL divergence is defined as 
${\rm KL}(\mu \| \mu') 
 \coloneqq \int_{\Omega} \log\left( \frac{d\mu}{d\mu'} \right) d\mu \,,$ 
where $d\mu/d\mu'$ is the Radon-Nikodym derivative of $\mu$ with respect to $\mu'$. In this section, we will only concern ourselves with absolutely continuous measures with densities $p_{\mu}, p_{\mu'}\in L^1 (\Omega)$, in which case the KL divergence reads as
\begin{equation}\label{eq:KL_def}
{\rm KL}(\mu \| \mu') 
 \coloneqq \int\limits_{\Omega} \log\left( \frac{p_\mu (x)}{p_{\mu'} (x) } \right) \, p_{\mu}(\rev{x}) \, d\rev{x} \, .
\end{equation}

\subsubsection{Background - the KL minimization problem}
By definition, the KL divergence is not symmetric, and hence it is not surprising that measuring
the divergence \rev{from $\mu$ to $\mu'$} or vice versa has a profound impact on our analysis. Furthermore, minimizing the KL divergence over its first or second argument to approximate the other measure in a limited family of distributions results in different behavior. In particular, minimizing ${\rm KL}(\mu||\mu')$ over $\mu$ (so-called reverse KL minimization) encourages $p_{\mu}$ and $p_{\mu'}$ to be close in high-probability regions of $\mu'$. For multivariate Gaussian $p_{\mu}$, this results in \emph{mode-seeking} approximations that fit one mode of $\mu'$. On the other hand, minimizing ${\rm KL}(\mu||\mu')$ over $\mu'$ (so-called forward KL minimization) encourages $p_{\mu}$ and $p_{\mu'}$ to be close over the entire support of $\mu$. For Gaussian $p_{\mu'}$, this results in \emph{mean-seeking} approximations that match the first two moments of $\mu$. We refer to~\cite{wainwright2008graphical} for a more in depth discussion of these two optimization problems. While reverse KL minimization is used in variational inference to approximate a target measure whose density is known (possibly up to the normalizing constant)~\cite{blei2017variational, rezende2015variational}, forward KL minimization (i.e., maximum likelihood estimation) is commonly used when a target measure is only prescribed using samples~\cite{bishop2006pattern, papamakarios2021normalizing}. Given that transport-based approximations are used in both settings, we present applications to minimizing both directions of the KL divergence in the next section.

Moreover, since  
\eqref{eq:KL_def} involves the Lebesgue densities of the measures, the direction of transport is also 
of great importance. To be more precise, suppose $T$ is invertible and take $\mu = T_\sharp \eta$ 
for a reference measure $\eta$ with density $p_\eta$, then by the change of variables formula 
we have that 
}
\begin{equation} \label{eq:pushforward_density}
\rev{ p_\mu = p_{T_\sharp \eta}(x) = } p_{\eta}(T^{-1}(x))|J_{T^{-1}}(x)|.
\end{equation}
Therefore, when we are dealing with pushforward measures (forward transport), controlling KL
will involve properties of the inverse map $T^{-1}$. This is also the case in forward uncertainty quantification problems, where usually the map is from $\R ^d$ to $\R$, and the more general co-area formula is invoked \cite{ditkowski2020density, evans2018measure}.

To alleviate some of the difficulties involved with using the inverse Jacobian $J_{T^{-1}}=J_T ^{-1}$, consider the inverse map $T^{-1}$ and define the notation 
\rev{
$T^\sharp \eta \coloneqq (T^{-1})_\sharp \eta$; 
we refer to this measure as the \emph{pullback} of $\eta$ by $T$.
The change of variables formula now gives, for the measure $\mu = T^\sharp \eta$, 
\begin{equation} \label{eq:pullback_density}
p_\mu = p_{T^\sharp \eta}(x) = p_{\eta}(T(x))|J_T(x)|.
\end{equation}}
Already here we see that, when dealing with the pullback, one only needs to study the 
Jacobian $J_T$ rather than its inverse. 

{\em This seemingly technical point will have a major impact on our analyses below.} Working with the pullback rather than the pushforward
is not just technically convenient, but it is also of great practical interest  in density estimation \cite{tabak2013family}: if $\nu = T^\sharp \eta$ for a reference $\eta$, then $p_\nu$ is given by formula 
\eqref{eq:pullback_density}. 
Therefore, seeking the inverse map directly in the \emph{backward transport} problem provides a direct estimate of the target measure.
%The backward transport problem is also of great interest in the context of 
%NFs \cite{papamakarios2017masked, rezende2015variational} where the goal is precisely to parameterize and estimate the inverse map $T^{-1}$. 
The map $T$ is 
then approximated by inverting $T^{-1}$, for example, using bisection or by means of other numerical algorithms; 
see Section~\ref{subsec:pullback_KR} for more details.

In what follows 
we present our stability analysis of the KL divergence for both forward and backward transport problems. 
The forward transport problem is more challenging to analyze and requires more stringent assumptions
while our backward transport analysis is easier and leads to broader assumptions on the 
measures and maps. Since the proofs themselves are long and somewhat technical, we present them separately, in Section~\ref{app:stability-proofs}.

% \begin{equation}\label{eq:KL_def}
% {\rm KL}(\nu || \mu) 
%  \coloneqq \int\limits_{\Omega} \log\left( \frac{p_\mu (x)}{p_\nu (x) } \right) \, p_{\mu}(y) \, dy \, .
% \end{equation}

\subsubsection{Forward transport}\label{subsubsec:forward-KL}

We present two stability results: 
Theorem~\ref{thm:kl_compact} applies to compact domains $\Omega$ with general references $\eta$, whereas
Theorem~\ref{thm:kl} takes 
$\Omega= \R^d$ and fixes the reference $\eta$ to be the standard Gaussian. The reason for presenting these separate results is 
technical issues in proving stability of KL in the unbounded setting 
leading to more stringent assumptions on the tail behavior of the maps or the associated densities.  On compact domains, we can establish stability for a wide
choice of reference measures $\eta$ and with more relaxed assumptions 
on the maps. The proof of the following theorem is  presented in 
Section~\ref{sec:compact_kl_pf}.

\begin{theorem}\label{thm:kl_compact}
Let $\eta$ be an absolutely continuous probability measure with density $p_{\eta}$ on a 
compact
domain $\Omega \subset \mathbb{R} ^d$ and $F, G:\Omega\to \Omega$ be measurable.
Suppose that: 
\begin{enumerate}[label=(A\arabic*)]
\item $F, G\in C^2(\Omega; \Omega)$ \label{as:c2_compact}
%\item $|\partial_i f_i(x)|, |\partial_i g_i (x)| > \tau > 0$ for every $x\in \R^d$. \label{as:diag_lbd}
\item $F,G$ are injective and $G$ is invertible on $F(\Omega)$. \label{as:bijective_compact}
\item $c_{\eta}  \coloneqq \min_{x\in \Omega} p_{\eta} (x)  > 0$.\label{as:eta_pos_compact}
\item $p_{\eta}$ has a Lipschitz constant $L_{\eta}\geq 0$.\label{as:eta_Lip_compact}
\end{enumerate}
Then it holds that
\begin{equation}\label{eq:kl_H1_compact}
\KL(F_\sharp \eta \: \| \: G_\sharp \eta) \leq C \|F-G\|_{W_\eta^{1,1}(\Omega; \Omega)} \, ,
\end{equation}
where $C>0$ depends only on $d, c_{\eta},L_{\eta}$, the smallest singular value of $J_G$, and $\|G\|_{C^2(\Omega; \Omega)}$; see \eqref{eq:KL_compact_fullbound} for details.
 \end{theorem}

We now turn our attention to the case of unbounded domains with $\eta$ taken to be the standard Gaussian 
measure.  The proof of the following theorem is  presented in Section~\ref{sec:proof-of-KL-thm}.

\begin{theorem}\label{thm:kl}
Let $\eta$ be the standard Gaussian measure on $\R ^d$ and
 $F, G\colon\R^d \to \R^d$ be measurable maps.
Suppose the following conditions hold: 
\begin{enumerate}[label=(B\arabic*)]
\item $F,G \in H_\eta^2(\R^d;\R^d)$ and are $C^2$  $\eta$-a.e. 
\label{as:c1}
\item $G(\R^d) \subseteq F(\R^d)$.\label{as:imagefg}
\item There exists a constant $c> 0$ so that
$|\det (J_F)|, |\det (J_G)|\geq  c>0 $~~ $\eta$-a.e.\ in  $\R^d$. \label{as:det_fg}
\item There exists a constant $c_G > 0$ so that 
the smallest singular value of $J_G $ is bounded from below by $c_G$  $\eta$-a.e. \label{as:loweval_Jg}
\item $|F_i(x)|,|G_i(x)|, |\partial _{x_j} G_i(x) |, |\partial _{x_j, x_{\ell}} G_i(x)| $ grow at most polynomially in $x$ for every $1\leq i,j, \ell \leq d $. \label{as:tails}
\end{enumerate}
Then it holds that 
\begin{equation}\label{eq:kl_H1}
\KL( F_\sharp \eta \| G_\sharp \eta ) \leq C \|F- G\|_{H_\eta^1(\R ^d ;\R^d)} \, ,
\end{equation}
where the constant $C >0$ depends only on $d, c, c_G, \|\nabla {\rm det}J_g\|_{L^2}, \|F\|_{W^{1,2(d-1)}},$ and $\|G\|_{W^{1,2(d-1)}}$.
 \end{theorem}

Let us motivate the hypotheses in Theorems \ref{thm:kl_compact} and  \ref{thm:kl}
by giving an overview of their proofs. The starting point of both proofs is similar: we use the change of variables formula to express the ${\rm KL}$-integral as an integral with respect to the reference $\eta$. To do so, we need to ensure that both  maps are invertible and $C^1$, hence Assumptions \ref{as:c2_compact} and \ref{as:bijective_compact} in the compact case, and \ref{as:c1} and \ref{as:det_fg} in the unbounded case. As noted before, for the KL-integral to be finite, one has to avoid image mismatch between $F$ and $G$; hence assumptions \ref{as:bijective_compact} and \ref{as:imagefg}. Using the change of variables formula, we then separate the KL-integral into two parts (see e.g., \eqref{eq:Jg_i} and \eqref{eq:Jg_ii} in the proof): 
\begin{equation*}
{\rm KL}(F_{\sharp}\eta || G_{\sharp}\eta ) \leq C  \underbrace{\|J_G - J_G \circ (G^{-1} \circ F) \|_{L^2_{\eta}}}_{\rm I} + \underbrace{ \|J_G - J_F \|_{L^2_{\eta}}}_{\rm II}    \, .   
\end{equation*}
Term ${\rm II}$ is the more intuitive of the two: the PDFs of the pushforwards are related to the Jacobian of the maps via \eqref{eq:pushforward_density}, and so here we compare their integrated distance. Already here some technical difficulties of the unbounded case arise: local regularity is not enough to bound this term, and we need to explicitly require Sobolev regularity (Assumption \ref{as:c1}). Furthermore, $J_G$ and $J_F$ are determinants, whose entries are {\em products} of first derivatives. To bound these by the $H^1$ norm, we use a linear algebra lemma by Ipsen and Rehman \eqref{eq:ipsen}, which in turn requires us to ensure that $F, G \in W^{1,2(d-1)}$. We 
achieve this  using the tails hypothesis \ref{as:tails}.

To understand the intuition behind integral ${\rm I}$, first note that if $G^{-1}\circ F = {\rm Id}$, then term ${\rm I}=0$. Hence, this term  measures ``how far $G^{-1}$ is from $F^{-1}$.'' Here, in order to compare the inverses, we use Lagrange mean value theorem-type arguments in both settings (Lemma~\ref{lem:q_unbd}), but there is a major difficulty in the unbounded case: since invertibility on $\R^d$ does not guarantee that the determinants and singular values of $J_G, J_F$ do not go to $0$ as $|x|\to \infty$, we need to require that explicitly in the form of Assumptions \ref{as:det_fg} and \ref{as:loweval_Jg}. Moreover, we find that we need to also bound the $L^2$ norm of the second derivative of $G$, hence again the tails condition \ref{as:tails}. We comment that the second derivatives also appear in the case of maps from $\R ^d$ to $\R$ as a way to control the geometric behavior of level sets, see \cite{ditkowski2020density}, and that in principle they can be replaced by control over the Lipschitz (or even H{\"o}lder) constants of the first derivatives.

Finally, we comment that despite the relative freedom in choosing $\eta$ in the compact settings, there are still restrictions: Assumption \ref{as:eta_pos_compact} requires that $p_{\eta}$ is bounded from below on its support, and Assumption \ref{as:eta_Lip_compact} requires the reference to be Lipschitz, and in particular continuous on the compact domain $\Omega$, and hence bounded. The settings where $\eta$ is continuous on an {\em open} set but is unbounded either from above or from below (e.g., the Wigner semicircle distribution $p_{\eta}(y) \,  \propto \,  \sqrt{1-y^2}$) remain open, as they are not covered by our current analysis. 
% As noted, in this work we approach these settings using backward transport (see Section \ref{subsubsec:bakcward-KL}) in which the control of the tails is made easier.

\subsubsection{Backward transport}\label{subsubsec:bakcward-KL}

 We now turn our attention to the backward transport problem. 
 Our assumptions here are more relaxed in comparison to the forward transport problem. In particular, 
 we take  $\Omega = \R ^d$ and assume that $\nu$ has sub-Gaussian tails. 
 % , we assume that the tail of the target $\nu$ is asymptotically lighter than (or equal to) that of the standard Gaussian reference~$\eta$;} see Assumption~\ref{as:Gaussiantail_pullback} below. 
The proof of the following theorem is presented in Section~\ref{app:proof_kl_pullback}. 
\begin{theorem} \label{thm:kl_pullback}
Let $\eta$ be the standard Gaussian measure on $\R^d$. Suppose the following assumptions hold: 
\begin{enumerate}[label=(C\arabic*)]
\item $F,G \in H_\eta^1(\R^d;\R^d)$ and $W_\eta^{1,2(d-1)}(\R^d;\R^d)$ \label{as:cont_pullback}
\item $F, G$ are invertible and $F^{-1}, G^{-1}$ are both Borel-measurable.\footnote{Since $\eta$ is absolutely continuous with respect to the Lebesgue measure, being $\eta$-measurable and Borel-measurable are equivalent.}
\item There exists a constant $c> 0$ so that
$|\det (J_F)|, |\det (J_G)|\geq  c>0 $  $\eta$-a.e.\ in  $\R^d$. \label{as:det_pullback}
\item There exists a constant $c_F < \infty$ so that $p_{F^\sharp \eta}(x) \leq c_F p_{\eta}(x)$ for all $x \in \R^d$.
\label{as:Gaussiantail_pullback}
\end{enumerate}
Then it holds that 
\begin{equation} \label{eq:KL_pullback}
    \KL(F^\sharp\eta||G^\sharp\eta) \leq C\|F - G\|_{H_\eta^1(\R ^d ;\R^d)},
\end{equation}
where the constant $C >0$ depends on $d, c, c_F, \|F + G\|_{L^2_\eta},\|F\|_{W^{1,2(d-1)}_\eta},$ and $\|G\|_{W^{1,2(d-1)}_\eta}$.
\end{theorem}

We comment on the assumptions in Theorem \ref{thm:kl_pullback} and compare them to those in the pushforward analog, Theorem \ref{thm:kl}. %~\ref{as:cont_pullback} and~\ref{as:Gaussiantail_pullback}.
For the pushforward, Assumption~\ref{as:tails} on the asymptotic polynomial growth of the map and its first derivatives implies that the $\eta$-weighted Sobolev norms are finite. Thus, Assumption~\ref{as:tails} is a sufficient condition for the map to lie in the function spaces prescribed in Assumption~\ref{as:cont_pullback}. On the other hand, Assumption~\ref{as:Gaussiantail_pullback} implies that the target distribution is sub-Gaussian~\cite{vershynin2018high}; see e.g., \cite[Remark 3]
{baptista2020representation}. This condition is easier to interpret and verify, compared to the asymptotic polynomial growth of the maps in Assumption~\ref{as:tails}, by using various equivalent conditions for sub-Gaussian distributions.

The proof of Theorem~\ref{thm:kl_pullback} follows closely that of Theorem~\eqref{thm:kl}, where the pushforward is not by the maps $F$ and $G$, but rather by their inverses. This simplifies matters considerably, since the KL divergence between the pullback densities now does not involve the Jacobian of the {\em inverse} maps. Thus, less stringent assumptions are needed on the maps $F, G$ in Theorem \ref{thm:kl_pullback}, e.g.,  there is no requirement for uniform control over the 
singular values of their Jacobians. Instead, we only need to control the closeness of the maps and their Jacobians in expectation over the target measure $F^\sharp \mu$. Under Assumption~\ref{as:Gaussiantail_pullback}, we can equivalently control the closeness of these terms in expectation over the reference measure $\mu$, thereby yielding an upper bound that depends on the difference between $F,G$ integrated over the reference measure.

\section{Applications}\label{sec:applications}
We are now ready to apply the abstract framework of Section \ref{sec:main-results} and the stability results of Section~\ref{sec:stability}  to a wide variety of concrete examples. Whenever we discuss convergence, we think of a parameterized sequence of spaces \rev{$\widehat \mmT^n$} indexed by some parameter $n$ \rev{that characterizes the complexity/capacity of our approximation class}, and then solve the optimization problem
\begin{equation}\label{eq:abst_opt_theta}
\hnu^n = \hT_\sharp \eta \, , \qquad \hT \in 
\rev{
\argmin\limits_{S \in \widehat \mmT^n} D(S_\sharp \eta, \nu) \, .
}
\end{equation}
In Section~\ref{sec:l2conv} we consider the convergence of $\hnu^n$ in Wasserstein distances
when the \rev{ $\widehat \mmT^n$ } form a dense subset of $L^p$. Quantitative convergence rates are presented in Section~\ref{sec:measures-with-regular-densities} under more stringent conditions including  when the target and reference have regular densities.

\subsection{$\W_p$ convergence for target measures with finite variance}\label{sec:l2conv}
Following Theorem \ref{lem:wplp} for the Wasserstein distance, our first application takes place in very general settings, where the {\em regularity} and {\em approximation} results are robust, but relatively weak.
\begin{proposition}\label{prop:densexn}
Let $\Omega \subseteq \R ^d$ be a Borel set. Suppose $\nu$ and $\eta$ both have finite variance and that 
$\eta$ gives no mass to any $(d-1)$-dimensional $C^2$ manifold in $\R ^d$. 
\rev{
Consider a countable collection of subspaces  
$\widehat \mmT^n, n \ge 1$ 
such that 
\begin{equation}\label{eq:density_hmtn}
   \widehat \mmT^n \subset \widehat \mmT^{n+1}, 
   \quad \text{and} \quad 
    \overline{\lim\limits_{n\to\infty} \bigcup_{j=1}^n \widehat \mmT^n} = L^2_\eta(\R^d; \R ^d) \, .
\end{equation} 
}
Define $\hnu^n$ by \eqref{eq:abst_opt_theta} with $D=\W_p$ with $p\in [1,2]$. 
Then $\lim_{n\to \infty} \W_p(\hnu^n, \nu) =  0 \, .$
\end{proposition}

\begin{proof}
Let $\TT$ be the $\W_2$ OT map from $\eta$ to $\nu$,
which is known to exist following \cite[Thm.\ 1.22]{santambrogio2015optimal} (see also \cite{brenier1987decomposition}). Since $\nu = \TT_\sharp\eta$ and both measures have finite variance, $\TT\in L^2_\eta(\R ^d; \R ^d)$. Therefore, by the density condition \eqref{eq:density_hmtn}, 
\rev{ $\lim _{n\to \infty}{\rm dist}_{L^2_\eta}(\widehat \mmT^n, T^\dagger) =0$.}
Simultaneously, by Theorem \ref{lem:wplp} (stability of Wasserstein 
distances), and Theorem~\ref{thm:abstract-error} (our abstract framework),
we have for all $p\in [1,2]$ and all $n\in \mathbb{N}$ that
\rev{
$ \W_p(\hnu^n, \nu) \leq {\rm dist}_{L^2_{\eta}}(\widehat \mmT^n, \TT).$
}
Passing to the limit $n\to \infty$, the right hand side  vanishes  yielding the desired 
result.
\end{proof}
We emphasize once more that the fact that we choose $\TT$ to be the $\W_2$-OT map  is independent of the choice of $\W_p$ with $p\in [1,2]$  in \eqref{eq:abst_opt}, and any $L^2_\eta$ map $\TT$ with $\TT_{\sharp}\eta = \nu$ would have worked. 
% In particular, the OT map $\TT$ is not necessarily computed.
Density properties such as \eqref{eq:density_hmtn} are known for many families of functions. For example, using standard density results for polynomials, splines, and neural networks in $L^2$, we immediately get 
the following corollary.
\begin{corollary}
Consider $\nu, \eta \in \PP(\Omega)$  with finite variances. Then, in the following cases, 
\rev{ $\hnu^n$ } as defined in \eqref{eq:abst_opt_theta} converges to $\nu$ in $\W_p$ for all $p\in [1,2]$:
\rev{
\begin{enumerate}
    \item $\Omega =[0,1]^d$ or $\R ^d$ and $\widehat \mmT^n$ are polynomials of degree $\leq n$. \label{conv_dense:pol}
    \item $\Omega = [0,1]^d$ and $\widehat \mmT^n$ are $m$-th degree tensor product splines of a fixed degree $m\geq 0$ with $n^d$ knots (a tensor-product grid). \label{conv_dense:spline}
    \item $\Omega = [0,1]^d$ and $\widehat \mmT^n$ are feed-forward neural networks with at least one layer, $n$ weights, ReLU activation functions.
    \item $\Omega = \R^d$ and $\widehat \mmT^n$ are feed-forward neural networks with at least four layers, $n$ weights, and ReLU activation functions.\label{conv_dense:nn} %of width $\max\{ d \lfloor N^{1/d} \rfloor, N+1 \}$ and depth $L \ge 1$ with $N, L \le n$.  
\end{enumerate}
}
\end{corollary}
\begin{proof}
Each statement follows from $L^2$ density results of the form of \eqref{eq:density_hmtn} for the corresponding approximation class 
existing in the literature. For polynomials (Statement~\ref{conv_dense:pol}) see  \cite{canuto1982approximation, ernst2012convergence, xiu2010numerical}. For splines (Statement~\ref{conv_dense:spline}), this is a consequence of the density of continuous functions in $L^2(\Omega)$, and the density of piecewise  constant functions in $ C^0(\Omega)$. For
neural networks~(Statement \ref{conv_dense:nn}) with \(\Omega = [0,1]^d\), we use the approximation of H{\"o}lder continuous functions by neural networks \cite{shen2020deep}, and the density of H{\"o}lder continuous functions in $L^2(\Omega)$. 
When \(\Omega = \R^d\),  we prove density of four-layer ReLU networks in \(L^p_\eta (\R^d)\) in  Theorem~\ref{thm:nn_unbounded_approx} below.
\end{proof}

\begin{remark*}
The result in this section can be generalized to \eqref{eq:abst_opt_theta} with $D=\W_p$ with {\em any} $p> 1$. The regularity component ($\TT \in L^p$) is then given for measures with finite $p$-moment; see \cite[Thm.\ 1.17]{santambrogio2015optimal} and \cite{gangbo1996geometry}.
\end{remark*}

The following theorem shows that four-layer feed-forward neural networks with the ReLU activation are dense in \(L^p_\eta(\R^d; \R^m)\). 
This result is of independent interest since typical universal approximation results for neural networks are stated over 
compact domains while our result is stated on all of $\R^d$. While similar results have been shown for operator learning, see  \cite{bhattacharya2021model, lanthaler2022error, kovachki2021neural},
to the best of our knowledge, this is the first result for standard neural networks. The proof is given in Section~\ref{sec:nn_unbounded_approx}.

\begin{theorem}
    \label{thm:nn_unbounded_approx}
    Let \(\eta \in \mathbb{P}(\R^d)\) and suppose \(F \in L^p_\eta (\R^d; \R^m)\) for any \(1 \leq p < \infty\). Then, for any \(\epsilon > 0\) there exists a number \(n = n(\epsilon) \in \mathbb{N}\) and a four-layer ReLU neural network \(\widehat{F} : \R^d \to \R^m\)
    with \(n\) parameters such that $\|F - \widehat{F}\|_{L^p_\eta} < \epsilon.$
\end{theorem}

\subsection{Measures with regular densities}\label{sec:measures-with-regular-densities}
As noted, Proposition \ref{prop:densexn} only guarantees convergence, but does not provide rates of convergence. Indeed, since we take $\mT =\mmT = L^2 (\Omega)$ (recall our notation from Section~\ref{sec:main-results}) we only know that the true pushforward map $\TT$ is in the space $\mmT$ in which the stability and approximation theory take place. It is therefore natural that  we can only get convergence. 
Below we turn our attention to those cases where $\mmT$ is a proper subspace of $\mT$, in the sense that we require  more regularity of $\TT$. This type of assumption will lead to provable
convergence rates. Once again we separate our results to cases where $\Omega$ is bounded or 
unbounded. This is due to the stringent technical assumptions required on our maps 
in order to bound the KL divergence following Theorems \ref{thm:kl_compact} and \ref{thm:kl}.

\subsubsection{Compact domains}\label{sec:bounded-domains}
Consider absolutely continuous measures $\eta, \nu$ on $\Omega =~[-1,1]^d$ with  densities $p_{\eta}$ and $p_{\nu}$, respectively. Choose $\TT$ to be the $\W_2$-optimal map from $\eta$ to $\nu$.
Then, by
\cite{caffarelli1992regularity} (see also \cite[Thm.~2.2]{colombo2017lipschitz}, \cite[Ch.~12]{villani-OT}, and \cite{figalli2017monge}), we know that for any $k\geq 1$ 
\begin{equation}\label{eq:locReg}
 p_{\eta}, p_{\nu} \in C^k(\Omega) \quad \text{and} \quad   p_{\eta},p_{\nu}>0 \quad \text{implies}   \quad \TT\in C^{k+1}(\Omega ; \Omega) \, .
 \end{equation}
 Since the domain is compact, we have that $\TT$ is in the weighted Sobolev space $H^{k+1}_{\eta}(\Omega ; \Omega )$. 
 This constitutes our regularity result.
 
We now take \rev{$\widehat \mmT ^n$} to be the space of polynomials of 
degree $n$.
 Let us briefly recall the construction of tensor-product polynomial and their approximation theory; for details see ~\cite{canuto1982approximation, xiu2010numerical}. For concreteness, we will construct these approximations using the Legendre polynomials \cite{szego1939orthogonal}. Let $\mbf{m} = (m_1, \dots, m_d)$ be a multi-index with non-negative integers $m_i$.
Recall the 
multi-dimensional Legendre polynomials of 
degree $\mbf{m}$ \cite{cheng2003hermite}, $ p_{\mbf{m}}(x) \coloneqq \prod_{i=1}^d  p_{m_i}(x_i),$ for all  
$ x \in \R^d$
where each coordinate of $p_{m_i}:[-1,1]\to \R^d$ is the univariate Legendre polynomial of degree $m_i$ for each $1\leq i\leq d$. Let \rev{$\widehat \mmT^{n}$} denote the space of mappings from $\R^d \to \R^d$ where each 
component of the map belongs to the span of polynomials of degree $|\mbf m|_1 = n$. Choosing the Legendre polynomials as a basis, it can be written as
\begin{equation}\label{eq:LegTm}
    \rev{ \widehat \mmT^n = } \left\{ T \in L^2_\eta(\R^d; \R^d) \;  \Big| \; 
    T_{i} (x) =  \sum_{|\mbf m|_1 \le n} c_{\mbf m}^i  p_{\mbf m}(x), \qquad i =1, \dots, d \right\} \, .
\end{equation}
We now recall a result from approximation theory that gives a rate of convergence
for polynomial projections on Sobolev-regular functions.\footnote{The result also holds for polynomial interpolants at Gauss quadrature points, the so-called spectral collocation methods, but we do not pursue this direction further here.}
 \begin{proposition}[Canuto and Quarteroni \cite{canuto1982approximation}]\label{prop:CanutoQuarteroni}
 Suppose $F\in H^t (\Omega)$, let \rev{$\widehat \mmT^n$} be defined as in \eqref{eq:LegTm}, and 
let \rev{$\pi _n F \in \widehat \mmT^n$} be the $L^2$ projection of a function $F\in L^2(\Omega;\Omega)$ onto \rev{$\widehat \mmT^n$.} Then, for any $1\leq s \leq t$ there exists a constant $C=C(s, t)>0$ such that
 \begin{equation}\label{eq:canuto}
     \|F- \pi_n F\|_{H^{s}(\Omega; \Omega)} \leq C n^{-e(s,t)} \|F\|_{H^{t}(\Omega; \Omega)} \, ,
    \quad \text{where} \quad 
     e(s, t) = t +\frac12 -2s \, . 
 \end{equation}

 \end{proposition}
 
 Combining the above result with the stability results of Section \ref{sec:stability} 
and Theorem~\ref{thm:abstract-error} we can prove the following quantitative spectral error bounds 
for polynomial approximations of transport maps. The proof is presented in Section~\ref{proof:prop:application-compact-rates}.

\begin{proposition}\label{prop:application-compact-rates}
Let $\Omega = [-1,1]^d$ and consider $\eta, \nu \in \PP (\Omega)$ with strictly positive densities $p_{\eta},p_{\nu}\in C^k (\Omega)$ with $k \ge 1$. Let \rev{$\widehat \mmT^n$} denote the space of polynomials of degree $n$   defined in \eqref{eq:LegTm},
and let $\hnu^n$ be as in \eqref{eq:abst_opt_theta}. Then it holds that:
\begin{enumerate}
    \item If $D=\W_p$ for $p\in [1,2]$, then $\W_p (\hnu^n , \nu) \le C n^{-(k+\frac32) } \, .$
    \item If $D = \MMD_\kappa$ with $\kappa(x, y) = \exp( - \gamma^2 |x - y|^2)$ then 
    $\MMD_\kappa (\hnu^n, \nu) \le C  n^{-(k+\frac32) }$.
    \item If  $D={\rm KL}$, let $k > \frac52 +2\lfloor \frac{d}{2} \rfloor$, let $\eta$ be the uniform measure on $\Omega$, and choose\footnote{Since ${\rm KL}$ is not symmetric, we emphasize the {\em order} in which the ${\rm KL}$ divergence is taken in \eqref{eq:abst_opt_kl}, unlike the Wasserstein distance and MMD.}
    \begin{equation}\label{eq:abst_opt_kl}
 \hT = 
 \rev{ 
 \argmin\limits_{S \in \widehat \mmT^n} 
 }
 \,{\rm KL}(S_\sharp \eta || \nu) \, .
\end{equation}
Then ${\rm KL}(\hnu^n \|\nu) \leq C n^{-k+\frac12+2\lfloor \frac{d}{2}\rfloor}$ for sufficiently large $n\gg 1$.
\end{enumerate}
The constant $C> 0$ is different in each item but it is independent of $n$. 
In particular, for $p_{\eta},p_{\nu}\in C^{\infty}(\Omega; \Omega)$, all three choices of $D$ yield faster-than-polynomial convergence.
\end{proposition}

A few comments are in order regarding item (3) of Proposition \ref{prop:application-compact-rates}. The proof slightly deviates from the general strategy outlined in Section \ref{sec:main-results}: the map to which we compare $\TT$ and apply the stability result (Theorem \ref{thm:kl_compact}) is not its best polynomial approximation $\pi_n \TT$, but rather a renormalized version of $\pi_n \TT$. The reason for this change is the possibility of image mismatch, i.e., that $\pi_n \TT (\Omega) $ might be larger than $\Omega$, which in turn  would yield an infinite KL divergence.
The renormalization of $\pi_n \TT$ leads to a looser upper bound than we would have intuitively anticipated. Rather than the $H^1$ error of polynomial approximation (see \eqref{eq:canuto}), it is the $L^{\infty}$ function approximation ($\|\pi_n \TT - \TT \|_{\infty}$) which dominates the KL divergence. Since our mechanism for obtaining pointwise convergence from the standard Sobolev-space convergence theory \eqref{eq:canuto} is Sobolev embedding, the error rate deteriorates with the dimension. 

It is for a similar reason that we require higher regularity in higher dimensions, i.e., that $\TT \in C^{5/2+2\lfloor d/2 \rfloor }$. To apply the KL-divergence stability result, Theorem \ref{thm:kl_compact}, we need to ensure invertibility of the polynomial $\pi_n \TT$, and we do so by guaranteeing that it is close enough to the invertible map $\TT$ in  $C^1$. Again we need to pass through Sobolev embedding, and hence the $d$-dependence of the regularity. 

How should we understand these two instances of $d$-dependence? First, since this dependence reflects an analytic passage rather than a property of the algorithm, it might not be sharp. Improvement of these passages is an interesting question. Second, we could get rid of this $d$ dependence by replacing the polynomial approximation class \rev{$ \widehat \mmT ^n$} by local approximation methods, e.g., splines, or radial basis functions. 
Then, $L^{\infty}$ and $C^1$ convergence are guaranteed for a fixed regularity of $\TT$.
The key advantage of polynomial approximation is that the convergence rate $n^{-k+1/2 +2\lfloor d/2 \rfloor}$ improves with the regularity $k$. In particular, for a fixed dimension and a smooth target, we get faster-than-polynomial convergence.

\subsubsection{Unbounded domains}\label{sec:unbounded-domains}
We now turn our attention to the case where $\Omega = \R^d$ and obtain analogous results to 
those of Section~\ref{sec:bounded-domains}.   The main challenge is in controlling the tails: even though the local regularity statement \eqref{eq:locReg} still holds, we can only use the approximation theory results if $\TT$ is regular in a Sobolev sense, which on unbounded domains requires control of the tail. Below, we will impose strong assumptions 
on the reference and target measures $\eta, \nu$ to be able to guarantee Sobolev regularity of the map $\TT$, which we take 
to be the $\W_2$-optimal map once more.

We fix the reference measure 
 $\eta = N(0, I)$ and take the target
 $\nu \in \PP(\R^d)$ with Lebesgue density $p_\nu = \exp( - w(x)) \, dx$. We further assume that 
$w: \R^d \to \R$ is strongly 
convex, i.e., there exists a constant $K >0$ so that $D^2 w  (x)\succeq K \cdot {\rm Id}$ for all $x\in \R ^d$,
where $D^2w$ is the Hessian of $w$ and $\succeq$ is an inequality of positive definite matrices; equivalently, we assume that $\nu$ is and strongly log-concave. 
It was shown in \cite[Thm.~3.1 and~Rem.~3.1]{kolesnikov2013sobolev} that  $\TT$
 satisfies
\begin{subequations}
\begin{equation}\label{kolesnikov-result-1}
    \int_{\R^d}  \| J_{T^\dagger}(x) \|_{\rm HS}^2 \; \eta(\dd x) \le \frac{I_\eta}{K} \, ,
    \end{equation}
    as well as
    \begin{equation}\label{kolesnikov-result-2}
    \int_{\R^d} \left[ \sum_{j=1}^d |  J_{\partial_{x_j} T^\dagger}(x) |_{\rm HS}^2 \right]^{1/2} \eta(\dd x) \le \frac{1}{2 \sqrt K} \int_{\R^d} | x|^2 \; \eta(\dd x),
\end{equation}
\end{subequations}
where $| \cdot |_{\rm HS}$ is the Hilbert-Schmidt matrix norm and  $I_\eta \coloneqq \int |\nabla w(x)|^2 d\eta$ is the Fisher information of $\eta$. For the standard Gaussian reference measure, $I_\eta = \int |x|^2 \eta(dx)$. By Minkowski's integral inequality we have
\begin{equation}
    \left[ \int_{\R^d}  \left( \sum_{j=1}^d |  J_{\partial_{x_j} T^\dagger}(x) |_{\rm HS} \right)^2 \eta(\dd x) \right]^{1/2}
    \le 
    \int_{\R^d} \left[ \sum_{j=1}^d |  J_{\partial_{x_j} T^\dagger}(x) |_{\rm HS}^2 \right]^{1/2} \eta(\dd x),
\end{equation}
which together with \eqref{kolesnikov-result-2} implies 
that $\TT$ belongs to the Sobolev class $H^2_\eta(\R^d; \R^d)$.

Analogously to Section~\ref{sec:bounded-domains}, we proceed to obtain error rates for $\hnu^{n}$ obtained by solving optimization problems of the form 
\eqref{eq:abst_opt_theta} on unbounded domains. The construction of \rev{$\widehat \mmT ^n$} in the unbounded case could be done using Hermite polynomials. However, since on unbounded domains polynomials are unbounded, it is more convenient (and conventional) to use the exponentially weighted {\em Hermite functions}.
Recall the 
multi-dimensional Hermite functions of 
degree ${\mbf m}$ \cite{cheng2003hermite}: 
\begin{equation*}
    h_{\mbf{m}}(x) \coloneqq \prod_{i=1}^d  h_{m_i}(x_i), \qquad 
    h_{m_i}(x_i) \coloneqq    (-1)^{|m_i|} \exp(x_i^2) \partial_{m_i} \exp( - x_i^2), \qquad \forall x \in \R^d \, ,
\end{equation*}
and let \rev{$\widehat \mmT^n$} denote the space of mappings from $\R^d \to \R^d$ where each 
component of the map belongs to the span of Hermite functions of degree $|\mbf m|_1 = n$, i.e. 
\begin{equation}\label{eq:hermite-poly}
    \rev{ \widehat \mmT^n =} \left\{ T \in L^2_\eta(\R^d; \R^d) \;  \Big| \; 
    T_{i} (x) =  \sum_{|\mbf m| \le n} c_{\mbf m}^i  h_{\mbf m}(x), \qquad i =1, \dots, d \right\} \, .
\end{equation}
We recall the following error bound for projections of 
Sobolev functions onto the span of Hermite functions. 
\begin{proposition}[Xu and Guo \cite{cheng2003hermite}]\label{prop:Xu-Guo}
Suppose $F \in H^t_\eta (\R^d)$, let \rev{$\widehat \mmT^n$} be defined as in \eqref{eq:hermite-poly}, 
and let \rev{$\pi_n F \in \widehat \mmT^n$} be the $L^2_\eta(\R^d)$ projection of $F$
onto \rev{$\widehat \mmT^n$}. Then, for $0 \le s \le t$ there exists a constant $C = C(s,t) >0$ 
such that
\begin{equation*}
    \| F - \pi_n F \|_{H^s_\eta(\R^d)} \le C n^{\frac{s-t}{2}} \| F\|_{H^t_\eta(\R^d)}.
\end{equation*}

\end{proposition}
We now present the analogue of Proposition~\ref{prop:application-compact-rates} 
for the Wasserstein and MMD distances. 
\begin{proposition}\label{prop:application-Rd-rates-wp-mmd}
Take $\Omega= \R^d$, $\eta = N(0, I)$ and let $\nu$ be strongly log-concave.
Let \rev{$\widehat \mmT^n$} denote the space of Hermite functions of degree $n$ as in~\eqref{eq:hermite-poly}, and define $\hnu^n$ as in~\eqref{eq:abst_opt_theta}. 
Then it holds that: 
\begin{enumerate}
    \item If $D= \W_p$ for $p\in [1,2]$, then $\W_p( \hnu^n, \nu) \le C n^{-1}$,
    \item If $D = \MMD_\kappa$ with 
    $\kappa(x,y) = \exp( - \gamma^2 | x- y|^2)$ then $\MMD_\kappa(\hnu^n, \nu) 
    \le C n^{-1}$.
\end{enumerate}
The constant $C> 0$ is different in each item but is independent of $n$.
\end{proposition} 

\begin{proof}
The proof is identical to that of Proposition~\ref{prop:application-compact-rates}
except that Proposition~\ref{prop:Xu-Guo} is used instead of 
 Proposition~\ref{prop:CanutoQuarteroni}. 
\end{proof}
\begin{remark}\label{remark:wishful-thinking}
The reason 
that we cannot obtain rates that are better than $n^{-1}$ is that under 
the hypotheses of the theorem, we can only establish that $\TT \in H^2_\eta(\R^d;\R^d)$,
as discussed earlier. Had we known that $\TT \in H^{2k}$ for some $k\geq 1$, we would have obtained an $n^{-k}$ rate. To the best of our knowledge, however, no such high-order Sobolev regularity results are known for $\W_2$ optimal maps on unbounded domains.
\end{remark}

One can also obtain bounds for the KL divergence 
for the forward transport problem. However, 
in the current framework, such an analysis involves assumptions on the approximating maps that seem impractical to verify,  mirroring our discussion of KL stability in Section \ref{subsec:KLdivergence}. Instead we dedicate the next section to 
the applied analysis of the backward transport for \emph{triangular} maps on unbounded domains by minimizing the KL 
divergence. Our analysis for triangular maps relies on a specialized stability result, Theorem \ref{thm:abstract-pullback}, which is analogous to the general Theorem~\ref{thm:abstract-error} for the backward transport problem. However, due to the triangularity assumption, it is much stronger; see Section~\ref{app:abstract-error-backward-transport} for details.

\subsection{Backward transport with triangular maps} \label{subsec:pullback_KR}

A popular class of transport maps is the family of \emph{triangular maps}. 
 A map   
 $T\colon \mathbb{R}^d \rightarrow \mathbb{R}^d$ is said to be lower-triangular (henceforth referred to simply 
 as a triangular map)
 if it has the form
\begin{equation} \label{eq:lower_triangular}
T(x) = \begin{bmatrix*}[l] T_1(x_1) \\ T_2(x_1,x_2) \\ \vdots \\ T_d(x_1,\dots,x_d) \end{bmatrix*} \, .
\end{equation}
Perhaps the most well-known example of such maps is the  Knothe--Rosenblatt (KR) rearrangement~\cite{bogachev2005triangular, santambrogio2015optimal, villani-OT}, 
% a transport map between the reference and the target measure, with an 
which has an explicit construction based on a sequence of one-dimensional transport problems. 

As we shall see in what follows, it is common and useful to work with triangular maps that are strictly monotone (henceforth we shall use the term \textit{monotone}, for brevity). In the triangular setting, 
the monotonicity of the map reduces to the requirement that 
each component $T_i$ of the map is monotone with respect to its last input variable, i.e., that $x_i \mapsto T_i(x_1,\dots,x_i)$ is monotone increasing for each $(x_1,\dots,x_{i-1}) \in \mathbb{R}^{i-1}$ and $i=1,\dots,d$. If the measures $\nu$ and $\eta$ are  absolutely continuous, then
there exists a unique triangular map $T$ satisfying $T_\sharp \eta = \nu$ and this map is precisely the KR rearrangement
\cite{bogachev2005triangular}.\footnote{Interestingly, the KR map also corresponds to the limit of a sequence of maps in $L^2_\mu(\R ^d; \R ^d)$ that are optimal in the sense of the Monge problem with respect to a weighted $L^2$ cost that penalizes movement more strongly in direction $x_i$ than in $x_{i+1}$, for all $i =1,\dots,d-1$; see \cite{carlier2010knothe}.}

%The uniqueness of the KR rearrangement makes it a convenient map for numerical approximation. 

The uniqueness of the KR map is desirable in practice as it can be leveraged to formulate 
algorithms with unique solutions \cite{marzouk2016sampling}. However, working with triangular maps has other practical
advantages which has made them a popular choice in the architecture design of NFs as well 
\cite{kobyzev2020normalizing, papamakarios2021normalizing}.
Most notably, triangular maps are convenient to invert by forward substitution. Furthermore, their  Jacobian determinants can be computed efficiently, making them well-suited for backward transport or density estimation problems. 
To see that, observe that the Jacobian of a triangular  $T$ is given by the product of the partial derivatives of each map component $T_i$ with respect to its last input variable: ${\rm det}(J_T) =~\prod_{i=1}^{d} \frac{\partial  T_i}{\partial_{x_i}}$. Thus, 
we only need the components of $T$ to be differentiable with respect to their last input variables 
to be able to apply the change-of-variables formula
\eqref{eq:pullback_density}.

To this end, let $\eta_{\leq i} = \prod_{\ell=1}^i \eta_{\ell}$ denote the first $i$ marginals of the 
measure $\eta$ and define the function spaces 
\begin{equation}\label{eq:Vi_def}
V_{\eta_{\leq i}}(\mathbb{R}^i;\R) \coloneqq 
\left\{ v \in L^2_{\eta_\le i}(\R^i; \R) \: \Big| \: \| v \|_{V_{\eta_{\leq i}}(\mathbb{R}^i;\R)} < +\infty \right\}, 
\end{equation}
where
\begin{equation*}
    \| v \|_{V_{\eta_{\leq i}}(\mathbb{R}^i;\R)} 
    \coloneqq \left( \| v \|^2_{L^2_{\eta_\le i}(\R^i; \R)} 
    + \int_{\R} |\partial_i v|^2 \dd \eta_{\le i}   \right)^{1/2}.
\end{equation*}
% \begin{equation}\label{eq:Vi_def}
% V_{\eta_{\leq i}}(\mathbb{R}^i;\R) \coloneqq L_{\eta_1}^2(\R) \otimes \dots \otimes L_{\eta_{i-1}}^2(\R) \otimes H_{\eta_i}^1(\R) \, ,
% \end{equation}
% equipped with the norm 
% \begin{equation*}
%     \| f \|_{V_{\eta_{\leq i}}(\mathbb{R}^i;\R)} 
%     \coloneqq \| f \|_{L^2_{\eta_\le i}(\R^i; \R)} 
%     + \left( \int_{\R} |\partial_i f|^2 \dd \eta_i   \right)^{1/2}.
% \end{equation*}
% instead of $H_{\eta_{\leq i}}^1(\R^i;\R) \subset V_{\eta_{\leq i}}(\mathbb{R}^i;\R)$,   
We can then prove  the analogue of Theorem~\ref{thm:kl_pullback} using the KL 
divergence to formulate backward triangular transport problems.
We highlight the simplicity of the conditions of this theorem compared to our previous results 
involving the KL divergence, Theorems \ref{thm:kl_compact}, \ref{thm:kl}, and \ref{thm:kl_pullback}.
% the following result shows we require less regularity than the maps in  to establish a similar stability result for the KL divergence, but with triangular maps. In fact, we only need each component $T_i \colon \R^i \mapsto \R$ to be in the
%standard Gaussian measure $\eta$.
The proof is given in Section~\ref{app:proof_kl_pullback_triangular}. 
%where the map must be in $H^1$.
%In this section we provide an analogous stability result to Theorem~\ref{} to lower triangular transport maps by providing a minimal set of conditions needed to bound the KL divergence. %Next, we hope to provide a class of distributions that satisfy the assumptions.

\begin{theorem} \label{thm:kl_pullback_KRmap}
Let $\eta$ be the standard Gaussian measure on $\R^d$ and $F, G\colon \R^d\to \R^d$ be invertible lower-triangular maps. Suppose the following conditions hold: 
\begin{enumerate}[label=(D\arabic*)]
\item $F_i,G_i \in V_{\eta_{\leq i}}(\R^i;\R)$ for $i=1,\dots,d$. \label{as:pullback_KR}
%\item $|\partial_i f_i(x)|, |\partial_i g_i (x)| > \tau > 0$ for every $x\in \R^d$. \label{as:diag_lbd}
\item There exists a constant $c> 0$ so that
$\partial_{x_i} F_i, \partial_{x_i} G_i > c >0 $ $\eta$-a.e.\ in  $\R^d$. \label{as:det_fg_pullback_KR}
\item There exists a constant $c_F < \infty$ so that $p_{F^\sharp \eta}(x) \leq c_F p_{\eta}(x)$ for all $x \in \R^d$.
\label{as:Gaussiantail_pullback_KR}
\end{enumerate}
Then it holds that 
\begin{equation} \label{eq:stability_KLpullback_triangular}
    \KL(F^\sharp\eta||G^\sharp\eta) \leq C \sum_{i=1}^{d} \|F_i - G_i\|_{V_{\eta_{\leq i}}(\R^i;\R)},
\end{equation}
where the constant $C >0$ depends only on $c_F, c$ and $ \|F + G\|_{L^2_\eta}$.
\end{theorem}

%\subsection{Convergence of triangular maps} \label{sec:Application_triangularmaps}
We are now ready to present an analogous result to those in Section~\ref{sec:unbounded-domains} for the backward transport problem 
 using the KL divergence. The main technical challenge here is that, to ensure that 
 our class of approximate transport maps are invertible, we need to guarantee that they are monotone. To do so, we consider the representation used in \cite{baptista2020representation,wehenkel2019unconstrained} that defines each monotone function as the integral of a positive function. 
%Thus, we Let $f_i \colon \R^{i-1} \rightarrow \R$ be a smooth function with respect to $x_i$ and let $r \colon \R \rightarrow \R_{>0}$ be bijective,  strictly positive, Lipschitz continuous and whose inverse is Lipschitz everywhere, except at the origin. As an example, we can take $r$ to be soft-plus function $r(z) = \log(1+\exp(z))$.
% Now define a triangular map $T$ with components
%\begin{equation} \label{eq:monotone_maps}
%    T_i(x_{1:i}) 
%    = 
%    f_i(x_{1:i-1},0) + \int_0^{x_i} r \big(\partial_{i} f_i(x_{1:i-1},t) \big) \dd t
%    =:
%    \mathcal{R}_i(f_i)(x_{1:i})
%\end{equation}

\begin{definition}[Integrated monotone parameterizations] \label{defn:monotone_triangular_maps} Let $r \colon \R \rightarrow \R_{>0}$ be a bijective, strictly positive, Lipschitz continuous function whose inverse is Lipschitz everywhere, except possibly at the origin. Let $f:\R ^d\to \R ^d$ be smooth. Then a monotone triangular map $T$ with components of the form
\begin{equation} \label{eq:monotone_maps}
    T_i(x_{1:i}) 
    = 
    f_i(x_{1:i-1},0) + \int_0^{x_i} r \big(\partial_{i} f_i(x_{1:i-1},t) \big) \dd t
    =:
    \mathcal{R}_i(f_i)(x_{1:i})
\end{equation}
is said to have an integrated monotone parameterization in terms of $f$.
Here we used the shorthand notation $x_{1:j} \equiv (x_1, \dots, x_j)$ for an index $1 \le j \le d$.
%By design, $T_i$ is smooth with respect to its last variable, $x_i$. %the operator $\mathcal{R}_i$ transforms a smooth $f_i$ into a function $T_i$ that is monotone with respect to its $i$-th input, and that is also smooth with respect to $x_i$.
\end{definition}

The above definition yields a family of parameterizations based on the choice of the 
functions $r$ and $f$. We also note that the $T_i$ are readily smooth with respect to $x_{1:i-1}$
following the assumptions on $f$.
The conditions required for $r$ were shown in~\cite{baptista2020representation} to yield a stability of the representation with respect to the functions $f_i$. Examples of $r$ satisfying the conditions in Definition~\ref{defn:monotone_triangular_maps} are the soft-plus function $r(z) = \log(1+\exp(z))$ and the shifted exponential linear unit $$r(z) = \rev{\left\{\begin{array}{ll} e^z - 1 & z \leq 0 \\ z & z > 0 \end{array} \right.}.$$%$r(z) = \{e^z, z < 0, 1, z > 0\}$.s

We can now turn to proving an error estimate for a sampling method based on monotone triangular maps:

\begin{proposition} \label{prop:KRconvergence} Suppose there exist constants $0 < c \leq C < \infty$ such that the target density satisfies $c p_{\eta}(x) \leq p_{\nu}(x) \leq C p_{\eta}(x)$ $\eta$-a.e.\thinspace in $\mathbb{R}^d$. Let $T^\dagger$ be the KR rearrangement of the form in~\eqref{eq:lower_triangular} such that $(T^\dagger)_\sharp \nu = \eta$ and assume further that $\mathcal{R}_{i}^{-1}(T^\dagger_i) \in H^2_\eta$
for all $i$, with $\mathcal{R}_i$ as in~\eqref{eq:monotone_maps}. Let \rev{$\widehat \mmT^n$} denote the space of maps $T$ 
that have an integrated monotone parameterization of the form 
$T_i = \mathcal{R}_i(f_i)$  with $f_i$ being a Hermite function of degree $n$ as in~\eqref{eq:hermite-poly}. Then, for the approximate measure $\widehat\nu^n$ defined as
\begin{equation}\label{eq:abst_opt_theta_pullback}
\widehat\nu^n = \hT^\sharp \eta \, , \qquad \hT \in 
\rev{
\argmin\limits_{T \in \widehat \mmT^n} 
}
\, \KL(\nu||T^\sharp\eta) \, , 
\end{equation}
it holds that $$\KL(\nu||\widehat\nu^n) \leq Kn^{-{1/2}} \, ,$$ for some $n$-independent constant $K < \infty$.
\end{proposition}

\begin{proof} To show this result, we borrow results on the %tail properties and 
smoothness of triangular maps from~\cite{baptista2020representation}, which are included in Appendix~\ref{app:lemma_monotonemaps} for completeness.

We begin by verifying the assumptions of Theorem~\ref{thm:kl_pullback_KRmap}, the stability result for the KL divergence to pullback by triangular maps, when taking one map to be the Knothe--Rosenblatt rearrangement $T^\dagger$ that pulls back $\eta$ to $\nu$, i.e., that satisfies $p_{(T^\dagger)^\sharp \eta} = p_{\nu}$. From the lower bound on the probability density function $p_{\nu}>cp_{\eta}$,
it follows from Lemma~\ref{lem:prop9} that each component of $T^\dagger$ has affine asymptotic behavior with a constant partial derivative, i.e., $T_i^\dagger(x_{1:i-1},x_i) = \mathcal{O}(x_i)$ and $\partial_{x_i} T_i(x_{1:i-1},x_i) = \mathcal{O}(1)$ as $|x_i| \rightarrow \infty$. Hence, $T_i^\dagger \in V_{\eta_{\leq i}}(\R^i;\R)$ for $i=1,\dots,d$. From the same proposition we also have $\partial_{x_i} T_i \geq c^{-} > 0$ for some $c^{-} > 0$ for all $(x_{1},\dots,x_{i-1}) \in \R^{i-1}$, and hence $T^\dagger$ satisfies Assumption~\ref{as:det_fg_pullback_KR}.
Lastly, by the assumption on the target density we have $p_{(T^\dagger)^\sharp\eta}(x) = p_{\nu}(x) \leq C\eta(x)$ for some constant $C$, hence satisfying Assumption~\ref{as:Gaussiantail_pullback_KR}.

Now we consider the class of monotone transport maps defined in Definition~\ref{defn:monotone_triangular_maps}. %We borrow results on the smoothness of these maps from~\cite{baptista2020representation}, which are included in Appendix~\ref{app:lemma_monotonemaps} for completeness. 
From Lemma~\ref{lem:prop3}, we have $T_i = \mathcal{R}_i(f_i) \in V_{\eta_{\leq i}}$ for any $f_{i} \in V_{\eta_{\leq i}}$ and Lipschitz function~$r$. Furthermore, $\partial_{x_i} \mathcal{R}_i(f_i)(x_{1:i}) > 0$ if $\essinf \partial_{x_i} f_{i} > -\infty$. These conditions %on the essential infimum 
are satisfied 
by taking our approximate functions to be \rev{$f_i \in \widehat \mmT^n$.} In particular, we can take $f_i \coloneqq \pi_n \mathcal{R}_i^{-1}(T_i^\dagger)$ by projecting the (possibly) non-monotone function that yields the $i$-th component of the KR rearrangement (via the operator $\mathcal{R}^{-1}$) onto \rev{$\widehat \mmT^n$.} We denote the resulting triangular and monotone map by $\mathcal{R}(f)$. Applying Theorem~\ref{thm:kl_pullback_KRmap} and Theorem~\ref{thm:abstract-pullback} (our abstract framework) with $T^\dagger$ and using the sub-optimality of the map $\mathcal{R}(f)$, we have $$\KL(\nu||\widehat\nu^n) %= \KL((T^\dagger)^\sharp\eta||(\widehat{T}^n)^\sharp\eta) 
\leq \KL((T^\dagger)^\sharp\eta||\mathcal{R}(f)^\sharp\eta) \leq C'\sum_{i=1}^{d} \| T_i^\dagger - \mathcal{R}_i(f_i) \|_{V_{\eta_{\leq i}}},$$ %\quad \forall f_i \in \hmT^n \,$$ 
where $C' = C(\|T^\dagger + \mathcal{R}(f)\|_{L^2_\eta}/2 + 1/c^{-})$. 

To apply the approximation theory results, we relate the convergence of $f_i$ and the monotone maps. For  $r$ with Lipschitz constant $L$, from Lemma~\ref{lem:prop5} %Proposition 5 in~\cite{baptista2020representation}
we have that $f_i^\dagger \coloneqq \mathcal{R}_i^{-1}(T_i^\dagger) \in V_{\eta_{\leq i}}$ and that the inverse operator $\mathcal{R}_i^{-1}$ is Lipschitz, i.e., there exists some constant $C_L < \infty$ (depending on the lower bound of the gradient of $T_i$ and on the choice of $r$)
such that 
\begin{equation}\label{eq:Rstability}
\|\mathcal{R}_i(f_i^\dagger) - \mathcal{R}_i(f_i)\|_{V_{\eta_{\leq i}}} \leq C_L\|f_i^\dagger - f_i\|_{V_{\eta_{\leq  i}}} \, .    
\end{equation}
Thus, we have $\KL(\nu||\widehat\nu^n) \leq C'C_L \sum_{i=1}^{d} \|f_i^\dagger - f_i\|_{V_{\eta_{\leq i}}} \leq C' C_L \sum_{i=1}^d \|f_i^\dagger - f_i\|_{H_{\eta_{\leq i}}^1} \lesssim n^{-1/2}$ by Proposition~\ref{prop:Xu-Guo} for $f_i^\dagger \in H^{2}_\eta$, where we recall that $f_i$ is the polynomial approximation of $f_i^\dagger$.
% for a finite depending on $C,C_L$ and $d$. 
Lastly, for $f_i^n \rightarrow f_i^\dagger$ we have that $\mathcal{R}_i(f_i^n) \rightarrow T_i^\dagger$ in $V_{\eta_{\le i}}$
as $n \rightarrow \infty$ for each $i = 1,\dots,d$. Thus, the constant $C'$ approaches $C(\|2T^\dagger\|_{L^2_\eta}/2 + 1/c^{-})$ and is bounded for all $n$.
\end{proof}
\begin{remark}\rev{Equation \eqref{eq:Rstability} can be viewed as a Lipschitz property of the integrated representation of monotone maps (see Definition~\ref{defn:monotone_triangular_maps}). This result and the techniques involved can also be used to proof the analogs of Proposition~\ref{prop:application-Rd-rates-wp-mmd} for the Wasserstein-$p$ and MMD distances with parameterized monotone and triangular maps. We do not pursue this direction in this work.}
\end{remark}

Here we make a similar comment to Remark \ref{remark:wishful-thinking}: if we strengthen the hypotheses of Proposition \ref{prop:KRconvergence} such that $\mathcal{R}_{i}^{-1}(T^\dagger_i) \in H^t_\eta$ for some $t>2$, then we get a faster rate of convergence, i.e., ${\rm KL}(\nu ||\hat{\nu}^n)\lesssim n^{(1-t)/2}$, as an immediate consequence of the relevant approximation theory, Proposition \ref{prop:Xu-Guo}. Hence, improvements in the study of Sobolev regularity of KR maps, see e.g., the work of~\cite{kolesnikov2014continuity}, combined with a higher-order stability analysis of the integrated parameterization in Definition~\ref{defn:monotone_triangular_maps} (analogous to Lemma \ref{lem:prop3})  will yield improvements in our numerical analysis of sampling methods.

\section{Numerical experiments} \label{sec:numerics}

In this section we numerically validate the approximation results obtained in Section~\ref{sec:applications} %and \ref{sec:pullback-stability-results} 
for various realizations of the abstract algorithm of Section~\ref{sec:main-results}. Sections~\ref{subsec:wasserstein_compact}--\ref{subsec:wasserstein_compact_results} investigate the algorithm that minimizes the Wasserstein distance, while Section~\ref{subsec:KLdivergence} investigates the Kullback-Leibler divergence between pullback measures. The code to reproduce the following numerical results is available at: \url{https://github.com/baptistar/TransportMapApproximation}.
%In subsection~\ref{subsec:wasserstein_compact} we minimize the Wasserstein distance for one-distribution distributions by taking advantage of its closed-form expression. In subsection~\ref{subsec:KLdivergence} we minimize the KL divergence using a widely-used parametric class of monotone and triangular transport maps. 

\subsection{Minimizing Wasserstein distance: methodology} \label{subsec:wasserstein_compact}
%As in Section \ref{sec:main-results}, let $\widehat\nu$ be the measure whose density is given by $\widehat{T}_\sharp\eta$ for some approximate map $\widehat{T}$.

%\textcolor{red}{Ricardo, I re-edited this subsection, following our discussion yesterday. Let me know if this is fair or if I mischarachterized anything. }
In this set of experiments, we let $D=W_p$, the Wasserstein-$p$ distance on $\Omega =[-1,1]$. We first comment on how the computation of \eqref{eq:abst_opt} is done in these settings.

Recall that for any two probability measures $\mu, \phi$ on $\R$, we have that $$W_p(\mu, \phi) = \|F_{\mu}^{-1}-F_{\phi}^{-1}\|_{L^p} \, ,$$ where $F_{\mu}^{-1}$ is the inverse CDF (quantile) of $\mu$, and similarly for $F_{\phi}^{-1}$; see e.g., \cite{santambrogio2015optimal}. Hence, our algorithm \eqref{eq:abst_opt} can be written as
\begin{equation}\label{eq:opt_wass_pre}
\widehat{T} \in 
\rev{
\argmin_{S\in \widehat \mmT} }\,
W_p^p(\nu, S_{\sharp}\eta ) 
= \rev{\argmin_{S\in \widehat \mmT}} \,
\int_0^1 |F_{\nu} ^{-1}(y) - F_{S_{\sharp}\eta}^{-1}(y)|^p \, dy \,  .
\end{equation}
Unfortunately, this is still not a feasible optimization problem: it requires that in each iteration we compute the inverse CDF of $S_{\sharp}\eta$. For non-invertible maps $S$, the CDF is not available in closed form and so it must be estimated empirically by means of Monte Carlo sampling from $\eta$, which makes the estimation expensive. Furthermore, the derivative (with respect to $S$) in the optimization then becomes nontrivial as well.

Rather, we make an assumption, that we will henceforward justify heuristically as well as empirically: assume that $S$ is a monotone increasing map ($S' > 0$). Why is this assumption useful? A monotone map that pushes forward $\eta$ to $S_{\sharp} \eta$ is necessarily an OT map between these measures with respect to the Wasserstein-$p$ metric for $p\geq 1$ (the unique map for $p>1$) \cite{santambrogio2015optimal}. Hence, $S$ can be written as $S=F_{S_{\sharp}\eta}^{-1} \circ F_\eta(x)$, and so by a change of variables $F_{S_{\sharp}\eta }^{-1}(y) = S(F_\eta^{-1}(y))$. Hence, under the monotonicity assumption, \eqref{eq:opt_wass_pre} transforms into
\begin{align*}
&{\rm If}~~ S'>0 \, \qquad \forall\,\rev{S \in \widehat \mmT} \, , \qquad {\rm then} \\
\rev{\argmin\limits_{S\in \widehat \mmT}} \int_0^1 &|F_{\nu} ^{-1}(y) - F_{S_{\sharp}\eta}^{-1}(y)|^p \, dy  = \rev{\argmin\limits_{S\in \widehat \mmT}} \int_0^1 |F_{\nu} ^{-1}(y) - S\left( F_{\eta}^{-1}(y) \right)|^p \, dy \, . \numberthis \label{eq:opt_wass_practical}
\end{align*}
This last expression is now more amenable to numerical optimization. Indeed, $F_{\eta}^{-1}$ and $F_{\nu}^{-1}$ are fixed throughout the optimization, and can be computed  at the beginning of the procedure. 

There is another gain to be made: for $p=2$, \eqref{eq:opt_wass_practical} can be reformulated as a least-squares problem with respect to $S$. For the linear expansion, $S(x) = \sum_{i=1}^{n} \alpha_i h_{m_i}(x)$ in terms of basis functions $h_{m_i} \colon \R \rightarrow \R$, the solution of~\eqref{eq:opt_wass_practical} for the vector 
$\alpha \in \R^{n}$ 
is then given in closed form by 
\begin{equation}\label{eq:wass_minim_ls}
    \alpha = A^{-1}b \, ,
\end{equation}
where the elements of $A \in \mathbb{R}^{n \times n}$ are the $L_\eta^2$ inner product of the expansion elements, i.e., 
$A_{ij} = \int_{0}^{1} h_{m_i}\left (F_{\eta}^{-1}(y) \right ) h_{m_j} \left (F_{\eta}^{-1}(y) \right ) \dd y$, and the entries of $b \in \mathbb{R}^{n}$ are the projections of those elements on the quantile function for $\nu$, i.e., $b_j = \int_{0}^{1} h_{m_i} \left (F_{\eta}^{-1}(y) \right ) F_{\nu}^{-1}(y) \dd y$. 

How do we justify the monotonicity assumption? First, we note that $\TT$, the true map for which $\TT_{\sharp}\eta = \nu$, can always be chosen to be monotone (i.e., by choosing $\TT = F_{\nu}^{-1} \circ F_\eta(x)$). Moreover, when using a iterative method (for $p \neq 2$) to minimize~\eqref{eq:opt_wass_practical}, we initialize our search at the identity map $S(x)=x$, which is monotone. We also expect that for a sequence of spaces \rev{$\widehat \mmT ^{n}$} which becomes dense in $L^2$ as $n\to \infty$, the polynomial approximations of $\TT$ will become monotone as well. While this is not a proof, it is the heuristic that leads us to minimize \eqref{eq:opt_wass_practical} in our experiments.
Finally, the empirical evidence we present below will also show that throughout most experiments the learned map $\widehat{T}$ in fact remains monotone.

\subsection{Minimizing the Wasserstein distance: numerical results} \label{subsec:wasserstein_compact_results}
As noted above, in our first experiment we choose the reference $\eta$ to be the uniform measure on $[-1,1]$ and let the target be
\begin{equation}\label{eq:nuk_compact1d}
\nu_k \equiv  (T_{k})_{\sharp} \eta \, , \qquad T_k(x) \equiv  x^{2k}\text{sign}(x) \, , \quad k\in \mathbb{N} \, .
\end{equation} 
The motivation behind this particular choice of $T_k$ and $\nu_k$ is that $\TT=T_k \in C^{2k} \setminus C^{2k+1}$ for each $k\in \mathbb{N}$, hence allowing us to test the sharpness of the  polynomial rates in Proposition \ref{prop:application-compact-rates}. For each order $k$, we seek an approximate map $\widehat{T}^n$ in the span of Legendre polynomials up to maximum degree $n\in \mathbb{N}$; see \eqref{eq:LegTm}.
 We minimize the Wasserstein-$2$ distance to find $\widehat{T}^{n}$ by discretizing the integrals in~\eqref{eq:wass_minim_ls} %~\eqref{eq:opt_wass_practical} 
 using $10^4$ Clenshaw-Curtis quadrature points and computing the coefficients of the linear expansion in closed form.
 
 Figure~\ref{fig:W2_compact_poly}
 plots the $W_2$ objective in~\eqref{eq:opt_wass_practical} and the $L^2$ error in the map with an increasing polynomial degree~$n$. We observe that the convergence rates are faster for higher degrees of regularity $k$, and closely match the theoretical convergence rates derived in Proposition~\ref{prop:application-compact-rates}. Furthermore, the $W_2$ and $L^2$ convergence rates closely match for each~$k$. For easier comparison, Tables~\ref{tab:conv_errors_k1} and \ref{tab:conv_errors_k3} present the $W_2$ and $L^2$ errors, as well as the predicted value for $k=1,3$, respectively. To verify that the resulting map is monotone, and hence is converging to the (unique) monotone transport map, we measure the following probability of the estimated map being monotone
 \begin{equation} \label{eq:ProbMonotone}
\mathbb{P}_{x \sim \eta, x' \sim \eta}\left[ \left \langle \widehat{T}^n(x) - \widehat{T}^n(x'), x - x'  \right \rangle > 0\right],
\end{equation}
using $10^4$ pairs of i.i.d.\thinspace test points drawn from the product reference measure $\eta \otimes \eta$. The values for the probability are included in Tables~\ref{tab:conv_errors_k1} and~\ref{tab:conv_errors_k3}. We notice that $\widehat{T}^n$ converges to the monotone map and thus the objective in~\eqref{eq:opt_wass_practical} converges to the exact Wasserstein distance, even though the estimated maps were not \textit{a priori} restricted to be monotone. Examples of the approximate maps resulting from our experiments (approximating \eqref{eq:nuk_compact1d} by minimizing~$W_2$) are 
presented in Figure~\ref{fig:Wp_map_Legendre}. %Note that the numerically computed maps are all monotone, though they were not restricted to be monotone a-priori.

% \textcolor{red}{Ricardo, I would add - (i) a table of numerically computed fits at each $k$ vs. the theoretically predicted one. (ii) some measure of the monotonicity of the different $S$'s during the optimization, e.g., take a particular setting (fixed $k$ and $n$) and show that $S$ was monotonic throughout the optimization proccess. }
\begin{figure}[!ht]
  \centering
  \begin{subfigure}[b]{0.48\textwidth}
    \includegraphics[width=\textwidth]{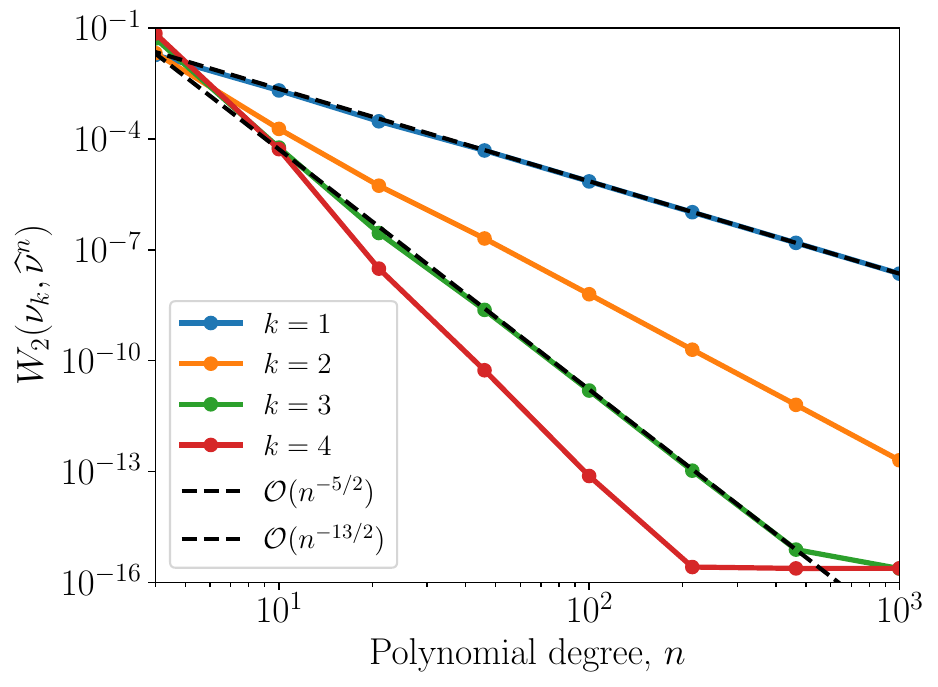}
    \caption{}
  \end{subfigure}
  \begin{subfigure}[b]{0.48\textwidth}
    \includegraphics[width=\textwidth]{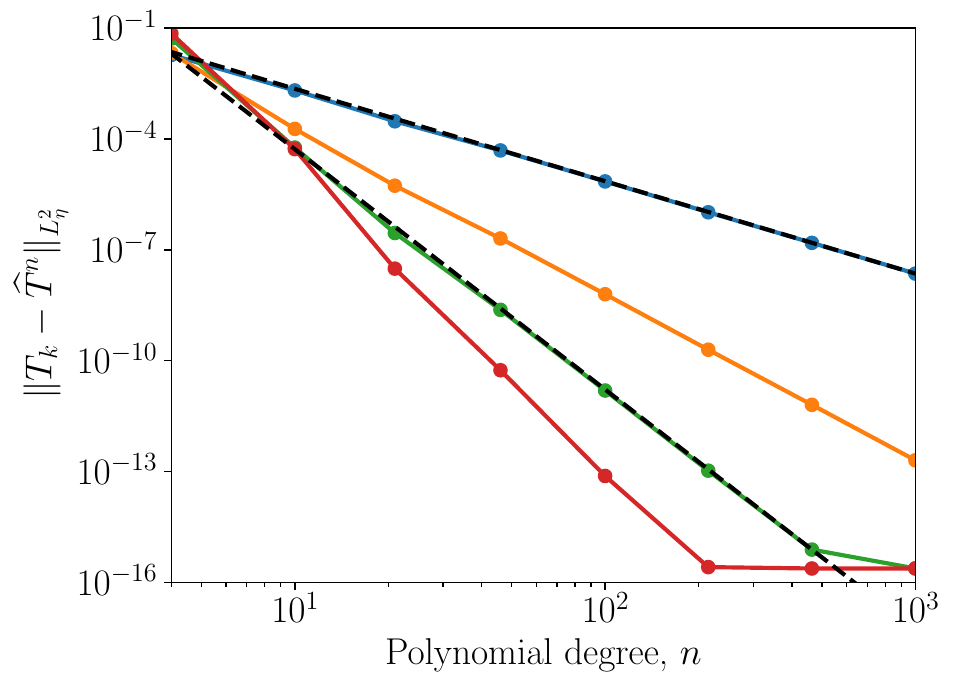}
    \caption{}
  \end{subfigure}
  \vspace{-0.5em}
  \caption{Convergence results for approximating $\nu_k$ in~\eqref{eq:nuk_compact1d} using Legendre polynomials of degree $n$ and the optimization problem \eqref{eq:opt_wass_practical} with $p=2$ for (a) $W_2 (\nu_k, \widehat{\nu}^n)$, (b) $\|T_k - \widehat{T}^n \|_{L^2}$.} %(both $\hat{\nu}^n$ and $\widehat{T}^n$ depend on the target $\nu_k$).}
  \label{fig:W2_compact_poly}
\end{figure}

\begin{table}
\centering
\scriptsize
\begin{tabular}{|c c c c c c c c|}% c c c|} 
 \hline
 Degree $n$ & $1$ & $2$ & $4$ & $10$ & $21$ & $46$ & $100$ %& $215$ & $464$ & $1000$ 
 \\ \hline
$W_2(\nu,\widehat\nu^n)$ & $ 1.12\times 10^{-1}$ & $ 1.12\times 10^{-1}$ & $ 1.86\times 10^{-2}$ & $ 2.04\times 10^{-3}$ & $ 2.98\times 10^{-4}$ & $ 4.84\times 10^{-5}$ & $ 7.05\times 10^{-6}$ 
%& $ 1.03\times 10^{6}$ & $ 1.53\times 10^{7}$ & $ 2.25\times 10^{8}$ 
\\
$\|T - \widehat T^n \|_{L^2_\eta}$ & $ 1.12\times 10^{-1}$ & $ 1.12\times 10^{-1}$ & $ 1.86\times 10^{-2}$ & $ 2.05\times 10^{-3}$ & $ 3.00\times 10^{-4}$ & $ 4.75\times 10^{-5}$ & $ 6.94\times 10^{-6}$
%& $ 1.02\times 10^{6}$ & $ 1.58\times 10^{7}$ & $ 2.47\times 10^{8}$ \\
\\
$O(n^{-5/2})$ & $ 7.05\times 10^{-1}$ & $ 1.25\times 10^{-1}$ & $ 2.20\times 10^{-2}$ & $ 2.23\times 10^{-3}$ & $ 3.49\times 10^{-4}$ & $ 4.91\times 10^{-5}$ & $ 7.05\times 10^{-6}$ %& $ 1.03\times 10^{6}$ & $ 1.51\times 10^{7}$ & $ 2.22\times 10^{8}$
\\
$\mathbb{P}[\text{Mon}]$ & $1.00$ & $ 1.00$ & $ 1.00$ & $ 1.00$ & $ 1.00$ & $ 1.00$ & $ 1.00$ 
\\ \hline
\end{tabular}
\caption{The Wasserstein-2 distance, $L^2$ error in the map, a reference $\mathcal{O}(n^{-5/2})$ rate, and the estimated map's monotonicity \eqref{eq:ProbMonotone} for a target distribution with $k = 1$. The scaling for the reference rate is set to match the $W_2$ distance at $n=100$.~\label{tab:conv_errors_k1}}
\end{table}

\begin{table}
\centering
\scriptsize
\begin{tabular}{|c c c c c c c c|}% c c c|} 
 \hline
 Degree $n$ & $1$ & $2$ & $4$ & $10$ & $21$ & $46$ & $100$ %& %$215$ & $464$ & $1000$ 
 \\ \hline
$W_2(\nu,\widehat\nu^n)$ & $ 1.73\times 10^{-1}$ & $ 1.73\times 10^{-1}$ & $ 5.20\times 10^{-2}$ & $ 5.80\times 10^{-5}$ & $ 2.83\times 10^{-7}$ & $ 2.36\times 10^{-9}$ & $ 1.55\times 10^{-11}$ 
%$ 1.05\times 10^{-13}$ & $ 7.70\times 10^{-16}$ & $ 2.36\times 10^{-16}$ 
\\
$\|T - \widehat T^n \|_{L^2_\eta}$ & $ 1.74\times 10^{-1}$ & $ 1.74\times 10^{-1}$ & $ 5.22\times 10^{-2}$ & $ 5.85\times 10^{-5}$ & $ 2.85\times 10^{-7}$ & $ 2.35\times 10^{-9}$ & $ 1.54\times 10^{-11}$ %& $ 1.05\times 10^{-13}$ & $ 7.71\times 10^{-16}$ & $ 2.42\times 10^{-16}$ & 
\\
$O(n^{-13/2})$ & $ 1.55\times 10^{2}$ & $ 1.71\times 10^{0}$ & $ 1.89\times 10^{-2}$ & $ 4.90\times 10^{-5}$ & $ 3.94\times 10^{-7}$ & $ 2.41\times 10^{-9}$ & $ 1.55\times 10^{-11}$ 
%& $ 1.05\times 10^{-13}$ & $ 1.53\times 10^{-15}$ & $ 2.24\times 10^{-17}$ & \\
\\
$\mathbb{P}[\text{Mon}]$ & $ 1.00$ & $ 1.00$ & $ 0.77$ & $ 1.00$ & $ 1.00$ & $ 1.00$ & $ 1.00$ 
\\ \hline
\end{tabular}
\caption{The Wasserstein-2 distance, $L^2$ error in the map, a reference $\mathcal{O}(n^{-13/2})$ rate, and the estimated map's monotonicity  \eqref{eq:ProbMonotone} for a target distribution with $k = 3$ \label{tab:conv_errors_k3}. The scaling for the reference rate is set to match the $W_2$ distance at $n=100$.}
\end{table}

In the next set of experiments, we choose $\eta$ to be the standard Gaussian measure $\mathcal{N}(0,1)$ on $\mathbb{R}$ and let $\nu$ be the one-dimensional Gumbel distribution, which is supported over the entire real line. The density of $\nu$ is given by
\begin{equation}\label{eq:gumble}
%p_{\nu, {\rm gumble}}(x) 
p_{\nu}(x) = \frac{1}{\beta} e^{-\left((x - \mu)\beta^{-1} + e^{-(x - \mu)\beta^{-1}} \right)} \, ,
\end{equation} where we choose the parameters $\mu = 1$ and $\beta = 2$. 
%\textcolor{red}{What did you choose $\beta$ to be?}
% defined by the pushforward of a map $F$ through $F_\sharp\eta$. whose domain is $\mathbb{R}$. 
For $p \in \{1,2\}$, we minimize the Wasserstein-$p$ distance in~\eqref{eq:opt_wass_practical} to approximate the target measure %as $\widehat{\nu}^n = (\widehat{T}^n)_\sharp \eta$ 
by the pushforward of a map $\widehat{T}^n$ that is parameterized as a linear expansion of Hermite functions up to degree $n \in \mathbb{N}$; see Section \ref{sec:unbounded-domains}. For $p = 1$, we optimize with respect to the coefficients in the Hermite expansion using an iterative BFGS optimization algorithm~\cite{Nocedal}. For $p = 2$, we use the closed-form expression in~\eqref{eq:wass_minim_ls}. For both $p=1,2$, we use $10^4$ Clenshaw-Curtis quadrature points to evaluate the objective and compute the unknown coefficients. 

Figure~\ref{fig:Wp_gumbel_poly} plots the $W_p$ objective based on the monotonicty assumption in~\eqref{eq:opt_wass_practical} , (labeled `Monotone $W_p$'), an empirical estimate of the Wasserstein distance  (labeled `Empirical $W_p$') that is computed using $10^7$ test points, and the $L^2$ error in between the estimated map $\widehat{T}^n$ and the optimal monotone map $T^\dagger$. Unlike the compact domain setting above, we observe an exponential (or faster than polynomial) decay rate with~$n$. In particular, the dashed lines in Figure~\ref{fig:Wp_gumbel_poly} illustrate a close agreement between the observed convergence rates for the Wasserstein distances and the exponential curves $\exp(-0.5n)$ and $\exp(-0.3n)$ for $p=1$ and $p=2$, respectively.
We also observe that the decay rate for the Wasserstein $W_p$ distance and $L^p$ error in the map are close, indicating that our stability results for Wasserstein 
distances are tight. Here again, we note that the empirical $W_1$ saturates around $5\times 10^{-3}$, due to the use of finitely many samples from both measures.

Figure~\ref{fig:Wp_map_Hermite} presents an example of the resulting approximate maps for this problem. %are %
%presented in Figure~\ref{fig:Wp_map_Hermite}. 
We observe that the estimated maps are converging to the 
true monotone map, although they remain non-monotone in the tails. The estimated maps are monotone in the region of high probability of the Gaussian reference measure $\eta$, however, and the region of non-monotonicity shrinks as $n$ increases. In fact, the probability of being monotone as measured using~\eqref{eq:ProbMonotone} is greater than $0.97$ for $n \geq 3$ and $1.00$ for $n \geq 9$ for both $p = 1,2$. As a result, the objective in~\eqref{eq:opt_wass_practical} approaches the exact Wasserstein distance, which is validated by the empirical estimates of $W_p(\nu,\widehat\nu^n)$ in Figure~\ref{fig:Wp_gumbel_poly}. We note that the empirical estimation requires an increasing number of samples to accurately compute small Wasserstein distances, and hence the estimate becomes inaccurate for large $n$.

\begin{figure}[!ht]
  \centering
  \begin{subfigure}[b]{0.48\textwidth}
    \includegraphics[width=\textwidth]{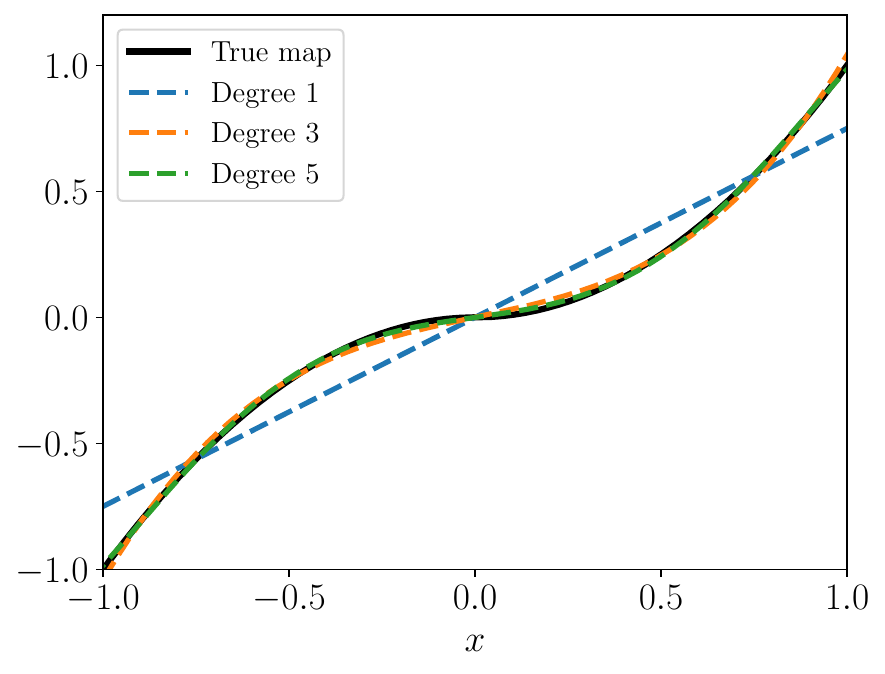}
    \caption{\label{fig:Wp_map_Legendre}}
  \end{subfigure}
  \begin{subfigure}[b]{0.48\textwidth}
    \includegraphics[width=0.95\textwidth]{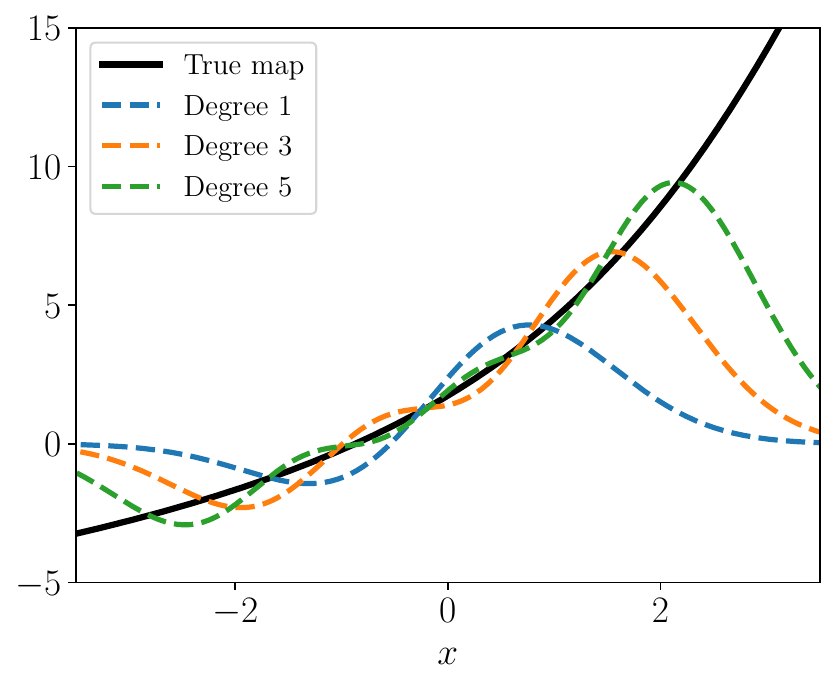}
    \caption{\label{fig:Wp_map_Hermite}}
  \end{subfigure}
  \vspace{-0.5em}
  \caption{Approximation of the map $T$ by minimizing the Wasserstein-2 distance for (a) the pushforward measure $(T_1)_\sharp\eta$ in~\eqref{eq:nuk_compact1d} using Legendre polynomials on a compact domain and (b) the Gumbel distribution in~\eqref{eq:gumble} using Hermite functions on the entire real line. \label{fig:Wp_map}
  %\textcolor{red}{Ricardo, a few questions (1) why approximate the Gumbel with Legendre? Clearly not the right way to go (as opposed to Hermite). Maybe we just need the Hermite subfigure. (2) Are these for the $W_2$ or $W_1$ convergence? (3) Axis are missing.}
  }
\end{figure}

\subsection{Minimizing KL divergence} \label{subsec:KLdivergence}

The next set of numerical experiments demonstrates our theoretical results for the backward transport problem, presented in Section \ref{subsec:pullback_KR}. We examine the convergence of monotone maps that seek to pull back the standard Gaussian $\eta = N(0,1)$ 
to the Gumbel distribution $\nu$, as defined in~\eqref{eq:gumble} with parameters $\mu = 0$ and $\beta = 1$. 
To compute this transport map we consider $N = 10^4$ empirical 
samples\footnote{We recall that our theoretical results all concern continuous densities (intuitively, $N\to \infty)$. The number of samples $N$ used here is thus chosen large enough to keep finite-sample effects relatively small.} 
$\{x^j\}_{j=1}^N$ drawn i.i.d.\thinspace from $\nu$. We form the empirical measure $\widecheck{\nu}_N = \frac{1}{N}\sum_{j=1}^N \delta_{x^j}$ and solve the optimization problem 
\begin{align}
    \widehat{T}^{n}  \in 
    \rev{
    \argmin_{S \in \widehat \mmT^{n}}
    }\,
    \KL(\widehat{\nu}_N||S^\sharp \eta) &= \rev{\argmin_{S \in \widehat \mmT^{n} } }
    \,\frac{1}{N} \sum_{j=1}^N -\log p_{S^\sharp\eta}(x^j) \nonumber \\
    &= 
    \rev{
    \argmin_{S \in \widehat \mmT^{n} }
    }
    \,\frac{1}{N} \sum_{j=1}^N |S(x^j)|^2 - \log|S'(x^j)|, \label{eq:KLminimization_GaussianRef_obj}
\end{align}
where we take \rev{$\widehat \mmT^{n}$} to be the space of monotone maps in Definition~\ref{defn:monotone_triangular_maps} that are transformations of non-monotone functions $f$, and where the class of non-monotone maps $f$ is the space of Hermite functions of degree $n$, for $n \in \{1,\dots,10\}$. Our goal is to verify the convergence rate in Proposition~\ref{prop:KRconvergence}.
\begin{figure}[!ht]
  \centering
  \begin{subfigure}[b]{0.48\textwidth}
    \includegraphics[width=\textwidth]{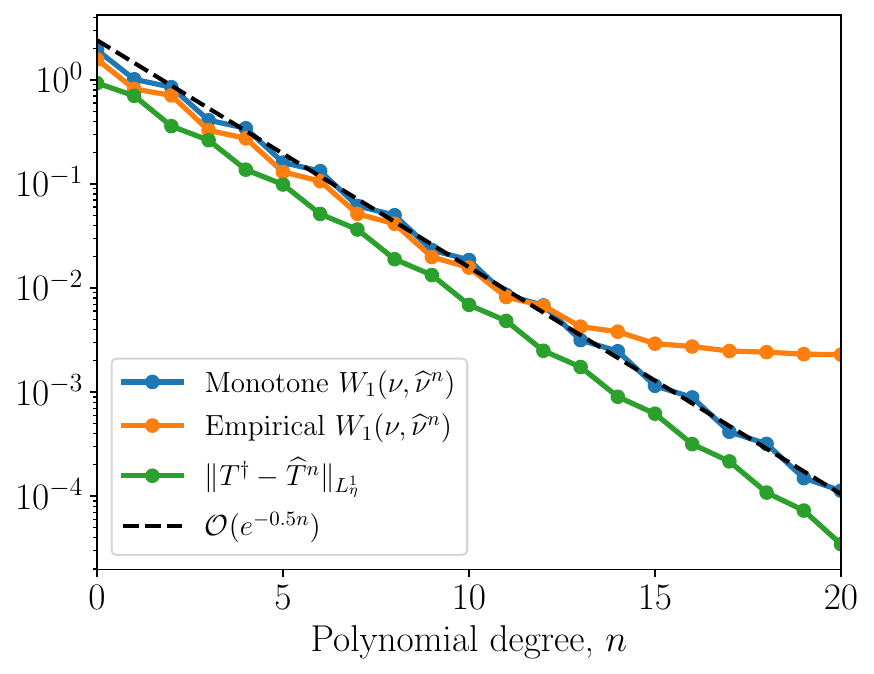}
    \caption{}
  \end{subfigure}
  \begin{subfigure}[b]{0.48\textwidth}
    \includegraphics[width=\textwidth]{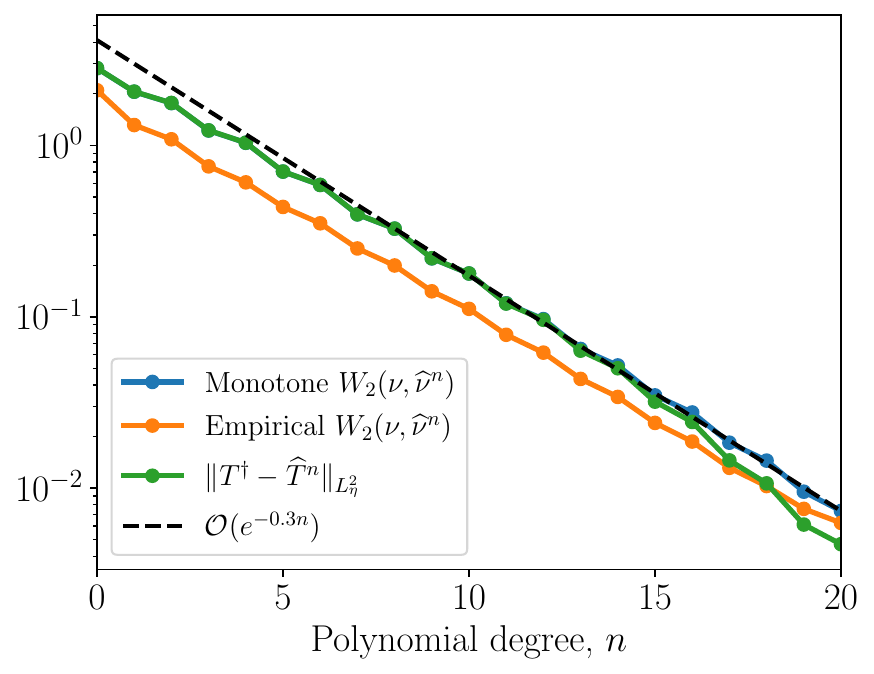}
    \caption{}
  \end{subfigure}
  \vspace{-0.5em}
  \caption{Convergence of the pushforward measure to the Gumbel distribution in terms of Wasserstein objective in~\eqref{eq:opt_wass_practical}, the empirical Wasserstein distance, and the $L^p$ error in the approximate map $\widehat{T}^n$ when \textbf{(a)} solving \eqref{eq:opt_wass_practical} for $p=1$, and \textbf{(b)} using \eqref{eq:wass_minim_ls} for~$p=2$. The dashed lines illustrate exponential convergence rates based on empirical fits to the computed Wasserstein distances. 
  % Wasserstein distances %based on a exponential curve fit to the computed distances. 
  %a close agreement between the convergence rate for the Wasserstein distances and the exponential curves $\exp(-0.5n)$ and $\exp(-0.3n)$ for $p=1$ and $p=2$, respectively.  
  \label{fig:Wp_gumbel_poly}}
\end{figure}

To compute the error of the estimated transport map, we rely on the fact that the monotone transport map pulling back the Gaussian reference to the Gumbel distribution is unique 
and can be identified analytically as $T^\dagger(x) = F_\eta^{-1} \circ F_\nu(x)$. 
 Figure~\ref{fig:tmapprox-gumbel_KL} presents the convergence of $\KL(\nu || (\widehat{T}^n)^{\sharp} \eta)$ and $\|T^\dagger - \widehat{T}^n\|_{L_\eta^2}$ as a function of the polynomial degree $n$. The KL divergence and $L^2_\eta$ error are computed using an independent test set of $10^5$ i.i.d.\thinspace samples drawn
 from the Gumbel distribution. %, as well as pullback samples obtained from our numerically 
% estimated map.  
Overall, we observe a faster rate of decay of the KL divergence than the $L^2_\eta$ error (and hence also the $H^1_\eta$ error) between the computed map $\widehat{T}^n$ and the 
true monotone map $T^\dagger$. It will be interesting in future work to study if this discrepancy is due to
lack of sharpness in our theoretical results. Moreover, as in Section~\ref{subsec:wasserstein_compact_results} when minimizing the Wasserstein distance, the KL divergence decays nearly exponentially fast with the degree $n$ (or at the very least, faster than polynomially). The dashed lines in Figure~\ref{fig:tmapprox-gumbel_KL} demonstrate that the Wasserstein distance and $L^2$ error closely match the exponential convergence rates $\exp(-0.2n)$ and $\exp(-0.1n)$, respectively.

Why do we see a faster-than-polynomial rate of convergence, whereas Proposition \ref{prop:KRconvergence} predicts only an $n^{-1/2}$ rate? In general, the regularity theory for optimal transport/KR maps for smooth distributions only guarantees that $\TT \in H^1 \cap C^{\infty}$. Since the approximation theory of Hermite functions relies on (global) Sobolev regularity, we can only guarantee the $n^{-1/2}$ rate (see the proof of Proposition~\ref{prop:KRconvergence} for details). For this particular experiment, however, we know that the transport map from the Gumbel distribution to the Gaussian has very ``light'' tails, and is therefore in $H^s$ for all $s\geq 0$. Hence, we immediately get a convergence rate faster than $n^{-s}$ for all $s\geq 0$. We discuss this issue in more detail in Remark~\ref{remark:wishful-thinking}.
 %since our   analysis throughout the article deals with discretizations of transport maps and ignores  the effect of finite sample size.

Figure~\ref{fig:tmapprox-gumbel_L2mappiweight} shows the true and approximate transport maps found by solving~\eqref{eq:KLminimization_GaussianRef_obj} with increasing polynomial degree $n$. In comparison to Figure~\ref{fig:Wp_map_Hermite}, where we sought the map pushing forward $\eta$ to $\nu$, here we seek the inverse map. In addition, unlike the direct parameterization of the map using Hermite functions, the parameterization in Definition~\ref{defn:monotone_triangular_maps} guarantees that the approximate maps are globally monotone and thus invertible. This feature lets us directly use these maps to estimate the density of $\nu$ using the change-of-variables formula, as described in Section~\ref{subsec:KL-error-bounds}.
 
%Figure \ref{fig:map-approximation-gumbel}(b) presents the $L^2_{\pi}$ error between the numerically estimated map $\hat{S}$  and the 
%analytic pullback map. While the error rates agree, the $L^2_{\pi}$ errors in absolute terms are slightly smaller then the
% the KL of the measures in Panel (a). It is not clear whether this discrepancy is due to
%lack of sharpness in our theoretical results, or just an artifact of the numerical 
% the setup of the problem.

\begin{figure}[!htb]
    \centering
    \begin{subfigure}[b]{0.48\textwidth}
        \centering
        \includegraphics[width=\linewidth]{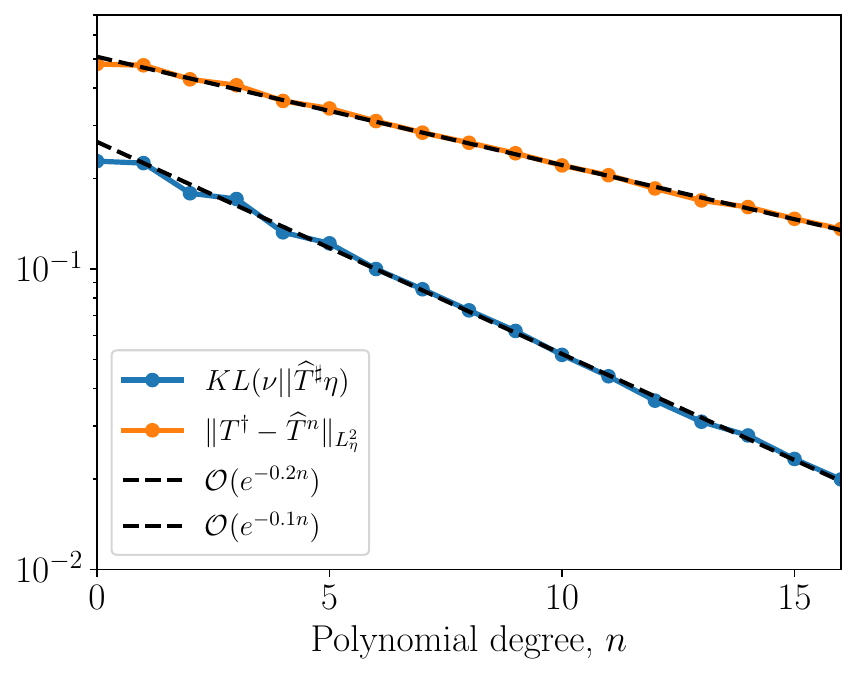}
        \caption{\label{fig:tmapprox-gumbel_KL}}
    \end{subfigure}%
    \hfill
    % \begin{subfigure}[b]{0.31\textwidth}
    %     \centering
    %     \includegraphics[width=\linewidth]{figures/KLminimization/W1_distance.eps}
    %     \caption{\label{fig:tmapprox-gumbel_W1}}
    % \end{subfigure}
    % \hfill
    \begin{subfigure}[b]{0.48\textwidth}
        \centering
        \includegraphics[width=0.98\linewidth]{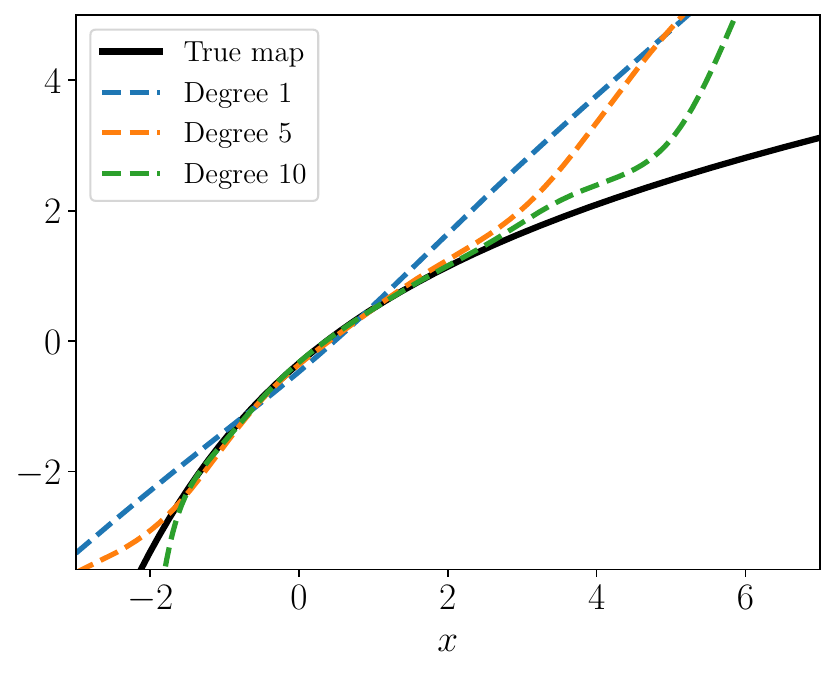}
        \caption{\label{fig:tmapprox-gumbel_L2mappiweight}}
        %\caption{$L_2$ error in the map weighted by the target $\pi$ }
    \end{subfigure}
    \caption{(a) Convergence of the
    pullback of a Gaussian reference $\eta$
    %pull-forward measure 
    to the Gumbel distribution in terms of KL divergence and the $L^2$ error in the approximate map $\widehat{T}^n$ when solving~\eqref{eq:KLminimization_GaussianRef_obj}. (b) The true ($T^\dagger$) and approximate ($\widehat{T}^n$) maps found with the monotone parameterization in Definition~\ref{defn:monotone_triangular_maps}, using Hermite functions of increasing polynomial degree.}
    %\caption{Convergence in (a) KL divergence and (b) weighted $L_\pi^2$ error in the map for increasing sample size and maximum polynomial degree \label{fig:map-approximation-gumbel}}
\end{figure}

\section{Conclusions}\label{sec:conclusions}

We have proposed a general framework to obtain a priori error bounds for the approximation of probability distributions via transport, with respect to various metrics and divergences. 
Our main result, Theorem~\ref{thm:abstract-error}, provides a strategy to obtain error rates between a target measure and its approximation via 
a numerically constructed transport map. Our strategy combines
the stability analysis of statistical divergences with regularity results for a ground 
truth map and off-the-shelf approximation rates for high-dimensional functions. We highlight that stability is often the question that requires new development, while regularity and approximation rates can be addressed in many cases using existing results in the literature. 
To this end, we have presented new stability results for the Wasserstein, MMD, and KL divergences, since these are some of the most popular choices in practice. Our numerical experiments demonstrate the sharpness of our analysis and investigate its 
validity in more general settings, beyond our theoretical assumptions. 

Overall, our theoretical results take a step towards understanding the approximation error of transport-based sampling and density estimation algorithms. At the same time, the present analysis suggests an extensive list of open questions for future research:

\begin{itemize}

    \item Developing stability results for other 
    families of divergences that are popular in practice, for example, the 
    $f$-divergences of \cite{birrell2022f}, the Jensen-Shannon entropy \cite{menendez1997jensen}, or even 
    functionals such as the evidence lower bound (ELBO) \cite{blei2017variational} and the entropy-regularized optimal transport cost \cite{jordan1998variational, pal2019difference, peyre2019computational}.

   \item Combining the approximation theory framework of this paper with statistical consistency and sample complexity results, in settings where maps are estimated from empirical data (without knowledge of the true underlying density).
   \rev{Here either the target measure $\nu$ is given by samples and $\eta$ is known (in NFs), or both the target $\nu$ and the reference $\eta$ are given through samples (in GANs). Then,  
   the goal of the sample complexity analysis is to obtain error bounds on 
   $D(\widehat{T}_\sharp \eta, \nu)$ or $\| \widehat{T} - T^\dagger \|_{\mmT}$ in terms 
   of the number of samples from $\nu$ and $\eta$.}
   Analysis of the resulting statistical errors is a topic of great interest in the literature; see, for example, \cite{divol2022optimal, hutter2021minimax, pooladian2021entropic} for estimating OT maps and~\cite{irons2022triangular, wang2022minimax} for triangular and other transport maps.  Obtaining sharp rates for the statistical error in the pushforward measure given different map approximation classes, under various metrics and divergences, is a major step towards a complete error analysis of transport-based generative modeling and density estimation.

    \item Obtaining sharp rates for the Wasserstein and MMD metrics relies 
    on having strong regularity results for the ground truth map $\TT$, which we 
    often take to be an OT map; recall Remark~\ref{remark:wishful-thinking}. Understanding higher-order Sobolev regularity 
    of OT or triangular maps on unbounded domains, however, remains a challenge. Once such regularity results are obtained for a certain $\TT$, then 
    we can immediately improve our rates. 

    \item Extension of our results to the case of infinite-dimensional spaces 
    is another interesting question. This setting is important, for example, to Bayesian inverse problems and to sampling the associated posterior measures on function spaces~\cite{stuart-acta-numerica}. The approximation of certain infinite-dimensional \textit{triangular} transport maps, representing measures over functions defined on {bounded} domains, has been investigated in~\cite{zech2022sparseII}.
    But approximation of infinite-dimensional transport maps in more general settings, e.g., non-triangular transformations, and for measures over functions defined on unbounded domains, is to our knowledge open. The development of computational algorithms for infinite-dimensional transport is similarly unexplored, and leads to many interesting theoretical questions. Many of the rates obtained in this article, as well as in other work
    in the literature, are dimension-dependent and so cannot be easily generalized to infinite dimensions.
\end{itemize}

\bigskip

\section{Proofs of stability results for the KL divergence}
\label{app:stability-proofs}
Below we collect the proofs of Theorems~\ref{thm:kl_compact}, \ref{thm:kl}, 
and \ref{thm:kl_pullback}. We begin with Theorem~\ref{thm:kl}, which requires 
the most technical arguments. The other proofs follow similar steps and 
we highlight their pertinent differences.

% \subsection{Proof of Theorem \ref{lem:wplp}}\label{subsec:proof-of-Wasserstein-stability}

% \footnote{ {\bf to remove }

% The main technical step in our proof is to control the $W_p$ distance between pushforwards of 
% a reference measure through two maps by the $L_q$ distance between the maps. We summarize this 
% result in the following technical lemma which may be of interest on its own.

% \begin{lemma}\label{lem:wplp}
% Let $\Omega \subseteq \mathbb{R}^d$ and $\Omega' \subseteq \mathbb{R}^s$
% be  Borel sets and fix  $\mu \in \PP^p(\Omega)$ for $p \ge 1$. 
% For any $q \ge p$ and  $F, G \in L^q_\mu(\Omega;\Omega')$
% it holds that
% % \begin{equation}\label{eq:wp_lp}
% % W_p (f_\sharp\mu, g_\sharp\mu) \leq \|f-g\|_{L^p (\Omega\to \mathbb{R}^n;\mu)} \, .
% % \end{equation}
% % Furthermore, for $f,g \in L^q= L^q(\Omega \to \mathbb{R}^n ; \mu)$ and $q \geq p$, we have
% \begin{equation}\label{eq:wp_lp}
% W_p (F_\sharp \mu, G_\sharp \mu) \leq \|F-G\|_{L^q_\mu (\Omega; \Omega')}.
% \end{equation}

% \end{lemma}

% \begin{proof}[Proof of Theorem~\ref{prop:wp_dist}]

% Since $\hmT$ was assumed to be closed in $L^q_\eta(\Omega;\Omega) $, there exists $T^* \in \hmT$ such that  $\|T^* - \TT\|_{L^q_\eta(\Omega;\Omega)} =\dist_{L^q_\eta(\Omega;\Omega)}(\TT,\hmT)$. By  \eqref{eq:wass_opt} it follows that  $$W_p (\hnu,\nu) = W_p(\hat{T}_\sharp\eta, \TT_\sharp\eta) \leq W_p(T^*_\sharp\eta, \TT_\sharp\eta) \leq \|T^* - \TT\|_{L^q_\eta(\Omega;\Omega)} \, ,$$
% where the last inequality is obtained by applying Lemma \ref{lem:wplp}.
% \end{proof}
% }

\subsection{Proof of Theorem~\ref{thm:kl}}\label{sec:proof-of-KL-thm}
% \begin{proof}
For notational convenience let us write $\phi\coloneqq F_\sharp \eta$ and $\gamma \coloneqq G_\sharp \eta$.
By Assumptions \ref{as:det_fg}--\ref{as:imagefg}, $F^{-1}$ and $G^{-1}$ are well-defined $\eta$-a.e.\ on ${\rm Im}(G)\subseteq \R ^d$. Hence, by the change of variables formula $$p_{\phi} (y) = \frac{p_{\eta}(F^{-1}(y))}{|\det (J_F(F^{-1}(y)))|} \, , \qquad p_{\gamma} (y) = \frac{p_{\eta}(G^{-1}(y))}{|\det (J_G(G^{-1}(y)))|} \, ,$$
and so by definition
\begin{align*}
{\rm KL}(\phi || \gamma) &= \int\limits_{\R ^d} p_{\phi}(y) \log \left( \frac{p_{\phi}(y)}{p_{\gamma}(y)}\right) \, dy \\
&=\int\limits_{\R ^d} p_{\phi}(y) \log \left( \frac{p_{\eta}(F^{-1}(y))|\det^{-1}(J_F(F^{-1}(y)))|}{p_{\eta}(G^{-1}(y))\,|\det ^{-1}(J_G(G^{-1}(y)))|}\right) \, dy \\
&=\int\limits_{\R ^d} \frac{p_{\eta}(F^{-1}(y))}{|\det (J_F(F^{-1}(y)))|}\left[ \log(p_{\eta}(F^{-1}(y)))-\log(p_{\eta}(G^{-1}(y)))\right] \, dy  \\  &+ \int\limits_{\R ^d}\frac{p_{\eta}(F^{-1}(y))}{|\det (J_F(F^{-1}(y)))|} \left[\log |\det \left(J_G(G^{-1}(y))\right)| -\log |\det \left(J_F(F^{-1}(y))\right)|\right] \, dy \,.\label{eq:kl_int2y} 
\end{align*}
We make the change of variables $y=F(z)$, and therefore $dy=|\det (J_F (z))|\, dz$. Denoting $Q  \coloneqq G^{-1} \circ F$, we have that,
\begin{align*}
{\rm KL}(\phi || \gamma) 
&=\int\limits_{\R ^d} p_{\eta}(z) \left[ \log(p_{\eta}(z))-\log(p_{\eta}(Q(z)))\right] \, dz \numberthis \label{eq:kl_int1z} \\  &+ \int\limits_{\R ^d}p_{\eta}(z) \left[\log |\det \left(J_G(Q(z))\right)| -\log |\det \left(J_F(z)\right)|\right] \, dz \, . \numberthis \label{eq:kl_int2z}
\end{align*}
We now proceed to bound \eqref{eq:kl_int1z}--\eqref{eq:kl_int2z} from above. The following lemma will be useful for both:
\begin{lemma}\label{lem:q_unbd}
For every $z\in \R ^d$, define as above $Q  \coloneqq G^{-1} \circ F$. Then
$$ |Q(z)| \leq |z| + c_G ^{-1} |F(z)-G(z)| \, , \qquad  \eta-{\rm a.e.} \, ,$$
where $c_G$ is the $\eta$-essential infimum on the smallest eigenvalue of $J_G$, see Assumption \ref{as:loweval_Jg}. As a result, $Q\in L^2 (\R ^d; \eta)$.
\end{lemma}
\begin{proof}
Note that $Q(z)-z= F^{-1}(y)-G^{-1}(y)$. By assumption \ref{as:c1}, Lagrange mean value theorem implies that there exists $x^* = x^*(z)\in \R^d$ on the line segment between $z$ and $Q=Q(z)$ such that 
$$
G(Q)-G(z) = J_G (x^*) (Q-z) \, .
$$
On the other hand, by the definitions $y=F(z)$ and $Q(z) = G^{-1}(F(z))$ and since both $F$ and $G$ are bijective (Assumptions \ref{as:imagefg} and \ref{as:det_fg}), then $G(Q)=y=F(z)$. Therefore
\begin{equation}\label{eq:zqz_gfz}
Q(z)-z = J_G^{-1}(x^*(z))(F(z)-G(z)) \, .
\end{equation}
For any square matrix $A$, denote its induced operator $\ell ^2$ norm by $|A|_{\ell ^2}$. Then by definition of the operator norm, and using \eqref{eq:zqz_gfz}, we have that for every $z\in \R ^d$
\begin{align*}
|Q(z)-z| &\leq | J_G ^{-1}(x^* (z))|_{\ell ^2} \cdot |F(z)-G(z)| \\
&\leq  {\rm ess} \sup_{z\in \R^d}| J_G ^{-1}(z)|_{\ell ^2} \cdot |F(z)-G(z)| \\
&\leq c_G^{-1} \cdot |F(z)-G(z)| \, , \numberthis \label{eq:qz_ubd}
\end{align*}
where we have used Assumption \ref{as:loweval_Jg} in the following way: By assumption, the smallest singular value of $J_G$ is bounded from below by $c_G>0$, and so the spectral radius of $J_G ^{-1}$ is bounded from above by $c_G ^{-1}$ $\eta$-a.e. Since $| A |_{\ell ^2} $ is also the spectral radius, this yields the $\eta$-a.e.\ upper bound on $|Q(z)|$. Since $F,G \in L^2(\R ^d ;\eta)$, this also implies that $Q(z)-z\in L^2(\R ^d ;\eta)$, and since $z\in L^2(\R ^d; \eta)$ (by direct computation for a Gaussian $\eta$), then $Q$ is square integrable as well. 
\end{proof}
\textbf{Upper bound on \eqref{eq:kl_int1z}.} Since $p_{\eta}(z) = (2\pi)^{-d/2}\exp(-\|z\|_2^2/2)$, we have
\begin{align*}
\eqref{eq:kl_int1z} &= \int\limits_{\R ^d} p_{\eta}(z) \left[ \log\left(p_{\eta}(z)\right)-\log\left(p_{\eta}(Q(z))\right)\right] \, \dd z \\
&= \int\limits_{\R ^d} p_{\eta}(z) \left[ \log\left(\frac{\exp(-\|z\|_2^2/2)}{(2\pi)^{d/2}}\right)-\log\left(\frac{\exp(-\|Q(z)\|_2^2/2)}{(2\pi)^{d/2}}\right)\right] \, \dd z \\
&= \frac12 \int\limits_{\R ^d}p_{\eta}(z) (\|Q(z)\|_2^2-\|z\|_2^2) \, \dd z \\
&= \frac12 \int\limits_{\R ^d}p_{\eta}(z) (Q(z)-z)^\top(Q(z)+z) \, \dd z \\
&\leq \frac12 \|Q(z)-z\|_{L^2(\R^d;\eta)}\cdot \|Q(z)+z\|_{L^2(\R ^d ; \eta)} \, ,
\end{align*}
where we have used the Cauchy-Schwartz inequality in $L^2(\R^d;\eta)$. By Lemma~\ref{lem:q_unbd}, we have that \begin{align*}
    \eqref{eq:kl_int1z} &\leq \frac{\|z\|_{L^2(\R ^d ;\eta)} + \|G^{-1}\circ F\|_{L^2(\R ^d ;\eta)}}{2c_G}\|F-G\|_{L^2(\R ^d;\eta)} \\
    &\leq \frac{2c_{G}\|z\|_{L^2(\R ^d ;\eta)} + \|F-G\|_{L^2(\R ^d ;\eta)}}{2c_G^2}\|F-G\|_{L^2(\R ^d;\eta)}\, . \numberthis \label{eq:kl_int1z_bd}
\end{align*}

Note that, taking a sequence $G_n \to F$ in $L^2$, the numerator of the fraction does not vanish because of the $\|z\|_2^2$ term.

\textbf{Upper bound on \eqref{eq:kl_int2z}.} By Assumption \ref{as:c1}, $|\det (J_F)|, |\det(J_G)|>c>0$ $\eta$-a.e. Since $\log$ is Lipschitz on the interval $[c,\infty)$ with Lipschitz constant $c^{-1}$, we have that
\begin{align*}
\eqref{eq:kl_int2z}&\leq c^{-1}\int\limits_{\R ^d}p_{\eta}(z) \left| |\det \left(J_G(Q(z))\right)| - |\det \left(J_F(z)\right)|\right| \, dz \\
&  \leq c^{-1}\left[ {\rm I} + {\rm II} \right] \, , \\
{\rm where}\qquad &{\rm I}  \coloneqq \int\limits_{\R ^d}p_{\eta}(z) \left| |\det \left(J_G(Q(z))\right)| - |\det \left(J_G(z)\right)|\right| \, dz \numberthis \label{eq:Jg_i} \\
&{\rm II}  \coloneqq \int\limits_{\R ^d}p_{\eta}(z) \left| |\det \left(J_G(z)\right)| - |\det \left(J_F(z)\right)|\right| \, dz \, , \numberthis \label{eq:Jg_ii}
\end{align*}
where we have simply added and subtracted $|\det \left(J_G(z)\right)|$ to the integrand and used the triangle inequality. 

We will first bound the integral ${\rm I}$, see \eqref{eq:Jg_i}. Denote $r(\cdot)={\rm det}(J_G(\cdot))$. By Assumption~\ref{as:c1}, $G\in C^2_{\rm loc}$ and therefore $r\in C^1_{\rm loc}$ (all in the sense of $\eta$-a.e.). Hence, for every $z\in \R^d$ there exists $\zeta(z)\in \R ^d$ such that
\begin{align*}
 \int\limits_{\R ^d}p_{\eta}(z) \left| |\det \left(J_G(Q(z))\right)| - |\det \left(J_G(z)\right)|\right| \, dz  &= \int\limits_{\R^d} p_{\eta}(z) \nabla r (\zeta (z))  \cdot (Q(z)-z) \, dz\\
 &\leq \|\nabla r\circ  \zeta \|_{L^2(\R^d;\eta)} \cdot \|Q(z)-z\|_{L^2(\R^d;\eta)} \\
 &\leq \|\nabla r\circ  \zeta \|_{L^2(\R^d;\eta)} \cdot  c^{-1}\|F-G\|_{H^1(\R^d;\eta)}\,,
 \end{align*}
 where we have used Cauchy-Schwartz inequality in $L^2(\R^d ;\eta)$ and Lemma \ref{lem:q_unbd}.
 
To complete the bound on \eqref{eq:Jg_i}, it remains to show that  $\nabla r \circ \zeta \in L^2 (\R ^d ; \eta )$.  Note that by Assumption~\ref{as:tails}, then $\nabla r \in L^2(\R^d ;\eta)$. Next, for almost all $z\in \R ^d$, $\zeta(z)$ lies on the line segment between $z$ and $Q (z)$, i.e., $|\zeta (z)|\leq \max \{ |z|, |Q(z)| \}$. If  $ |Q(z)|\leq   |z|$ as $|z|\to \infty$, then clearly $\nabla r \circ \zeta \in L^2 (\R ^d ; \eta )$. Otherwise, we need to analyze $Q(z)$ as $|z|\to \infty$, which is given (see Lemma \ref{lem:q_unbd}) by 
$$|Q(z)| \leq z+ c_G ^{-1} (F(z)-G(z)) \, . $$
We now see the role of Assumption \ref{as:tails}: Since $\nabla r = \nabla {\rm det}J_G$, the polynomial asymptotic growth of $G$ and its first and second derivative means that $|\nabla r (z)|$ has polynomial growth as well. Hence, the composition of two polynomially-bounded functions is also polynomially bounded, and it is in $L^2(\R ^d; \eta)$.
 
We have now bouded \eqref{eq:Jg_i} from above. To complete the proof of Theorem \ref{thm:kl}, it remains to bound integral ${\rm II}$ (see \eqref{eq:Jg_ii}) from above. To do this, we provide the following lemma bounding the $L^1$ difference between the Jacobian determinants by the Sobolev distance between the functions.
\begin{lemma} \label{lem:determinant_bound} For two maps $F,G$ in $H^{1} \cap W^{1,2}$, the difference of the integrated Jacobian-determinants is bounded by
$$\int\limits_{\R ^d}p_{\eta}(z) | |{\rm det} (J_G(z)) | - |{\rm det} (J_F(z)) | | \, dz \leq C \|F - G\|_{H^1(\R^d; \eta)}\, \, ,$$
where $C =d\cdot  \,  \| \max \{ |J_F|_{\ell ^2} , |J_G|_{\ell ^2} \}\|_{L^{2(d-1)}(\R ^d; \eta)}^{d-1}$.
\end{lemma}
\begin{proof}We first recall the following matrix norm inequality due to Ipsen and Rehman \cite[Theorem 2.12]{ipsen2008matrix}: for any two complex $d\times d$ matrices $A$ and $B$, then 
\begin{equation}\label{eq:ipsen}
    \left| {\rm det}(A) - {\rm det}(B) \right| \leq d \cdot |B-A|_{\ell ^2} \cdot \max \{ |A|_{\ell ^2} , |B|_{\ell ^2} \}^{d-1} \, ,
\end{equation}
where $|\cdot |_{\ell ^2}$ is the induced $\ell ^2$ norm, i.e., the spectral radius. Hence, using Cauchy-Schwartz inequality
\begin{align*}
    \int\limits_{\R ^d}p_{\eta}(z) | &|{\rm det} (J_G(z)) | - |{\rm det} (J_F(z)) | | \, dz  \\
    &\leq  d \cdot  \int\limits_{\R ^d}p_{\eta}(z)  |J_G (z)- J_F (z) |_{\ell ^2} \cdot \max \{ |J_F (z)|_{\ell ^2}, |J_G (z)|_{\ell ^2} \}^{d-1}  \, dz \\
    &\leq  d\cdot  \| |J_F - J_G |_{\ell ^2} \|_{L^2(\R ^d ;\eta)} \cdot\| \,  \, \max \{ |J_F|_{\ell ^2} , |J_G|_{\ell ^2} \}\|_{L^{2(d-1)}(\R ^d; \eta)}^{d-1} \, .
\end{align*}
By an upper bound with the Frobenius norm $|A|_{\ell ^2} \leq |A|_F$ \cite{golub2013matrix}, then for a.e.\ $z\in \R ^d$
$$|J_F (z)- J_G (z) |_{\ell ^2}^2 \leq \sum\limits_{i,j=1}^d |\partial_i F_j (z) - \partial_i G_j (z) |^2 \, .$$
Hence $\left\| |J_F (z)- J_G (z) |_{\ell ^2} \right\|_{L^2(\R^d ;\eta)} = \left\| F-G\right\|_{\dot{H}^1 (\R^d; \eta)} \leq \left\| F-G\right\|_{H^1 (\R^d; \eta)} \, ,$ 
where $\|\cdot\|_{\dot{H}^1 (\R^d; \eta)}$ denotes the weighted homogeneous Sobolev norm of order 1. Similarly, $|J_F|_{\ell ^2}$ and $|J_G|_{\ell ^2}$ can be bounded from above  by the Frobenius norm, and so
$${\rm II} ~~ \leq  d \cdot \|F-G\|_{H^1} \cdot (\|F\|_{W^{1,2(d-1)}(\R^d;\eta)}^{2(d-1)} + \|G\|_{W^{1,2(d-1)}(\R^d;\eta)}^{2(d-1)})^{1/2} \,.$$
\end{proof}

To complete the bound on~\eqref{eq:Jg_ii}, we observe that the tail condition, Assumption \ref{as:tails}, also implies that $F,G\in W^{1,2(d-1)}(\R ^d; \eta)$, since $2(d-1)\geq 2$ for $d\geq 2$ and $F,G \in C^1_{\rm loc}$.

Finally, collecting the upper bounds on \eqref{eq:kl_int1z}, \eqref{eq:kl_int2z} (which decomposes into $\eqref{eq:Jg_i}$ and \eqref{eq:Jg_ii}), we have that there exists a constant $K>0$ depending on norms of $F$, $G$, and the dimension (but not on norms of $F-G$), such that:
 $\KL(\phi||\gamma) \leq K \|F-G\|_{H^1(\R ^d; \eta)} \, .$

\subsection{Proof of Theorem~\ref{thm:kl_compact}}\label{sec:compact_kl_pf}

The proof of the compact case simplifies that of the unbounded case considerably: it is still the case that ${\rm KL}(\phi || \gamma )$ decomposes into the integrals \eqref{eq:kl_int1z} and \eqref{eq:kl_int2z}, where the integrals are over the compact domain $\Omega$. 

As in the unbounded case, we have that $|Q(z)-z|\leq c_G ^{-1}|F(z)-G(z)|$ by Lagrange mean value theorem  $\eta$-a.e. That $Q\in L^2$ now simply follows from continuity of $F$ and $G$.

{\bf Upper bound on \eqref{eq:kl_int1z}.}
Since we assumed that $p_{\eta} , c_{\eta }>0$, we can use the fact that $\log$ is a Lipschitz function on $(c_{\eta},\infty)$ with Lipschitz constant $c_{\eta}^{-1}$. Combined with the Lipschitz property of $p_{\eta}$ (Assumption \ref{as:eta_Lip_compact}), we have that
\begin{align*}
\left|\int\limits_{\Omega} p_{\eta}(z) \left[ \log\left(p_{\eta}(z)\right)-\log\left(p_{\eta}(Q(z))\right)\right] \, dz \right| &\leq \int\limits_{\Omega} p_{\eta}(z) \frac{L_{\eta}}{c_{\eta}}|Q(z)-z| \, dz \\
&= \int\limits_{\Omega} p_{\eta}(z) \frac{L_{\eta}}{c_{\eta}}\left| J^{-1}_G (x^{\star}(z))(F(z)-G(z))\right| \, dz \\
&\leq \frac{L_{\eta}}{c_{\eta}}\left( \max_{x\in \Omega} | J_G(x)^{-1} |_{\ell ^2}\right) \int\limits_{\Omega} p_{\eta}(z) |F(z)-G(z)| \, dz \\
& = \frac{L_{\eta}}{c_{\eta}}\left( \max_{x\in \Omega} | J_G(x)^{-1} |_{\ell ^2}\right)    \cdot \|F-G\|_{L^1_{\eta}} \, ,
\end{align*}
where, $|\cdot |_{\ell ^2}$ is the matrix $\ell ^2$ norm (or the spectral radius). Since $J_G (x)$ is everywhere invertible (Assumption \ref{as:bijective_compact}), it is monotonic and so its smallest singular value is always nonnegative, and by continuity ($G\in C^2$ on a compact set $\Omega$) it is a strictly positive minimum, which we denote by $c_G>0$ (this is why we do not need to assume such a bound in the compact case, compare to Theorem \ref{thm:kl}).
Hence $|J_G (x)^{-1}|_{\ell^2}< c_G ^{-1}$, and so 
\begin{equation}\label{eq:compactKL1}\left|\int\limits_{\Omega} p_{\eta}(z) \left[ \log\left(p_{\eta}(z)\right)-\log\left(p_{\eta}(Q(z))\right)\right] \, dz \right| \leq \frac{L_{\eta}}{c_{\eta}\cdot c_{G}}  \cdot \|F-G\|_{L^1(\R ^d; \eta )} \, .
\end{equation}

\textbf{Upper bound on \eqref{eq:kl_int2z}.} As in the unbounded case, we bound this integral from above by the sum of two integrals ${\rm I} + {\rm II}$, as defined in \eqref{eq:Jg_i} and \eqref{eq:Jg_ii}, respectively. 

To bound ${\rm I}$ from above, Denote $r \coloneqq {\rm det}(J_G)$. By assumption \ref{as:c2_compact}, $G\in C^2$ and therefore $r\in C^1$ on a compact domain, and hence has a Lipschitz constant which we denote by $L_{J_G}\geq 0$. 
\begin{align*}
 \int\limits_{\R ^d}p_{\eta}(z) \left| |\det \left(J_G(Q(z))\right)| - |\det \left(J_G(z)\right)|\right| \, dz  &\leq  L_{J_G} \int\limits_{\R^d} p_{\eta}(z) | Q(z)-z| \, dz\\
 &\leq \frac{L_{J_G}}{c_{J_G}}\|F-G\|_{L^1(\R ^d; \eta )} \, , \numberthis \label{eq:compactKL2_I}
 \end{align*}
 where the second inequality  on $\int p_{\eta}(z) |z-Q(z)| \, dz $ has already been derived above. 
 
 To bound ${\rm II}$ (see \eqref{eq:Jg_ii}), we can again use the matrix inequality \eqref{eq:ipsen} as well as the matrix norms inequality $|A|_{\ell ^2} \leq \sqrt{d} |A|_{\ell ^1}$ \cite{golub2013matrix}, to write 
 \begin{align*}
    \int\limits_{\Omega}p_{\eta}(z) | &|{\rm det} (J_G(z)) | - |{\rm det} (J_F(z)) | | \, dz  \\
    &\leq d^{\frac{d}{2}} \cdot  \int\limits_{\Omega}p_{\eta}(z)  |J_g (z)- J_F (z) |_{\ell ^1} \cdot \max \{ |J_F (z)|_{\ell ^1}, |J_G (z)|_{\ell ^1} \}^{d-1}  \, dz \\
    &\leq d^{\frac{d}{2}} \cdot  \max\limits_{x\in \Omega}\max \{ |J_F (x)|_{\ell ^1}, |J_G (x)|_{\ell^1} \}^{d-1} \cdot  \int\limits_{\Omega}p_{\eta}(z)  |J_G (z)- J_F (z) |_{\ell ^1}    \, dz \, .
\end{align*}
Since for any square matrix $|A|_{\ell ^1} = \max_{1\leq j\leq d}\sum_{i=1}^d |A_{ij}|$, then
$$|J_G (z)- J_F (z) |_{\ell ^1}  \leq \sum\limits_{i,j=1}^d |\partial_i G_j (z) - \partial_i F_j (z)|  \, , \qquad \eta-{\rm a.e.} \, , $$
and so we get that 
\begin{equation}\label{eq:compactKL2_II} \int\limits_{\Omega}p_{\eta}(z) | |{\rm det} (J_G(z)) | - |{\rm det} (J_F(z)) | | \, dz \leq d ^{\frac{d}{2}} \max \{ \|F\|_{C^1} , \|G\|_{C^1} \}^{d-1} \cdot \|F-G\|_{W^{1,1}(\R ^d ;\eta)} \, . 
\end{equation}
In all, combining the upper bounds \eqref{eq:compactKL1},  \eqref{eq:compactKL2_I}, and \eqref{eq:compactKL2_II}, we get that
\begin{equation}\label{eq:KL_compact_fullbound}
    {\rm KL} (\phi|| \gamma) \leq \left( \frac{L_{\eta}}{c_{\eta}\cdot c_G} +\frac{L_{J_G}}{c_{J_G}}\right)\cdot \|F-G\|_{L^1(\R ^d;\eta)} + d^{\frac{d}{2}} \max \{ \|F\|_{C^1}, \|G\|_{C^1} \}^{d-1} \cdot \|F-G\|_{W^{1,1}(\R^d ;\eta )} \, .
\end{equation}

\subsection{Proof of Theorem~\ref{thm:kl_pullback}} \label{app:proof_kl_pullback}
%\begin{proof}
For notational convenience, let us write $\phi\coloneqq F^\sharp \eta$ and $\gamma \coloneqq G^\sharp \eta$. Under Assumption~\ref{as:Gaussiantail_pullback} on the target density, i.e., $p_\phi(x) \leq c_Fp_{\eta}(x)$ for all $x$, we can bound the KL divergence as follows %We first decompose the KL divergence into two terms %$\phi = F^\sharp p_{\eta}$ and $\gamma = G^\sharp p_{\eta}$. We can decompose the KL divergence we have  
\begin{align*}
    \KL(\phi||\gamma) = \int_{\R^d} \log \left(\frac{p_{\phi}(y)}{p_{\gamma}(y)} \right)p_{\phi}(y) \dd y \leq c_F \int_{\R^d} \log \left(\frac{p_{\phi}(y)}{p_{\gamma}(y)} \right) p_{\eta}(y) \dd y.
\end{align*}
%Under Assumption~\ref{as:Gaussiantail_pullback} on the target density $F_\sharp\eta \leq c_F\eta$.
Using the form of the pull-back densities in~\eqref{eq:pullback_density}, the integral on the right-hand side above can be decomposed into two terms
\begin{equation*}
    \int_{\R^d} \log \left(\frac{p_{\phi}(y)}{p_{\gamma}(y)} \right) p_{\eta}(y) \dd y = \underbrace{\int_{\R^d} \log \frac{p_\eta(F(y))}{p_\eta(G(y))} p_\eta(y) \dd y}_{\rm I} + \underbrace{\int_{\R^d} \log \frac{|\det J_F|}{|\det J_G|} p_\eta(y) \dd y}_{\rm II}.
\end{equation*}
Using the form of the standard Gaussian density $p_\eta$ and the Cauchy-Schwarz inequality, the term ${\rm I}$ is bounded by
\begin{align}
{\rm I} = \frac{1}{2}\int_{\R^d} (G(y)^2 - F(y)^2) p_{\eta}(y) \dd y %&\leq  \frac{c_F}{2} \int_{\R^d} (G(y)^2 - F(y)^2) \dd\eta(y) \\
&= \frac{1}{2} \int_{\R^d} (G(y) - F(y))(G(y) + F(y)) p_{\eta}(y) \dd y \nonumber \\
&\leq \frac{1}{2} \|G - F\|_{L^2(\R^d;\eta)} \|G + F \|_{L^2(\R^d;\eta)}. \label{eq:pullback_bound_I}
\end{align}

Next, we turn to bound the term ${\rm II}$. Under Assumption~\ref{as:det_pullback}, the difference of the log-determinants is a Lipschitz function with constant $1/c$. In addition, using Lemma~\ref{lem:determinant_bound} with Assumption~\ref{as:cont_pullback} on the function spaces for the maps, the term ${\rm II}$ is then bounded by
\begin{align}
II %&= \int_{\R^d} (\log|\det J_F| - \log|\det J_G|) p_{\eta}(y) \dd y \nonumber \\
&\leq \frac{1}{c} \int_{\R^d} ||\det J_F| - |\det J_G|| p_{\eta}(y) \dd y \nonumber \\
&\leq \frac{d}{c} \|F - G\|_{H^1_\eta} \max\{\|F \|^{d-1}_{W^{1,2(d-1)}_\eta}, \|G \|^{d-1}_{W^{1,2(d-1)}_\eta}\} \, . \label{eq:pullback_bound_II}
\end{align} 
For more details, see the analogous passage in proving Lemma \ref{lem:determinant_bound}. Collecting the bounds in~\eqref{eq:pullback_bound_I} and~\eqref{eq:pullback_bound_II}, we have
$\KL(\phi||\gamma) \leq C\|F - G \|_{H^1_\eta}$,
where the constant $$C \coloneqq c_F(\|G + F\|_{L^2_\eta}/2 + d/c\max\{\|F \|^{d-1}_{W^{1,2(d-1)}_\eta}, \|G \|^{d-1}_{W^{1,2(d-1)}_\eta}\}) \,. $$

%\end{proof}

\section{Proofs from Section~\ref{sec:applications}}
\label{sec:application-proofs}

\subsection{Proof of Theorem~\ref{thm:nn_unbounded_approx}}
\label{sec:nn_unbounded_approx}
    Without loss of generality, assume \(\epsilon < 1\). For any \(M > 0\) and \(x \in \R^d\), define 
    \begin{equation}\label{eq:FMdef}
    F_M(x) \coloneqq \begin{cases}
    F(x) & |F(x)| \leq M \\
    \frac{M}{|F(x)|} F(x) & |F(x)| > M.
    \end{cases}
    \end{equation}
    Note that, by definition, it follows that \(\|F_M\|_{L^p_\eta} \leq \|F\|_{L^p_\eta}\) for any \(M > 0\). Clearly \(F_M \to F\) as \(M \to \infty\) pointwise hence, by dominated convergence, \(F_M \to F\) in \(L^p_\eta\).
    Therefore, we can find \(M > 1\) large enough such that
    \begin{equation}\label{eq:FmF}
    \|F_M - F\|_{L^p_\eta} < \frac{\epsilon}{2} \, .
    \end{equation}
    By Lusin's theorem \cite[Theorem 7.1.13]{bogachev1}, there exists a compact set \(K \subset \R^d\) such that 
    \begin{equation}\label{eq:etaKc}
    \eta(\R^d \setminus K) < \frac{\epsilon^p}{2^{p+2}(3^p + 1)M^p} \, .
    \end{equation}
    and \(F_M |_K\) is continuous. By \cite[Theorem 3.1]{pinkus1999approximation}, there exists a number \(n_1 = n_1(\epsilon) \in \mathbb{N}\) 
    \footnote{
    Observe that $n_1$ depends on \(K\) which in turn depends on \(\epsilon\), \(F\), and \(\eta\). In order to obtain rates we need to quantify how $n_1$ depends on $K$ but we do not need this for the purposes of this proof.}    and a
    ReLU network \(G_1: \mathbb{R}^d \to \R^m\) with \(n_1\) parameters such that
    \begin{equation}\label{eq:G1FM}
    \sup_{x \in K} |G_1(x) - F_M (x) | < 2^{-\frac{2p+1}{p}} \epsilon \, .
    \end{equation}
    By \cite[Lemma C.2]{lanthaler2022error}, there exists a number \(n_2 = n_2(\epsilon) \in \mathbb{N}\) and a 3-layer ReLU network \(G_2 : \R^m \to \R^m\)
    with \(n_2\) parameters such that
    \begin{equation}\label{eq:G2_2M}
    \sup_{|x| \leq 2M} |G_2(x) - x| < 2^{-\frac{2p+1}{p}} \epsilon , \quad \sup_{x \in \R^d} |G_2(x)| \leq 3M.
    \end{equation}
    Define \(\widehat{F} : \R^d \to \R^m\) as \(\widehat{F} = G_2 \circ G_1\) which is a 4-layer ReLU network with \(n = n_1 + n_2\) parameters.
    Notice that using \eqref{eq:FMdef} and \eqref{eq:G1FM}, we have that
    \[\sup_{x \in K} |G_1(x)| \leq \sup_{x \in K} |G_1(x) - F_M(x)| + \sup_{x \in K} |F_M(x)| \leq 2M\]
    and therefore \eqref{eq:G2_2M} can be applied to yield,
    \begin{align*}
    \sup_{x \in K} |\widehat{F}(x) - G_1(x)| &= \sup_{x\in K} |G_2 (G_1 (x))-G_1 (x)|\\
    &\leq \sup_{|x| \leq 2M} |G_2(x) - x| \\
    &< 2^{-\frac{2p+1}{p}} \epsilon \, . \numberthis \label{eq:hatF_G1}
    \end{align*}
    By combining \eqref{eq:G1FM} and \eqref{eq:hatF_G1} and the triangle inequality, we get
    \[\sup_{x \in K} |\widehat{F}(x) - F_M(x)| \leq \sup_{x \in K} |\widehat{F}(x) - G_1(x)| + \sup_{x \in K} |G_1(x) - F_M(x)| < 2^{-\frac{p+1}{p}} \epsilon \, ,\]
    and therefore, since $\eta$ is a probability measure and $K$ is a proper subset of $\R^d$, i.e., $\eta (K)<1$,
    \[\int_K |\widehat{F} - F_M|^p \: d \eta \leq \sup_{x \in K} |\widehat{F}(x) - F_M (x)|^p < \frac{\epsilon^p}{2^{p+1}}.\]
    Furthermore, By the definition of $F_M$, \eqref{eq:FMdef}, and by \eqref{eq:etaKc},
    \begin{align*}
        \int_{\R^d \setminus K} |\widehat{F} - F_M|^p \: d \eta &\leq 2 \eta(\R^d \setminus K) \left ( \sup_{x \in \R^d} |\widehat{F}(x)|^p + \sup_{x \in \R^d} |F_M(x)|^p \right ) \\
        &< \frac{\epsilon^p}{2^{p+1}(3^p + 1)M^p} (3^p M^p + M^p) = \frac{\epsilon^p}{2^{p+1}}.
    \end{align*}
    Combining the two integrals on $K$ and $\R^d \setminus K$, it follows that, $\|\widehat{F} - F_M\|_{L^p_\eta} < \frac{\epsilon}{2}.$
    Hence by using \eqref{eq:FmF} as well we finally get the desired result
    \[\|\widehat{F} - F\|_{L^p_\eta} \leq \|\widehat{F} - F_M\|_{L^p_\eta} + \|F_M  - F\|_{L^p_\eta} < \epsilon\, .\]

\subsection{Proof of Proposition~\ref{prop:application-compact-rates}}\label{proof:prop:application-compact-rates}

Each upper bound in Proposition \ref{prop:application-compact-rates} is an application of Theorem~\ref{thm:abstract-error} after verifying Assumption~\ref{ass:abstract} 
for the corresponding divergence. Throughout the proof we take $\TT$ to be the $\W_2$-optimal map 
 pushing $\eta$ to $\nu$.

%\begin{enumerate}
Let us start with item 1. By \eqref{eq:locReg} we have $\TT \in H^{k+1}_{\eta}(\Omega ; \Omega)$. Hence, the $L^2$ approximation error is obtained from \eqref{eq:canuto} with $s = 0$ and $t = k+1$, which reads as
 \begin{equation}\label{eq:leg_aprox} \dist_{L^2_{\eta}(\Omega;\Omega)} \rev{(\TT, \widehat \mmT^n)} \leq 
 Cn^{-k-3/2} \|  \TT \|_{H^{k+1}_{\eta}(\Omega, \Omega)} \, ,
 \end{equation}
where $C>0$ depends on $d$ and $\TT$, but not on $n$. The reference measure satisfies 
$\eta \in \PP^p (\Omega)$, since $p_{\eta}\in C^0(\Omega)$ and $\Omega$ is compact. Combined with 
$\hT, \TT\in L^2_\eta(\Omega)$,
 we obtain stability by Theorem~\ref{lem:wplp}. An application of Theorem~\ref{thm:abstract-error} then 
yields the result.

 Next we consider item 2. The approximation-theoretic bound \eqref{eq:leg_aprox} remains applicable so we only need to 
verify the stability of MMD, i.e., that the hypotheses of Theorem~\ref{thm:MMD-stability} are satisfied. Let $\psi(x) = \kappa(x - \cdot)$, the canonical feature map of the 
Gaussian kernel $\kappa$. By Remark~\ref{rem:mmd_hyp}
it follows that
 $   \| \psi(x) - \psi(y) \|_\mK^2 \le L^2 |x - y|^2,$
for some constant $L(\gamma) >0$. Thus, $\MMD_\kappa$ satisfies 
the conditions of Theorem~\ref{thm:MMD-stability}  and the desired result then 
follows analogously to item 1.

It remains to prove item 3, concerning the KL-divergence, which is significantly more technical.
First, we need to verify that the minimizer $\hT$ \eqref{eq:abst_opt_theta} exists, i.e., that there exists a polynomial \rev{$S\in \widehat \mmT^n$} such that ${\rm KL}(S_{\sharp}\eta || \nu) < +\infty$. The issue here is image mismatch: for an arbitrary polynomial \rev{ $S\in \widehat \mmT^n$}, it may  be that $|S(\Omega)\setminus \Omega|>0$, and so by definition the ${\rm KL}$ distance is $+\infty$. However, since \rev{$\widehat \mmT^n$} is a vector space, we can divide 
such a map $S$ by a sufficiently large number $a>0$ such that $(a^{-1} S)(\Omega)\subseteq \Omega$, which would lead to ${\rm KL}((a^{-1}S)_{\sharp}\eta  \| \nu) < +\infty$. Hence, a  minimizer $\hT$ exists.

To prove the error rate, we make a small variation on our usual strategy. Note that by  \eqref{eq:abst_opt_theta} it holds that,
\begin{equation}\label{eq:kl_compact_ineq_gen}
{\rm KL}(\hat{\nu}^n  \| \nu)= {\rm KL}(\hat{\nu}^n \| \TT _{\sharp} \eta) \leq {\rm KL }(S _{\sharp} \eta \| \TT _{\sharp} \eta ) \, , \qquad \forall \,\rev{S\in \widehat \mmT^n } \, .
\end{equation}
We will therefore apply Theorem \ref{thm:kl_compact} to $\TT$ and a judiciously constructed $S$.

To choose the proper $S$ in \eqref{eq:kl_compact_ineq_gen}, first, denote as before by $\pi_n \TT$ 
the $L^2$ projection of $\TT$ onto \rev{$\widehat \mmT^n$}. Unfortunately, even though $\hT(\Omega)\subseteq \Omega$, we cannot guarantee that $\pi_n \TT(\Omega) \subseteq \Omega$. Define  a renormalized version  
\begin{equation}\label{eq:psibardef}
    \varphi ^n(x) \coloneqq \frac{1}{c_n} \pi_n \TT (x) \, ,  \qquad  c_n \coloneqq \max \large\{1 \, , ~ \max\limits_{j=1,\ldots ,d}\max\limits_{x\in \Omega} |(\pi _n \TT)_j (x)|  \large\} \, .
\end{equation}
Observe that if  $\pi _n \TT(\Omega)\subseteq \Omega$ then $c_n =1$ and $\varphi ^n = \pi_n \TT$. Otherwise, $c_n>1$ and we have made sure, by definition, that $\varphi^n(\Omega)\subseteq \Omega$. Hence, since we assumed that the target density $p_{\nu} >0$ a.e.\ in $\Omega$, we have that ${\rm KL}(\varphi ^n_{\sharp}\eta \| \nu) < +\infty$.

We now wish to apply the stability result, Theorem \ref{thm:kl_compact}, to the right-hand side of \eqref{eq:kl_compact_ineq_gen} with $S=\varphi ^n$. Let us verify that the assumptions are being satisfied, with $F =\varphi ^n$ and $G=\TT$:
\begin{itemize}
    \item Assumption \ref{as:c2_compact} on local regularity is satisfied since $\varphi ^n$ is a polynomial, and $\TT$ is the $\W_2$-optimal map, which satisfies the regularity result \eqref{eq:locReg}. 
    \item To show that $\TT$ is injective, we note that it solves the Monge-Ampere equation $${\rm det}(J_{\TT}(x))=\frac{p_{\eta}(x)}{p_{\mu}(\TT(x))} \, , \qquad x\in \Omega \, .$$
    Since both $p_{\eta}, p_{\mu}$ are strictly positive, we have that $J_{\TT}$ is everywhere invertible on $\Omega$, and therefore $\TT$ is injective. Since $\varphi ^n(\Omega) \subseteq \Omega =\TT(\Omega)$, then $\varphi ^n$
    is invertible on $\varphi ^n(\Omega)$.

    To show that Assumption \ref{as:bijective_compact} is satisfied, we still need to show that $\varphi ^n$ is injective. Since $\varphi ^n$ is a constant rescaling of $\pi_n \TT$, it is sufficient to show that $\pi_n \TT$ is monotonic.  Since $\TT$ is invertible on a compact domain, $|{\rm det} J_{\TT}|>c >0 $ on $\Omega$ for some $c>0$. Below in Lemma \ref{lem:C12-bound-of-poly-maps} we show that $T^n$ converges to $\TT$ in $C^1$ and so in particular $\|{\rm det}J_{\pi_n \TT} - {\rm det} J_{\TT} \|_{\infty} \to 0$ as $n\to \infty$. Hence, for sufficiently large $n$ we have that $|{\rm det}J_{\pi_n \TT}|>c/2 >0$ and therefore $\pi_n \TT$ is injective.

\item Assumptions \ref{as:eta_pos_compact} and \ref{as:eta_Lip_compact} only concern the reference measure $\eta$ and are readily satisfied by 
the hypotheses of the theorem.
\end{itemize}

We can therefore apply Theorem \ref{thm:kl_compact} to \eqref{eq:kl_compact_ineq_gen} and obtain 
\begin{equation}
    {\rm KL}(\hat{\nu}^n \|\nu) \leq C \|\varphi ^n - \TT \|_{W_\eta^{1,1}(\Omega; \Omega)} \leq C \|\varphi ^n - \TT \|_{ H_\eta^{1}(\Omega;\Omega)} \, ,
\end{equation}
where the constant $C$ changes between the two inequalities. First note that, by Theorem \ref{thm:kl_compact}, the constant $C$ only depends on $\eta$ and on $G=\TT$, and hence is uniform in $n$. It remains for us to derive an upper bound for $\|\varphi ^n - \TT \|_{H_\eta^{1}(\Omega;\Omega)}$. Using the definition of $\varphi^n$, \eqref{eq:psibardef}, and the triangle inequality, 
\begin{align*}
     \|\varphi ^n - \TT \|_{H_\eta^{1}} &= \|c_n^{-1}\pi_n \TT -\TT\|_{H_\eta^1} \\
     &\leq c_n^{-1} \|\pi_n \TT -\TT\|_{H_\eta^1} + |1-c_n^{-1}| \cdot \|\TT\|_{H_\eta^1} \, .
     \end{align*}

Recall now that by definition \eqref{eq:psibardef}, $c_n \ge 1$, and it quantifies the image mismatch between $T^n (\Omega)$ and $\Omega$. Since $\TT(\Omega)=\Omega$ we can relate the constant $c_n$ to the $L^{\infty}$ distance between $\TT$ and its polynomial approximation $\pi_n \TT$. Choose any $x\in \Omega$. Then by definition \eqref{eq:psibardef}, in the worst case it is mapped $c_n$ away from $[-1,1]$ in each coordinate. On the other hand, since $\TT (x)\in \Omega$, the furthest away that $\pi_n \TT (x)$ can be from $\Omega$ is represented by the case where $\TT (x)$ is a corner  of $\Omega$, e.g., $(1,\ldots, 1)$. Even in this worst case scenario, it is still the case that $|\pi_n \TT (x) - \TT(x)|\leq \|\pi_n \TT - \TT \|_{\infty}$. Hence
$$ d|1-c_n^{-1}|^2 \leq \|\TT- \pi_n \TT\|_{\infty}^2 \, .$$
This is a loose bound but it is sufficient  for our needs. Since
$c_n \geq 1$ we obtain
\begin{equation}
    \| \varphi ^n - \TT \|_{H_\eta^{1}}  \leq \|\pi_n \TT -\TT\|_{H_\eta^1} +  d^{-\frac12} \|\pi_n \TT - \TT\|_{\infty} \cdot \|\TT\|_{H_\eta^1} \, . 
\end{equation}
An upper bound on the first  terms on the right hand side is a direct consequence of the classical approximation theory \eqref{eq:canuto}. We obtain a bound on the $L^{\infty}$ difference in Lemma~\ref{lem:C12-bound-of-poly-maps}, and so, combined with \eqref{eq:canuto} we get, 
$$  \| \varphi ^n - \TT \|_{H_\eta^{1}} \le C n^{-k+\frac12} +  n^{-k+\frac12+2\lfloor \frac{d}{2}\rfloor} \, .  $$
For sufficiently high $n$, the second term is dominant, and hence the theorem.

\begin{lemma}\label{lem:C12-bound-of-poly-maps}
Suppose the conditions of  Proposition~\ref{prop:application-compact-rates} are satisfied, 
%Then $\|\bar{T}^n\|_{C^2(\Omega; \Omega)}$  is uniformly bounded in $n$, and  
Then for $j=0,1$
\begin{equation}\label{eq:poly_c01_rates}
\|\TT - \pi_n \TT\|_{C^j} \le C \|\TT\|_{H^{k+1}} n^{-\left( k-\frac12-2j-2\lfloor \frac{d}{2} \rfloor \right)} \, .
\end{equation}
for a positive and $n$-independent constant $C > 0$.
\end{lemma}

\begin{proof}
Recall the Sobolev-Morrey embedding theorem \cite[Ch.~5, Thm.~6]{evans2010partial} stating that 
$C^{s-\lfloor \frac{d}{2}\rfloor-1}(\Omega) \subset H^s(\Omega)$  if $s > d/2$.
Fix $j=0,1$, and choose $s_j =1+j+\lfloor d/2 \rfloor$. Then
$$\|\TT - \pi_n \TT\|_{C^j} \lesssim \|\TT-\hT\|_{H^{s_j}(\Omega;\R^d)} \, .$$
Then, to apply the approximation theory result \eqref{eq:canuto}, we recall that, by the hypotheses of this lemma, and \eqref{eq:locReg}, $\TT\in H_\eta^{k+1}(\Omega; \Omega)$, and so \eqref{eq:poly_c01_rates} is obtained. Finally, we note here that for $C^1$ convergence, we need to require that $e(s_1,k+1)>0$ (using the notation of \eqref{eq:canuto}), i.e., that $k>5/2 +2\lfloor d/2 \rfloor $.
\end{proof}

\subsection{An abstract theorem for backward transport problem} \label{app:abstract-error-backward-transport}
To derive analogous error bounds to those for pushforward measures in Section~\ref{sec:applications}, we first present an abstract result to bound the KL divergence between a target measure and an approximate pullback measure. %in Section~\ref{subsec:pullback_nontriangularmaps}. 
%Next, we specialize the  result for triangular maps $T$ in Section~\ref{subsec:pullback_KR}. 
%Both of these results can be combined with the following corollary of Theorem~\ref{thm:abstract-error} to bound the KL divergence from the approximate measure in~\ref{eq:approximate_pullback} to the target measure $\nu$. 
An application of this theorem to derive convergence results for an increasing class of monotone and triangular maps is presented in Section~\ref{subsec:pullback_KR}. The {\em abstract} framework is a direct analog of the pushforward-based Theorem~\ref{thm:abstract-error}.

\begin{assumption} \label{as:backward_transport} For measures $\nu,\eta \in \mathbb{P}(\R)$, let the following conditions hold:
\begin{enumerate}[label=(\roman*)]
    \item \it{(Stability)} there exists a constant $C > 0$ so that for any set f invertible maps \rev{$F,G \in \mathcal{T}$} it holds that $$\KL(F^\sharp\eta||G^\sharp\eta) \leq 
    \rev{C\|F - G\|, \quad \forall F,G \in \mathcal{T} \, ,}$$
    \rev{for some norm $\|\cdot \|$ of an ambient space containing $\mathcal{T}$.}
    \item \it{(Feasibility)} There exists a map $\TT \in \mathcal{T}$ satisfying $(\TT)^\sharp \eta = \nu$.
\end{enumerate}
\end{assumption}
\begin{theorem} \label{thm:abstract-pullback} Suppose Assumption~\ref{as:backward_transport} holds and consider the approximate measure %$\widehat{\nu}$ in~\eqref{eq:approximate_pullback}
\begin{equation} \label{eq:approximate_pullback}
     \widehat{\nu} \equiv \widehat{T}^\sharp\eta, \qquad \widehat{T} \coloneqq 
     \rev{\argmin_{S \in \widehat \mmT}} 
     \,\KL(\nu || S^\sharp \eta),
\end{equation}
where, as before, \rev{$\widehat \mmT$} denotes a parameterized class of maps. %. It is now evident that \eqref{eq:approximate_pullback} is the pull-back analog of our abstract algorithm \eqref{eq:abst_opt}, and so amenable to the same type of analysis.
Then it holds that
$$\KL(\nu||\widehat{\nu}) \leq \rev{C\dist_{\|\cdot \|}(\widehat \mmT,T^\dagger)} \, ,$$
where $C$ is the same constant as in Assumption~\ref{as:backward_transport}(i). 
%a class of invertible maps that satisfy Assumptions~\ref{as:cont_pullback}-\ref{as:Gaussiantail_pullback} and assume there exists a map $T^\dagger \in \mathcal{T}$ satisfying $(T^\dagger)^\sharp\eta = \nu$. Then, for $\widehat{\nu}$ in~\eqref{eq:approximate_pullback} it holds that
%$$\KL(\nu||\widehat{\nu}) \leq C\dist_{\mathcal{T}}(\hmT,T^\dagger),$$
\end{theorem}
\begin{proof}
The proof is identical to that of Theorem~\ref{thm:abstract-error} with $D$ chosen to be KL divergence from $\widehat\nu$ to $\nu$.
\end{proof}

\subsection{Proof of Theorem~\ref{thm:kl_pullback_KRmap}} \label{app:proof_kl_pullback_triangular}

Under Assumption~\ref{as:Gaussiantail_pullback_KR} on the target density, we follow the same approach as in the proof of Theorem~\ref{thm:kl_pullback} to bound the KL divergence by the sum of two terms
\begin{equation*}
    \KL(F^\sharp\eta||G^\sharp\eta) \leq c_F \underbrace{\int_{\R^d} \log \frac{p_\eta(F(y))}{p_\eta(G(y))} p_\eta(y) \dd y}_{\rm I} + c_F \underbrace{\int_{\R^d} \log \frac{|\det J_F|}{|\det J_G|} p_\eta(y) \dd y}_{\rm II}.
\end{equation*}
Using the Cauchy-Schwarz inequality, the term ${\rm I}$ is bounded as in the general (non-triangular) case, see \eqref{eq:pullback_bound_I}. This yields
\begin{align*}
{\rm I} &\leq \frac{1}{2} \|G - F\|_{L^2(\R^d;\eta)} \|G + F \|_{L^2(\R^d;\eta)} \\
&= \frac{1}{2} \|G + F \|_{L^2(\R^d;\eta)} \left(\sum_{i=1}^{d} \|G_i - F_i\|_{L^2(\R^i;\eta_{\leq i})} \right)\, ,
\end{align*}
where in the last equality we used the fact that $F$ and $G$ are triangular. 

We turn to bound the term ${\rm II}$. Recall that the determinant Jacobian of a triangular map is the product of the partial derivatives of each map component with respect to its diagonal variable, i.e., $\det J_F = \prod_{i=1}^{d} \frac{\partial F_i}{\partial x_i}$. Thus, the difference of the log-determinants simplifies to 
\begin{align*}
    II = \int_{\R^d} \log\frac{|\det J_F|}{|\det J_G|} p_{\eta}(y) \dd y &= \sum_{i=1}^{d} \int_{\R^d} (\log\partial_{x_i} F_i - \log\partial_{x_i} G_i) p_{\eta}(y) \,  \dd y \, ,
\end{align*}
Under Assumption~\ref{as:det_fg_pullback_KR} on the lower bound of the partial derivatives, the log is a Lipschitz function with constant $1/c$ and term $II$ is bounded by
\begin{align} \label{eq:pullback_triangular_bound_II}
    II \leq \frac{1}{c} \sum_{i=1}^{d} \int_{\R^d} |\partial_i F_i - \partial_i G_i| p_{\eta}(y) \dd y \leq \frac{1}{c} \sum_{i=1}^{d} \|F_i - G_i\|_{V_i(\R^i;\eta_{\leq i})}
\end{align}
Collecting the bounds in~\eqref{eq:pullback_bound_I} and~\eqref{eq:pullback_triangular_bound_II} we have $${\rm KL}(F^\sharp\eta||G^\sharp\eta) \leq C \sum_{i=1}^d \|F_i - G_i\|_{V_i(\R^i;\eta_{\leq i})},$$ with the constant $C \coloneqq c_F(\|G + F\|_{L^2_\eta}/2 + 1/c)$. We note that the upper bound is finite for map components $F_i,G_i$ that satisfy Assumption~\ref{as:pullback_KR}.

\subsection{Lemmas for Proposition~\ref{prop:KRconvergence}} \label{app:lemma_monotonemaps}

\begin{lemma}[Proposition 3 in~\cite{baptista2020representation}] \label{lem:prop3}
Let $r \colon \mathbb{R} \rightarrow \mathbb{R}_{>0}$ be a Lipschitz function, i.e., there exists a constant $L <\infty$ so that
$$|r(\xi) - r(\xi')| \leq L|\xi - \xi'|,$$
holds for any $\xi,\xi' \in \mathbb{R}$. Then $\mathcal{R}_i(f) \in V_{\eta_{\leq i}}$ for any $f \in V_{\eta_{\leq i}}$ where $\mathcal{R}_i$ is defined in~\eqref{eq:monotone_maps}. Furthermore, there exist a constant $C < \infty$ so that
$$\|\mathcal{R}_i(f_1) - \mathcal{R}_i(f_2)\|_{V_{\eta_{\leq i}}} \leq C\|f_1 - f_2\|_{V_{\eta_{\leq i}}},$$
holds for any $f_1,f_2 \in V_{\eta_{\leq i}}$.
\end{lemma}

\begin{lemma}[Proposition 5 in~\cite{baptista2020representation}] \label{lem:prop5} 
Let $r^{-1} \colon \mathbb{R}_{>0} \rightarrow \mathbb{R}$ be a Lipschitz function except at the origin, i.e., for any $c > 0$ there exists a constant $L_c < \infty$ so that 
$$|r^{-1}(\xi) - r^{-1}(\xi')| \leq L_c|\xi - \xi'|,$$
holds for any $\xi,\xi' \geq c$. Then for any $s_i \in V_{\eta_{\leq i}}$ such that $\essinf \partial_{x_i} s_i > 0$, we have $\mathcal{R}_i^{-1}(s) \in V_{\eta_{\leq i}}$ and $\essinf \partial_{x_i} \mathcal{R}_i^{-1}(s) > -\infty$. Furthermore for any $c > 0$, there exists a constant $C_c < \infty$ such that
$$\|\mathcal{R}_i^{-1}(s_1) - \mathcal{R}_i^{-1}(s_2)\|_{V_{\eta_{\leq i}}} \leq C_c\|s_1 - s_2\|_{V_{\eta_{\leq i}}}$$
holds for any $s_1,s_2 \in V_{\eta_{\leq i}}$ where $\essinf \partial_{x_i} s_1 \geq c$ and $\essinf \partial_{x_i} s_2 \geq c$.
\end{lemma}

\begin{lemma}[Proposition 9 in~\cite{baptista2020representation}] \label{lem:prop9} Let $\nu$ be an absolutely continuous measure and $\eta$ be the standard Gaussian $\mathcal{N}(0,I_d)$ on $\R^d$. Let $T \colon \R^d \rightarrow \R^d$ be the Knothe-Rosenblatt rearrangement satisfying $T_\sharp \nu = \eta$. If the probability density function $p_\nu$ satisfies $cp_\eta(x) \leq p_\nu(x) \leq Cp_\eta(x)$ for all $x \in \R^d$ with some constants $0 < c \leq C < \infty$, then for all $x_{<i} \in \R^{i-1}$ and $i = 1,\dots,d$,  $T_i(x_{1:i-1},x_i) = \mathcal{O}(x_i)$ and $\partial_{x_i} T_k(x_{1:i-1},x_i) = \mathcal{O}(1)$ as $|x_i| \rightarrow \infty$.  Furthermore, we have $\essinf \partial_{x_i} T_{i}(x_{1:i-1},x_i) > 0$ for all $i=1,\dots,d$.
\end{lemma}

\section*{Acknowledgments} RB and YM gratefully acknowledge support from the United States Department of Energy M2dt MMICC center under award DE-SC0023187.
BH is supported by the National Science Foundation grant DMS-208535.
RB and BH also gratefully acknowledge support from Air Force Office of Scientific Research under MURI award number FA9550-20-1-0358. AS was supported in part by
Simons Foundation Math + X Investigator Award \#376319 (Michael I.\ Weinstein) and the Binational Science Foundation grant \#2022254, and acknowledges the support of the AMS-Simons Travel Grant.

\bibliographystyle{siamplain}
\bibliography{Ref_MeasTransApprox}

\begin{thebibliography}{100}

\bibitem{ambrosio2005gradient}
{\sc L.~Ambrosio, N.~Gigli, and G.~Savar{\'e}}, {\em Gradient flows: in metric
  spaces and in the space of probability measures}, Springer Science \&
  Business Media, 2005.

\bibitem{andrieu2003introduction}
{\sc C.~Andrieu, N.~De~Freitas, A.~Doucet, and M.~I. Jordan}, {\em An
  introduction to {MCMC} for machine learning}, Machine learning, 50 (2003),
  pp.~5--43.

\bibitem{arjovsky2017wasserstein}
{\sc M.~Arjovsky, S.~Chintala, and L.~Bottou}, {\em Wasserstein generative
  adversarial networks}, in International conference on machine learning, PMLR,
  2017, pp.~214--223.

\bibitem{baptista2020representation}
{\sc R.~Baptista, Y.~Marzouk, and O.~Zahm}, {\em On the representation and
  learning of monotone triangular transport maps}, arXiv preprint
  arXiv:2009.10303,  (2020).

\bibitem{benamou2000computational}
{\sc J.-D. Benamou and Y.~Brenier}, {\em A computational fluid mechanics
  solution to the {M}onge-{K}antorovich mass transfer problem}, Numerische
  Mathematik, 84 (2000), pp.~375--393.

\bibitem{benamou2014numerical}
{\sc J.-D. Benamou, B.~D. Froese, and A.~M. Oberman}, {\em Numerical solution
  of the optimal transportation problem using the monge--amp{\`e}re equation},
  Journal of Computational Physics, 260 (2014), pp.~107--126.

\bibitem{bhattacharya2021model}
{\sc K.~Bhattacharya, B.~Hosseini, N.~B. Kovachki, and A.~M. Stuart}, {\em
  Model {Reduction} {And} {Neural} {Networks} {For} {Parametric} {PDEs}}, The
  SMAI journal of computational mathematics, 7 (2021), pp.~121--157.

\bibitem{binkowski2018demystifying}
{\sc M.~Bi{\'n}kowski, D.~J. Sutherland, M.~Arbel, and A.~Gretton}, {\em
  Demystifying {MMD GANs}}, in International Conference on Learning
  Representations, 2018.

\bibitem{birrell2022f}
{\sc J.~Birrell, P.~Dupuis, M.~Katsoulakis, Y.~Pantazis, and L.~Rey-Bellet},
  {\em $(f, \gamma)$-divergences: Interpolating between f-divergences and
  integral probability metrics}, Journal of machine learning research,  (2022).

\bibitem{bishop2006pattern}
{\sc C.~M. Bishop and N.~M. Nasrabadi}, {\em Pattern recognition and machine
  learning}, vol.~4, Springer, 2006.

\bibitem{blei2017variational}
{\sc D.~M. Blei, A.~Kucukelbir, and J.~D. McAuliffe}, {\em Variational
  inference: A review for statisticians}, Journal of the American statistical
  Association, 112 (2017), pp.~859--877.

\bibitem{bogachev1}
{\sc V.~I. Bogachev}, {\em Measure Theory}, vol.~1, Springer, New York, 2007.

\bibitem{bogachev2}
{\sc V.~I. Bogachev}, {\em Measure Theory}, vol.~2, Springer, New York, 2007.

\bibitem{bogachev2006nonlinear}
{\sc V.~I. Bogachev and A.~V. Kolesnikov}, {\em Nonlinear transformations of
  convex measures}, Theory of Probability \& Its Applications, 50 (2006),
  pp.~34--52.

\bibitem{bogachev2005triangular}
{\sc V.~I. Bogachev, A.~V. Kolesnikov, and K.~V. Medvedev}, {\em Triangular
  transformations of measures}, Sbornik: Mathematics, 196 (2005), pp.~309--335.

\bibitem{bonnotte2013knothe}
{\sc N.~Bonnotte}, {\em From {K}nothe's rearrangement to {B}renier's optimal
  transport map}, SIAM Journal on Mathematical Analysis, 45 (2013), pp.~64--87.

\bibitem{brenier1987decomposition}
{\sc Y.~Brenier}, {\em D{\'e}composition polaire et r{\'e}arrangement monotone
  des champs de vecteurs}, CR Acad. Sci. Paris S{\'e}r. I Math., 305 (1987),
  pp.~805--808.

\bibitem{butler2018convergence}
{\sc T.~Butler, J.~Jakeman, and T.~Wildey}, {\em Convergence of probability
  densities using approximate models for forward and inverse problems in
  uncertainty quantification}, SIAM Journal on Scientific Computing, 40 (2018),
  pp.~A3523--A3548.

\bibitem{butler2022p}
{\sc T.~Butler, T.~Wildey, and W.~Zhang}, {\em $l^p$ convergence of approximate
  maps and probability densities for forward and inverse problems in
  uncertainty quantification}, International Journal for Uncertainty
  Quantification, 12 (2022).

\bibitem{caffarelli1992regularity}
{\sc L.~A. Caffarelli}, {\em The regularity of mappings with a convex
  potential}, Journal of the American Mathematical Society, 5 (1992),
  pp.~99--104.

\bibitem{caffarelli2000monotonicity}
{\sc L.~A. Caffarelli}, {\em Monotonicity properties of optimal transportation
  and the fkg and related inequalities}, Communications in Mathematical
  Physics, 214 (2000), pp.~547--563.

\bibitem{canuto1982approximation}
{\sc C.~Canuto and A.~Quarteroni}, {\em Approximation results for orthogonal
  polynomials in sobolev spaces}, Mathematics of Computation, 38 (1982),
  pp.~67--86.

\bibitem{carlier2016vector}
{\sc G.~Carlier, V.~Chernozhukov, and A.~Galichon}, {\em Vector quantile
  regression: an optimal transport approach}, The Annals of Statistics, 44
  (2016), pp.~1165--1192.

\bibitem{carlier2010knothe}
{\sc G.~Carlier, A.~Galichon, and F.~Santambrogio}, {\em From knothe's
  transport to brenier's map and a continuation method for optimal transport},
  SIAM Journal on Mathematical Analysis, 41 (2010), pp.~2554--2576.

\bibitem{cheng2003hermite}
{\sc X.~Cheng-Long and G.~Ben-Yu}, {\em Hermite spectral and pseudospectral
  methods for nonlinear partial differential equation in multiple dimensions},
  Computational \& Applied Mathematics, 22 (2003), pp.~167--193.

\bibitem{colombo2017lipschitz}
{\sc M.~Colombo, A.~Figalli, and Y.~Jhaveri}, {\em Lipschitz changes of
  variables between perturbations of log-concave measures}, Annali della Scuola
  Normale Superiore di Pisa. Classe di Scienze. Serie 5, 17 (2017),
  pp.~1491--1519.

\bibitem{stuart-mcmc}
{\sc S.~L. Cotter, G.~O. Roberts, A.~M. Stuart, and D.~White}, {\em {MCMC}
  methods for functions: modifying old algorithms to make them faster},
  Statistical Science, 28 (2013), pp.~424--446.

\bibitem{creswell2018generative}
{\sc A.~Creswell, T.~White, V.~Dumoulin, K.~Arulkumaran, B.~Sengupta, and A.~A.
  Bharath}, {\em Generative adversarial networks: An overview}, IEEE signal
  processing magazine, 35 (2018), pp.~53--65.

\bibitem{cui2022deep}
{\sc T.~Cui and S.~Dolgov}, {\em Deep composition of tensor-trains using
  squared inverse rosenblatt transports}, Foundations of Computational
  Mathematics, 22 (2022), pp.~1863--1922.

\bibitem{cuturi2013sinkhorn}
{\sc M.~Cuturi}, {\em Sinkhorn distances: Lightspeed computation of optimal
  transport}, Advances in neural information processing systems, 26 (2013).

\bibitem{deb2021rates}
{\sc N.~Deb, P.~Ghosal, and B.~Sen}, {\em Rates of estimation of optimal
  transport maps using plug-in estimators via barycentric projections},
  Advances in Neural Information Processing Systems, 34 (2021),
  pp.~29736--29753.

\bibitem{ditkowski2020density}
{\sc A.~Ditkowski, G.~Fibich, and A.~Sagiv}, {\em Density estimation in
  uncertainty propagation problems using a surrogate model}, SIAM/ASA Journal
  on Uncertainty Quantification, 8 (2020), pp.~261--300.

\bibitem{divol2022optimal}
{\sc V.~Divol, J.~Niles-Weed, and A.-A. Pooladian}, {\em Optimal transport map
  estimation in general function spaces}, arXiv preprint arXiv:2212.03722,
  (2022).

\bibitem{marzouk-opt-map}
{\sc T.~A. El~Moselhy and Y.~M. Marzouk}, {\em {B}ayesian inference with
  optimal maps}, Journal of Computational Physics, 231 (2012), pp.~7815--7850.

\bibitem{ernst2012convergence}
{\sc O.~G. Ernst, A.~Mugler, H.-J. Starkloff, and E.~Ullmann}, {\em On the
  convergence of generalized polynomial chaos expansions}, ESAIM: Mathematical
  Modelling and Numerical Analysis, 46 (2012), pp.~317--339.

\bibitem{evans2010partial}
{\sc L.~C. Evans}, {\em Partial differential equations}, vol.~19, American
  Mathematical Soc., 2010.

\bibitem{evans2018measure}
{\sc L.~C. Evans and R.~F. Garzepy}, {\em Measure theory and fine properties of
  functions}, Routledge, 2018.

\bibitem{figalli2017monge}
{\sc A.~Figalli}, {\em The {M}onge-{A}mp{\`e}re equation and its applications},
  European Mathematical Society, 2017.

\bibitem{froese2012numerical}
{\sc B.~D. Froese}, {\em A numerical method for the elliptic
  {M}onge--{A}mp\`{e}re equation with transport boundary conditions}, SIAM
  Journal on Scientific Computing, 34 (2012), pp.~A1432--A1459.

\bibitem{galichon2017survey}
{\sc A.~Galichon}, {\em A survey of some recent applications of optimal
  transport methods to econometrics}, The Econometrics Journal, 20 (2017),
  pp.~C1--C11.

\bibitem{galichon2018optimal}
{\sc A.~Galichon}, {\em Optimal transport methods in economics}, Princeton
  University Press, 2018.

\bibitem{gangbo1996geometry}
{\sc W.~Gangbo and R.~J. McCann}, {\em The geometry of optimal transportation},
  Acta Mathematica, 177 (1996), pp.~113--161.

\bibitem{genevay2018learning}
{\sc A.~Genevay, G.~Peyr{\'e}, and M.~Cuturi}, {\em Learning generative models
  with sinkhorn divergences}, in International Conference on Artificial
  Intelligence and Statistics, PMLR, 2018, pp.~1608--1617.

\bibitem{golub2013matrix}
{\sc G.~H. Golub and C.~F. Van~Loan}, {\em Matrix computations}, JHU press,
  2013.

\bibitem{goodfellow2014generative}
{\sc I.~Goodfellow, J.~Pouget-Abadie, M.~Mirza, B.~Xu, D.~Warde-Farley,
  S.~Ozair, A.~Courville, and Y.~Bengio}, {\em Generative adversarial nets},
  Advances in neural information processing systems, 27 (2014).

\bibitem{gui2021review}
{\sc J.~Gui, Z.~Sun, Y.~Wen, D.~Tao, and J.~Ye}, {\em A review on generative
  adversarial networks: Algorithms, theory, and applications}, IEEE
  Transactions on Knowledge and Data Engineering,  (2021).

\bibitem{gutierrez2001monge}
{\sc C.~E. Guti{\'e}rrez and H.~Brezis}, {\em The {M}onge-{A}mpere equation},
  vol.~44, Springer, 2001.

\bibitem{hutter2021minimax}
{\sc J.-C. H{\"u}tter and P.~Rigollet}, {\em Minimax estimation of smooth
  optimal transport maps}, The Annals of Statistics, 49 (2021), pp.~1166--1194.

\bibitem{ipsen2008matrix}
{\sc I.~C.~F. Ipsen and R.~Rehman}, {\em Perturbation bounds for determinants
  and characteristic polynomials}, SIAM Journal on Matrix Analysis and
  Applications, 30 (2008), pp.~762--776,
  \url{https://doi.org/10.1137/070704770}.

\bibitem{irons2022triangular}
{\sc N.~J. Irons, M.~Scetbon, S.~Pal, and Z.~Harchaoui}, {\em Triangular flows
  for generative modeling: Statistical consistency, smoothness classes, and
  fast rates}, in International Conference on Artificial Intelligence and
  Statistics, PMLR, 2022, pp.~10161--10195.

\bibitem{jaini2019sum}
{\sc P.~Jaini, K.~A. Selby, and Y.~Yu}, {\em Sum-of-squares polynomial flow},
  in International Conference on Machine Learning, PMLR, 2019, pp.~3009--3018.

\bibitem{jordan1998variational}
{\sc R.~Jordan, D.~Kinderlehrer, and F.~Otto}, {\em The variational formulation
  of the fokker--planck equation}, SIAM journal on mathematical analysis, 29
  (1998), pp.~1--17.

\bibitem{kantorovich1942translocation}
{\sc L.~Kantorovich}, {\em On the translocation of masse}, Doklady Acad. Sci.
  URSS (NS), 7--8 (1942), pp.~227--229.

\bibitem{knothe1957contributions}
{\sc H.~Knothe}, {\em Contributions to the theory of convex bodies.}, Michigan
  Mathematical Journal, 4 (1957), pp.~39--52.

\bibitem{kobyzev2020normalizing}
{\sc I.~Kobyzev, S.~J. Prince, and M.~A. Brubaker}, {\em Normalizing flows: An
  introduction and review of current methods}, IEEE transactions on pattern
  analysis and machine intelligence, 43 (2020), pp.~3964--3979.

\bibitem{kolesnikov2013sobolev}
{\sc A.~Kolesnikov}, {\em On sobolev regularity of mass transport and
  transportation inequalities}, Theory of Probability \& Its Applications, 57
  (2013), pp.~243--264.

\bibitem{kolesnikov2014continuity}
{\sc A.~V. Kolesnikov and M.~R{\"o}ckner}, {\em On continuity equations in
  infinite dimensions with non-gaussian reference measure}, Journal of
  Functional Analysis, 266 (2014), pp.~4490--4537.

\bibitem{kovachki2020conditional}
{\sc N.~Kovachki, R.~Baptista, B.~Hosseini, and Y.~Marzouk}, {\em Conditional
  sampling with monotone {GAN}s}, arXiv preprint arXiv:2006.06755,  (2020).

\bibitem{kovachki2021neural}
{\sc N.~Kovachki, Z.~Li, B.~Liu, K.~Azizzadenesheli, K.~Bhattacharya,
  A.~Stuart, and A.~Anandkumar}, {\em Neural operator: Learning maps between
  function spaces}, arXiv preprint arXiv:2108.08481,  (2021).

\bibitem{lanthaler2022error}
{\sc S.~Lanthaler, S.~Mishra, and G.~E. Karniadakis}, {\em Error estimates for
  deeponets: A deep learning framework in infinite dimensions}, Transactions of
  Mathematics and Its Applications, 6 (2022).

\bibitem{li2017mmd}
{\sc C.-L. Li, W.-C. Chang, Y.~Cheng, Y.~Yang, and B.~P{\'o}czos}, {\em {MMD
  GAN:} towards deeper understanding of moment matching network}, Advances in
  neural information processing systems, 30 (2017).

\bibitem{li2021quantitative}
{\sc W.~Li and R.~H. Nochetto}, {\em Quantitative stability and error estimates
  for optimal transport plans}, IMA Journal of Numerical Analysis, 41 (2021),
  pp.~1941--1965.

\bibitem{lindsey2017optimal}
{\sc M.~Lindsey and Y.~A. Rubinstein}, {\em Optimal transport via a
  monge--amp\`{e}re optimization problem}, SIAM Journal on Mathematical
  Analysis, 49 (2017), pp.~3073--3124.

\bibitem{lu2020universal}
{\sc Y.~Lu and J.~Lu}, {\em A universal approximation theorem of deep neural
  networks for expressing probability distributions}, Advances in neural
  information processing systems, 33 (2020), pp.~3094--3105.

\bibitem{marzouk2016sampling}
{\sc Y.~Marzouk, T.~Moselhy, M.~Parno, and A.~Spantini}, {\em Sampling via
  measure transport: An introduction}, Handbook of Uncertainty Quantification,
  (2016), pp.~1--41.

\bibitem{menendez1997jensen}
{\sc M.~Men{\'e}ndez, J.~Pardo, L.~Pardo, and M.~Pardo}, {\em The
  jensen-shannon divergence}, Journal of the Franklin Institute, 334 (1997),
  pp.~307--318.

\bibitem{monge1781memoire}
{\sc G.~Monge}, {\em M{\'e}moire sur la th{\'e}orie des d{\'e}blais et des
  remblais}, De l'Imprimerie Royale,  (1781).

\bibitem{muandet2017kernel}
{\sc K.~Muandet, K.~Fukumizu, B.~Sriperumbudur, and B.~Sch{\"o}lkopf}, {\em
  Kernel mean embedding of distributions: A review and beyond}, Foundations and
  Trends{\textregistered} in Machine Learning, 10 (2017), pp.~1--141.

\bibitem{muzellec2019subspace}
{\sc B.~Muzellec and M.~Cuturi}, {\em Subspace detours: Building transport
  plans that are optimal on subspace projections}, Advances in Neural
  Information Processing Systems, 32 (2019).

\bibitem{Nocedal}
{\sc J.~Nocedal and S.~J. Wright}, {\em {N}umerical {O}ptimization}, Springer,
  2nd~ed., 2006.

\bibitem{nochetto2019pointwise}
{\sc R.~H. Nochetto and W.~Zhang}, {\em Pointwise rates of convergence for the
  oliker--prussner method for the monge--amp{\`e}re equation}, Numerische
  Mathematik, 141 (2019), pp.~253--288.

\bibitem{pal2019difference}
{\sc S.~Pal}, {\em On the difference between entropic cost and the optimal
  transport cost}, arXiv preprint arXiv:1905.12206,  (2019).

\bibitem{panaretos2020invitation}
{\sc V.~M. Panaretos and Y.~Zemel}, {\em An invitation to statistics in
  {W}asserstein space}, Springer Nature, 2020.

\bibitem{papamakarios2021normalizing}
{\sc G.~Papamakarios, E.~T. Nalisnick, D.~J. Rezende, S.~Mohamed, and
  B.~Lakshminarayanan}, {\em Normalizing flows for probabilistic modeling and
  inference.}, Journal of Machine Learning Research, 22 (2021), pp.~1--64.

\bibitem{papamakarios2017masked}
{\sc G.~Papamakarios, T.~Pavlakou, and I.~Murray}, {\em Masked autoregressive
  flow for density estimation}, Advances in neural information processing
  systems, 30 (2017).

\bibitem{parno2018transport}
{\sc M.~D. Parno and Y.~M. Marzouk}, {\em Transport map accelerated markov
  chain monte carlo}, SIAM/ASA Journal on Uncertainty Quantification, 6 (2018),
  pp.~645--682.

\bibitem{peyre2019computational}
{\sc G.~Peyr{\'e}, M.~Cuturi, et~al.}, {\em Computational optimal transport:
  With applications to data science}, Foundations and Trends{\textregistered}
  in Machine Learning, 11 (2019), pp.~355--607.

\bibitem{pinkus1999approximation}
{\sc A.~Pinkus}, {\em Approximation theory of the mlp model in neural
  networks}, Acta Numerica, 8 (1999), p.~143–195.

\bibitem{pooladian2023minimax}
{\sc A.-A. Pooladian, V.~Divol, and J.~Niles-Weed}, {\em Minimax estimation of
  discontinuous optimal transport maps: The semi-discrete case}, arXiv preprint
  arXiv:2301.11302,  (2023).

\bibitem{pooladian2021entropic}
{\sc A.-A. Pooladian and J.~Niles-Weed}, {\em Entropic estimation of optimal
  transport maps}, arXiv preprint arXiv:2109.12004,  (2021).

\bibitem{rezende2015variational}
{\sc D.~Rezende and S.~Mohamed}, {\em Variational inference with normalizing
  flows}, in International conference on machine learning, PMLR, 2015,
  pp.~1530--1538.

\bibitem{casella}
{\sc C.~Robert and G.~Casella}, {\em {Monte Carlo statistical methods}},
  {Springer Science \& Business Media}, 2013.

\bibitem{rosenblatt1952remarks}
{\sc M.~Rosenblatt}, {\em Remarks on a multivariate transformation}, The annals
  of mathematical statistics, 23 (1952), pp.~470--472.

\bibitem{sagiv2020wasserstein}
{\sc A.~Sagiv}, {\em The wasserstein distances between pushed-forward measures
  with applications to uncertainty quantification}, Communications in
  Mathematical Sciences, 18 (2020), pp.~707--724.

\bibitem{sagiv2022spectral}
{\sc A.~Sagiv}, {\em Spectral convergence of probability densities for forward
  problems in uncertainty quantification}, Numerische Mathematik,  (2022),
  pp.~1--22.

\bibitem{santambrogio2015optimal}
{\sc F.~Santambrogio}, {\em Optimal Transport for Applied Mathematicians:
  Calculus of Variations, {PDE}s, and Modeling}, Progress in Nonlinear
  Differential Equations and Their Applications, Springer, 2015.

\bibitem{seguy2018large}
{\sc V.~Seguy, B.~B. Damodaran, R.~Flamary, N.~Courty, A.~Rolet, and
  M.~Blondel}, {\em Large-scale optimal transport and mapping estimation}, in
  International Conference on Learning Representations, 2018, pp.~1--15.

\bibitem{shen2020deep}
{\sc Z.~Shen}, {\em Deep network approximation characterized by number of
  neurons}, Communications in Computational Physics, 28 (2020).

\bibitem{spantini2019coupling}
{\sc A.~Spantini, R.~Baptista, and Y.~Marzouk}, {\em Coupling techniques for
  nonlinear ensemble filtering}, arXiv preprint arXiv:1907.00389,  (2019).

\bibitem{spantini2018inference}
{\sc A.~Spantini, D.~Bigoni, and Y.~Marzouk}, {\em Inference via
  low-dimensional couplings}, The Journal of Machine Learning Research, 19
  (2018), pp.~2639--2709.

\bibitem{stuart-acta-numerica}
{\sc A.~M. Stuart}, {\em Inverse problems: a {B}ayesian perspective}, Acta
  Numerica, 19 (2010), pp.~451--559.

\bibitem{szego1939orthogonal}
{\sc G.~Szego}, {\em Orthogonal polynomials}, vol.~23, American Mathematical
  Soc., 1939.

\bibitem{tabak2013family}
{\sc E.~G. Tabak and C.~V. Turner}, {\em A family of nonparametric density
  estimation algorithms}, Communications on Pure and Applied Mathematics, 66
  (2013), pp.~145--164.

\bibitem{tabak2010density}
{\sc E.~G. Tabak and E.~Vanden-Eijnden}, {\em {Density estimation by dual
  ascent of the log-likelihood}}, Communications in Mathematical Sciences, 8
  (2010), pp.~217 -- 233.

\bibitem{vershynin2018high}
{\sc R.~Vershynin}, {\em High-dimensional probability: An introduction with
  applications in data science}, vol.~47, Cambridge university press, 2018.

\bibitem{villani-OT}
{\sc C.~Villani}, {\em Optimal transport: Old and new}, vol.~338 of Grundlehren
  der mathematischen Wissenschaften, Springer, New York, 2009.

\bibitem{wainwright2008graphical}
{\sc M.~J. Wainwright, M.~I. Jordan, et~al.}, {\em Graphical models,
  exponential families, and variational inference}, Foundations and
  Trends{\textregistered} in Machine Learning, 1 (2008), pp.~1--305.

\bibitem{wang2022minimax}
{\sc S.~Wang and Y.~Marzouk}, {\em On minimax density estimation via measure
  transport}, arXiv preprint arXiv:2207.10231,  (2022).

\bibitem{wehenkel2019unconstrained}
{\sc A.~Wehenkel and G.~Louppe}, {\em Unconstrained monotonic neural networks},
  Advances in neural information processing systems, 32 (2019).

\bibitem{westermann2023measure}
{\sc J.~Westermann and J.~Zech}, {\em Measure transport via polynomial density
  surrogates}, arXiv preprint arXiv:2311.04172,  (2023).

\bibitem{xiu2010numerical}
{\sc D.~Xiu}, {\em Numerical methods for stochastic computations: a spectral
  method approach}, Princeton university press, 2010.

\bibitem{zech2022sparseI}
{\sc J.~Zech and Y.~Marzouk}, {\em Sparse approximation of triangular
  transports, part i: The finite-dimensional case}, Constructive Approximation,
   (2022), pp.~1--68.

\bibitem{zech2022sparseII}
{\sc J.~Zech and Y.~Marzouk}, {\em Sparse approximation of triangular
  transports, part ii: The infinite-dimensional case}, Constructive
  Approximation, 55 (2022), pp.~987--1036.

\bibitem{zhang2018advances}
{\sc C.~Zhang, J.~B{\"u}tepage, H.~Kjellstr{\"o}m, and S.~Mandt}, {\em Advances
  in variational inference}, IEEE transactions on pattern analysis and machine
  intelligence, 41 (2018), pp.~2008--2026.

\end{thebibliography}

\end{document}